\documentclass[12pt]{amsart}   	
\usepackage[margin=1in]{geometry}                		
\usepackage{ifxetex,ifluatex}
\newif\ifxetexorluatex
\ifxetex
  \xetexorluatextrue
\else
  \ifluatex
    \xetexorluatextrue
  \else
    \xetexorluatexfalse
  \fi
\fi

\ifxetexorluatex
  \usepackage{fontspec}
\else
  \usepackage[T1]{fontenc}
  \usepackage[utf8]{inputenc}
  \usepackage{lmodern}
\fi

\usepackage{stmaryrd}
\usepackage{mathdots}
\usepackage{bbm}
\usepackage{amsmath,amssymb,amsfonts}
\usepackage{amsthm}
\usepackage{enumerate}
\usepackage{yfonts}
\usepackage{multicol}
\usepackage{ragged2e}
\usepackage{faktor}
\usepackage{pgfplots}
\usepackage{esint}

\usepackage{pinlabel}
\usepackage{graphicx}	
\usepackage{subcaption}
\usepackage{tikz}
\usepackage{tikz-cd}

\usepackage[shortlabels]{enumitem}
\usepackage{longtable}

\newtheorem{theorem}{Theorem}
\newtheorem*{theorem*}{Theorem}

\newtheorem{proposition}{Proposition}
\newtheorem*{proposition*}{Proposition}

\newtheorem{corollary}[proposition]{Corollary}
\newtheorem*{corollary*}{Corollary}

\newtheorem{lemma}[proposition]{Lemma}
\newtheorem*{lemma*}{Lemma}

\newtheorem{remark}[proposition]{Remark}
\newtheorem*{rmk*}{Remark}

\newtheorem*{claim*}{Claim}

\theoremstyle{definition}
\newtheorem{definition}[proposition]{Definition}
\newtheorem*{definition*}{Definition}

\newtheorem{example}[proposition]{Example}

\numberwithin{equation}{section}
\numberwithin{theorem}{section}
\numberwithin{proposition}{section}
\numberwithin{figure}{section}

\usepackage{setspace}
\usepackage{hyperref}
\hypersetup{colorlinks,allcolors=blue}
\usepackage[capitalize]{cleveref}

\usepackage{verbatim}

\def\settozero{\setcounter{enumii}{-1}\renewcommand\theenumii{\arabic{enumii}}}
\usepackage{faktor}
\usepackage[T1]{fontenc}

\newcommand{\E}{\mathbb E}

\newcommand{\R}{\mathbb{R}}
\newcommand{\Z}{\mathbb{Z}}
\newcommand{\M}{\mathbb{M}}

\newcommand{\diff}{\mathrm{d}}

\newcommand{\C}{\mathcal{C}}

\newcommand{\J}{\mathcal{J}}
\newcommand{\K}{\mathcal{K}}
\newcommand{\Q}{\mathcal{Q}}
\newcommand{\T}{\mathcal{T}}
\newcommand{\Ell}{\mathscr L}

\usepackage{xparse}
\NewDocumentCommand{\Hom}{mmg}{\ensuremath{\text{Hom}_{\IfNoValueTF{#3}{}{#3}}(#1,#2)}}
\NewDocumentCommand{\Tor}{mmmg}{\ensuremath{\text{Tor}_{#3}^{\IfNoValueTF{#4}{}{#4}}(#1,#2)}}
\NewDocumentCommand{\Ext}{mmmg}{\ensuremath{\text{Ext}^{#3}_{\IfNoValueTF{#4}{}{#4}}(#1,#2)}}

\usepackage[utf8]{inputenc}
\usepackage[T1]{fontenc}
\usepackage{mathrsfs}

\usepackage[backend = biber]{biblatex}
\hypersetup{breaklinks=true}
\usepackage{doi}
\usepackage{xurl}
\title{One-dimensional Wave Kinetic Theory}
\author[Katja D. Vassilev]{Katja D. Vassilev}
\address{Department of Mathematics, University of Michigan, Ann Arbor, MI 48109}
\email{kdv@umich.edu}
\pgfplotsset{compat=1.18}				
\begin{document}
\begin{abstract}
Although wave kinetic equations have been rigorously derived in dimension $d \geq 2$, both the physical and mathematical theory of wave turbulence in dimension $d = 1$ is less understood. Here, we look at the one-dimensional MMT (Majda, McLaughlin, and Tabak) model on a large interval of length $L$ with nonlinearity of size $\alpha$, restricting to the case where there are no derivatives in the nonlinearity. The dispersion relation here is $|k|^\sigma$ for $0 < \sigma \leq 2$ and $\sigma \neq 1$, and when $\sigma = 2$, the MMT model specializes to the cubic nonlinear Schr\"odinger (NLS) equation. In the range of $1 < \sigma \leq 2$, the proposed collision kernel in the kinetic equation is trivial, begging the question of what is the appropriate kinetic theory in that setting. In this paper we study the kinetic limit $L \to \infty$ and $\alpha \to 0$ under various scaling laws $\alpha \sim L^{-\gamma}$ and exhibit the wave kinetic equation up to timescales $T \sim L^{-\epsilon}\alpha^{-\frac{5}{4}}$ (or $T \sim L^{-\epsilon} T_{\mathrm{kin}}^{\frac{5}{8}}$). In the case of a trivial collision kernel, our result implies there can be no nontrivial dynamics of the second moment up to timescales $T_{\mathrm{kin}}$. 
\end{abstract}
\maketitle
\tableofcontents

\section{Introduction}
    Wave kinetic theory is the formal study of non-equilibrium statistical mechanics associated to nonlinear wave systems paralleling Boltzmann's kinetic theory for colliding particles. The theory attempts to extract the macroscopic dynamics arising from microscopic wave interactions. This is done with a \textit{wave kinetic equation} (WKE) describing the expected evolution of the energy spectrum as the size of the domain $L \to \infty$, while the strength of the nonlinearity $\alpha \to 0$. Deriving such an equation rigorously requires a \textit{scaling law}, describing how these two limits are taken.

Rigorous justification of these kinetic equations falls under the broad umbrella of  \textit{Hilbert's Sixth Problem}, which has recently begun to exhibit much progress in the wave setting in dimensions $d \geq 2$ up to (or past) the \textit{kinetic (or Van Hove) timescale}. This characteristic time reflects the scale separation between the nonlinear interactions and wave kinetic equation. Here,
\begin{equation}\label{eq-Tkin}
T_{\mathrm{kin}} \sim \frac{1}{\alpha^2}. 
\end{equation}
However, little is understood physically, and practically nothing has been justified rigorously in dimension $d = 1$. Here, we focus on 4-wave interacting systems in dimension 1 and study the MMT (Majda, McLaughlin, and Tabak) model with no derivatives in the nonlinearity, which has the cubic nonlinear Schr\"odinger (NLS) equation as a special case. More specifically, we consider 
\begin{equation}\label{DISP} 
\begin{cases}\left(i\partial_t - (2\pi)^{1-\sigma}|\nabla|^\sigma \right) u + \alpha |u|^2u = 0, \hspace{.8 cm} x \in \mathbb T_L = [0,L], \\ u(0,x) = u_{\mathrm{in}} (x),\end{cases}
\end{equation}
for $0 < \sigma \leq 2$ and $\sigma \neq 1$, where $\sigma = 2$ corresponds to the cubic NLS. Here, $\alpha$ represents the strength of the nonlinearity and the dispersion relation, up to factors of $2\pi$, is
\begin{equation} \label{eq-disp}
\omega(k) := |k|^\sigma.
\end{equation}

For the NLS ($\sigma = 2$), the relevant kinetic equation, which has been rigorously proved to hold in $d \geq 2$, is trivial in dimension one. The proposed kinetic equation for the MMT model is nontrivial only for $0 < \sigma < 1$. When $1 < \sigma \leq 2$, the question becomes what, if any, is the appropriate kinetic theory. Here, we provide a rigorous justification, under various scaling laws, of the (potentially trivial) wave kinetic equation up to $O\left(L^{-\epsilon}\alpha^{-\frac{5}{4}}\right)$ timescales. When the kinetic equation is trivial, our result implies that there can be no equation describing the evolution of the second moment with nontrivial dynamics on timescales $\leq T_{\mathrm{kin}}$.

\subsection{Statement of the main result}
Kinetic theory seeks to provide effective dynamics of $\E|\widehat{u}(t,k)|^2$ where we define the Fourier transform as 
\begin{equation}\label{ftransform}
\widehat{u}(t,k) = \frac{1}{L^{1/2}}\int_{\mathbb T_L}u(t,x) e^{-2\pi i k x}dx, \hspace{1cm} u(t,x) = \frac{1}{L^{1/2}}\sum_{k \in \Z_L} \widehat{u}(k) e^{2\pi i k x},
\end{equation}
and the expectation is taken over a random distribution of the initial data. Random initial data which allows for the kinetic description is called \textit{well-prepared}. Here, we consider random homogeneous initial data given by 
\begin{equation}\label{DAT} \tag{DAT}
u_{\mathrm{in}}(x) = \frac{1}{L^{1/2}}\sum_{k \in \Z_L} \widehat{u_{\mathrm{in}}}(k) e^{2\pi i k x}, \hspace{1cm} \widehat{u_{\mathrm{in}}}(k) = \sqrt{n_{\mathrm{in}}(k)}g_k(\omega), 
\end{equation}
where $\Z_L := (L^{-1}\Z)$,  $n_{\mathrm{in}}: \R \to [0, \infty)$ is a given Schwartz function (or a function of sufficient smoothness and decay), and $\{g_k(\omega)\}$ is a collection of i.i.d. random variables ($\omega$ used here is not to be confused with the dispersion relation). We assume that each $g_k$ is either a centered normalized complex Gaussian, or uniformly distributed on the unit circle of $\mathbb C$. 

With the assumptions (\ref{DAT}) on the initial data, the relevant wave kinetic equation is given by
\begin{equation}\label{WKE} \tag{WKE}
\begin{cases}
\partial_t n(t,\xi) = \K(n(t))(\xi) \\
n(0,\xi) = n_{\mathrm{in}}(\xi).
\end{cases}
\end{equation}
Here, the collision operator is 
\begin{equation}\label{KIN}\tag{KIN}
\K(\phi)(\xi) = \int_{\begin{subarray}{b}(\xi_1, \xi_2, \xi_3 ) \in \R^{3} \\ \xi_1 - \xi_2 + \xi_3 = \xi\end{subarray}} \phi \phi_1 \phi_2 \phi_3 \left(\frac{1}{\phi} - \frac{1}{\phi_1} + \frac{1}{\phi_2} - \frac{1}{\phi_3}\right) \delta_{\R}(\omega_1 - \omega_2 + \omega_3 - \omega)\diff\xi_1 \diff\xi_2 \diff\xi_3,
\end{equation}
where we denote, for $i = 1,2,3$,
\begin{align*}
&\phi = \phi(t, \xi), \hspace{.5cm} \phi_i = \phi(t, \xi_i), \\
&\omega = |\xi|^\sigma, \hspace{.5cm} \omega_i = |\xi_i|^\sigma.
\end{align*}

Note that we only integrate over $(\xi_1, \xi_2, \xi_3) \in \R^3$ satisfying \[{\xi_1 - \xi_2 + \xi_3 -\xi = 0} \text{ and } {\omega_1 - \omega_2 + \omega_3 - \omega = 0}.\] This \textit{resonant manifold} is highly dependent on the dispersion relation and the dimension, as explained in \cite{Majda1997}. When $0 < \sigma < 1$, the resonant manifold admits nontrivial resonances. However, when $1 < \sigma \leq 2$, the dispersion relation is convex and there are only trivial resonances, meaning we must have ${\{\xi_1, \xi_3\} = \{\xi_2, \xi\}}$. In this case the terms in the main part of the integrand of \eqref{KIN} involving $\phi$ cancel, so $\K = 0$. This is only the case in dimension 1 as in higher dimensions, the extra dimensions can help balance the convexity of the dispersion relation, allowing for nontrivial resonances. 

Of particular importance in a rigorous derivation of any kinetic equation is the scaling law relating how we take $L \to \infty$ and $\alpha \to 0$, given by $\alpha = L^{-\gamma}$. A turbulent regime requires \textit{weak nonlinearity}, so that $0 < \gamma \leq 1$. See \cite{WKE2023} for a heuristic explanation of why these are the relevant scaling laws. We may now state our theorems precisely: 
\begin{theorem}{(NLS)} \label{mainthm}
Fix $\gamma \in (0, 1)$, $\epsilon \ll 1$, and Schwartz function $n_{\mathrm{in}} \geq 0$. Consider the equation (\ref{DISP}) with $\sigma = 2$ and random initial data (\ref{DAT}), and assume $\alpha = L^{-\gamma}$ so that $T_{\mathrm{kin}} \sim L^{2\gamma}$. Fix $T = L^{\frac{5}{4}\gamma - \epsilon}$. Then, for $L^{0+} \leq t \leq T$ (where $0+$ represents a number strictly bigger than 0 sufficiently close to it), 
\begin{equation}\label{mainthm-eq}
\mathbb E\left|\widehat{u}(t,k)\right|^2 - n_{\mathrm{in}}(k)  = o_{\ell_k^\infty}\left(\frac{t}{T_{\mathrm{kin}}}\right).
\end{equation}
\end{theorem}

\begin{theorem}{(MMT)} \label{mmtthm}
Fix $\gamma \in (0,1)$, $\epsilon \ll 1$, Schwartz function $n_{\mathrm{in}} \geq 0$, and $A > 0$.  Consider the equation (\ref{DISP}) with $0 < \sigma< 2$ and $\sigma \neq 1$ and random initial data (\ref{DAT}), and assume $\alpha = L^{-\gamma}$ so that $T_{\mathrm{kin}} \sim L^{2\gamma}$. Fix $T$ according to 
\begin{equation*}
T = \begin{cases} L^{-\epsilon}\min(L^{\frac{1}{2 - \sigma}}, L^{\frac{5}{4}\gamma})& \mathrm{if \ } 0 < \sigma < 1, \\ L^{-\epsilon}\min(L,L^{\frac{5}{4}\gamma} )& \mathrm{if \ } 1 < \sigma < 2.
\end{cases}
\end{equation*}
Then, with probability $\geq 1 - L^{-A}$, (\ref{DISP}) has a smooth solution up to time $T$ and for ${L^{0+} \leq t \leq T}$,
\begin{equation}\label{mmtthm-eq}
\mathbb E\left|\widehat{u}({t,k})\right|^2 = n_{\mathrm{in}}(k)  + \frac{t}{T_{\mathrm{kin}}} \mathcal K(n_{\mathrm{in}})(k) + o_{\ell_k^\infty}\left(\frac{t}{T_{\mathrm{kin}}}\right),
\end{equation}
where the expectation is taken only when (\ref{DISP}) has a smooth solution. 
\end{theorem}

A few comments on the above theorems are in order: 

\begin{itemize}

\item The term $o_{\ell_k^\infty}\left(\frac{t}{T_{\mathrm{kin}}}\right)$ in Theorems \ref{mainthm} and \ref{mmtthm} is a quantity that is bounded in $\ell_k^\infty$ by $L^{-\theta}\frac{t}{T_{\mathrm{kin}}}$ for some $\theta > 0$. Additionally, we take $L^{0+} \leq t$ for technical reasons due to the way we state \eqref{mainthm-eq} and \eqref{mmtthm-eq}. For $0 \leq t \leq L^{0+}$, the proof provides a uniform bound of $L^{0+}T_{\mathrm{kin}}^{-1}$ on the quantity $\mathbb E\left|\widehat{u}(t,k)\right|^2 - n_{\mathrm{in}}(k)$. 
\item As mentioned above, the collision kernel (\ref{KIN}) is trivial when $1 < \sigma \leq 2$. Therefore, the theorems imply that there is no nontrivial behavior of $\E |\widehat{u}(t,k)|^2$ up to $T_{\mathrm{kin}}$ since that would contribute $O\left( \frac{t}{T_{\mathrm{kin}}} \right)$ to $\E |\widehat{u}(t,k)|^2$ on the intervals of the theorem. This collapse of the kinetic theory does not mean that there is no wave kinetic theory in this setting, but rather that any such theory would differ greatly from the existing one, and happen at later timescales. In particular, such a description cannot take the form of \eqref{WKE}.

\item While the NLS is known to be globally well-posed in $L^2$, we don't have a similar result for MMT, even locally. For this reason, the expectation in Theorem \ref{mmtthm} is taken over a set with overwhelming probability rather than in Theorem \ref{mainthm} where the expectation is taken without condition. 

\item The assumption that $n_{\mathrm{in}}$ is Schwartz is not necessary here in the sense that we only need control of finitely many derivatives of $n_{\mathrm{in}}$, say 40, and finite polynomial decay, depending on $O(\epsilon^{-1})$. See Remark \ref{rmk-how-large} for further clarification. 

\item Although Theorem \ref{mainthm} allows $T > L$, Theorem \ref{mmtthm} does not. This is primarily due to the limitations of the counting estimates when $\sigma \neq 2$, presented in Section \ref{subsec-counting}. For this reason, we are only able to reach times close to $\alpha^{-\frac{5}{4}} \sim T_{\mathrm{kin}}^{\frac{5}{8}}$ for certain scaling laws. In particular, if $0 < \sigma < 1$, we reach $L^{-\epsilon}\alpha^{-\frac{5}{4}}$ for scaling laws $0 < \gamma < \frac{4}{5}\frac{1}{2 - \sigma}$ and if $1 < \sigma < 2$, the corresponding scaling laws are $0 < \gamma < \frac{4}{5}$.
\end{itemize}
\subsection{Background literature} The physical study of \textit{wave kinetic theory} began with Peierls' work on anharmonic crystals \cite{Peierls} in 1929, yielding the phonon Boltzmann equation, the first wave kinetic equation, studied formally in \cite{PhononBoltzmann, Spohn08}. Shortly after this, Nordheim (and later Uehling and Uhlenbeck) formulated the quantum Boltzmann equation to describe quantum-interacting gases \cite{Nordheim, Uehling1933}. This framework was also adapted in plasma physics  \cite{Plasma1965, Plasma1967, Plasmadavidson1972} and water waves \cite{WaterWaves62, WaterWaves63, WaterWaves66, WaterWaves69} by the end of the 1960s. Currently, wave kinetic theory is central to many disciplines outside of physics and mathematics, including oceanography \cite{JanssenOceanography}, meteorology \cite{Meteorology}, and optics \cite{Optical1D, Optical1D17}. Zakharov discovered stationary solutions to wave kinetic equations that resemble Kolmogorov spectra in hydrodynamic turbulence \cite{Zakharov1965, TextbookCascade}, drawing a strong connection between wave kinetic theory and turbulence and lending an alternate name of \textit{wave turbulence theory}. These stationary power solutions exhibit a \textit{forward and backward cascade} of energy, which has been extensively studied, see \cite{TextbookCascade, TextbookWT} for a textbook treatment. 
 
\subsubsection{Rigorous derivations} Rigorous justification of kinetic equations has only recently begun to take shape. Although physicists have long employed Feynman diagram expansions to derive kinetic equations, proving mathematical convergence of these expansions first occurred in 1977 in the linear setting, when Spohn derived a transport equation for electrons moving through random impurities for short times \cite{SpohnFeynman}. This was extended in the early 2000s by Erd\"os-Yau \cite{ErdosYau} and Erd\"os-Salmhofer-Yau \cite{ErdosSalmhoferYau} to a global in time result. In the nonlinear setting, the mathematical study of wave turbulence first focused on demonstrating the energy cascade predicted by the kinetic theory via the growth of Sobolev norms, see for instance \cite{CKSTT, SobolevGrowth, SobolevGrowth2D}. In 2011, Lukkarinen-Spohn used a Feynman diagram expansion to study a similar problem to the derivation of the kinetic equation for the discrete NLS on a lattice \cite{NLSLattice}.
 
 This prompted many works to explore the problem of rigorously deriving kinetic equations \cite{2DNLS, BGHSCR, BGHSdyn, BGHSonset, CG1, CG2}, culminating in a rigorous derivation by Deng-Hani of the WKE for the cubic NLS in dimension $d \geq 3$, first for short multiples of kinetic time \cite{2019, WKE, WKE2023}, and then for arbitrarily long times \cite{WKElong}. There has also been progress on related equations including the Wick NLS \cite{WickNLS} and the NLS with additive stochastic forcing \cite{StochasticNLS}. While these equations all involve 4-wave interactions, work has also been done in the 3-wave setting with the Zakharov-Kuznetsov (ZK) equation with stochastic forcing in the homogeneous \cite{ZK} and inhomogeneous \cite{ZKinhomo} case, and with diffusion \cite{MaZK}, as well as the NLS with quadratic nonlinearity \cite{QuadraticInhomo}. 

\subsubsection{1D Results} The rigorous derivations of WKEs described above focus almost exclusively on dimensions $d \geq 2$ or higher, as major obstructions to bounding the Feynman expansion emerge in lower dimensions. On the other hand, numerical simulation of turbulence is less computationally intensive in lower dimensions, with the one-dimensional MMT model first proposed with the aim of performing empirical numerical studies on wave turbulence \cite{1DWT, Majda1997}. However, the MMT can have a trivial collision kernel, in which case understanding the proper kinetic picture is wide open. Most notably the cubic NLS, an integrable equation, has spawned many numerical results investigating what is now termed \textit{integrable turbulence} \cite{ZakharovIntTurb, IntegrableTurbulence, AZ15, AZ16}. Recent applied works on the NLS have focused on quasi-kinetic equations for the evolution of \textit{kurtosis}, a quantity related to the fourth moment, for short times \cite{4Moment, ExtremeEvents4, HeavyTail4, Janssen, Kurtosis}.

Although there has not been much progress on the rigorous derivation of kinetic theory in one dimension, related progress has begun. For instance, \cite{2DGravityWaves} proves energy estimates that are relevant to wave kinetic theory for two-dimensional gravity waves, which reduce to a one-dimensional system in the irrotational case. For the cubic NLS, \cite{dimensionone} investigates the longtime dynamics using normal forms for scaling laws $\gamma > 1$, which lie outside of the kinetic regime. The work here constitutes the first rigorous derivation of kinetic equations in one-dimension with weak nonlinearity.

Since the completion of this paper, several works have been announced making progress on the question of one dimensional turbulence, all addressing a similar subcritical timescale to the one used here. In particular, Deng-Ionescu-Pusateri constructed bounds on corresponding solutions \cite{1DWWWKE} using \cite{2DGravityWaves} and Wu used a similar framework to this work to derive the WKE for the dimension one $\beta$-FPUT model \cite{FPUWKE}. It would be desirable to extend these results to $T_{\mathrm{kin}}$, but this would require bounding Feynman diagram expansions that diverge in dimension one, but not higher dimension. This remains to be properly understood. 

\subsubsection{Well-posedness} While this paper deals solely with the derivation of kinetic equations, we should address the highly related question of well-posedness of such equations. Indeed, as the derivations rely on matching a Feynman diagram expansion to iterates of the WKE, some control on solutions is implicitly needed, although plays no direct role in the proof. There have been many recent works establishing local (and sometimes, global) well-posedness for 4-wave homogeneous \cite{EscobedoVelazquez, RadialWP, WP4, KineticMMT, FPUTWP, AmpatzoglouLegerIll}, 4-wave inhomogeneous  \cite{AmpatzoglouMildGlobal, AmpatzoglouLegerStrong, AMPTWP}, and 6-wave inhomogeneous \cite{6waveWP} WKEs. Additionally, \cite{AmpatzoglouLegerIll} establishes ill-posedness for a class of 4-wave homogeneous WKEs. Several works also investigated the related question of stability \cite{CascadeNLS, Stability4wave, FPUTWP} and instability \cite{Instability4wave} of steady states to WKEs.

For the 1D MMT, \cite{KineticMMT} established local well-posedness when the dispersion parameter $\sigma = \frac{1}{2}$, although this choice is mostly for algebraic simplification reasons. Their result also include a more general nonlinearity for MMT containing derivatives, namely ${\alpha |\nabla|^\beta\left(\left||\nabla|^\beta u\right|^2|\nabla|^\beta u \right)}$, where they establish local well-posedness only for values of $\beta \leq 0$. In dimension $d = 3$, \cite{AmpatzoglouLegerIll} considers the same system, establishing $\beta = \frac{1}{4}$ as the threshold for well-posedness. In the case of well-posedness for $\beta > 0$, it is plausible that one could derive the WKE, however there is an additional difficulty as iteration introduces losses for high frequencies. The only related result in this direction are \cite{2DGravityWaves, 1DWWWKE}, where the authors established longtime stability of the water wave equation. To deal with the losses caused by derivatives in the nonlinearity, they proved energy estimates which allowed them to control high frequencies.

\subsection{High-level proof overview}
As is the case in most derivations of kinetic equations, we rely on Feynman diagram expansions in the form of Duhamel iterates to write 
\begin{equation*}
u = u^{(0)} + u^{(1)} + u^{(2)} + \ldots u^{(N)} + R_N,
\end{equation*}
for sufficiently large finite $N$. Here, $u^{(j)}$ represents the $j$-th iterate and $R_N$ is the remainder. Expressing both of these properly requires a significant amount of setup provided in Section \ref{sec-prep}, so a detailed overview of the proof will be provided in Section \ref{sec-overview}. However, we lay out the general strategy to analyze the iterates and the remainder. As we do not reach $T_{\mathrm{kin}}$, the number of iterates $N$ that we look at will be finite $O(\epsilon^{-1})$ as in \cite{2019} rather than the $O(\log L)$ necessary to reach $T_{\mathrm{kin}}$ in \cite{WKE, WKE2023}. The overarching technique is quite similar to that in \cite{WKE, WKE2023}, although here we explicitly handle more general dispersion relations and utilize a novel method to bound the large component of the iterates.  The main components of the analysis are: 

\begin{itemize}
\item The \textbf{Feynman diagram} expansion here is given by terms of \textit{ternary trees}, so that each iterate $u^{(n)}$ is the sum of expressions related to all ternary trees of order $n$ (the order of a tree is the number of branching nodes). These ternary trees keep track of the nonlinear interactions at each iteration and are \textit{decorated} with wave numbers, whose number diverges as $L \to \infty$. As we need to consider the correlations $\E \left(u^{(k)}\overline{u^{(l)}}\right)$, we obtain \textit{couples}, two ternary trees along with a pairing of their leaves. See Figure \ref{fig-eg-couple} for two ternary trees which are paired to form a couple. Section \ref{sec-prep} is devoted to setting up the expansion precisely. 

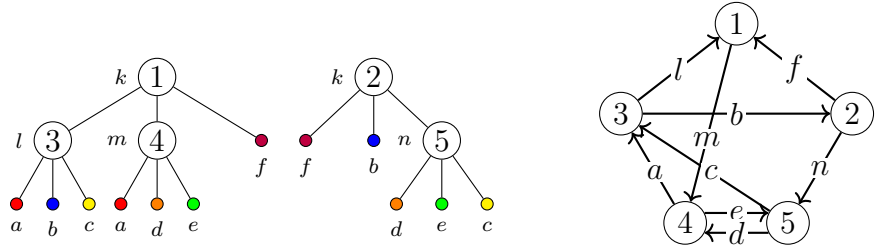
\begin{figure}
\begin{tikzpicture}[scale = .6, level distance=1.4cm,
  level 1/.style={sibling distance=2.3cm},
  level 2/.style={sibling distance=0.8cm}]]
\tikzstyle{hollow node}=[circle,draw,inner sep=1.6]
\tikzstyle{solid node}=[circle,draw,inner sep=1.6,fill=black]
\tikzset{
red node/.style = {circle,draw=black,fill=red,inner sep=1.6},
blue node/.style= {circle,draw = black, fill= blue,inner sep=1.6}, 
purple node/.style= {circle,draw = black, fill= purple,inner sep=1.6}, 
orange node/.style= {circle,draw = black, fill= orange,inner sep=1.6},
yellow node/.style= {circle,draw = black, fill= yellow,inner sep=1.6},
green node/.style = {circle,draw=black,fill=green,inner sep=1.6}}
\node[hollow node, label = left: {\tiny $k$}] at (14,.4){1}
    child{node[hollow node, label = left: {\tiny $l$}]{3}
        child{node[red node, label = below: {\tiny $a$}]{}}
        child{node[blue node, label = below: {\tiny $b$}]{}}
        child{node[yellow node, label = below:{\tiny $c$}]{}}
    }
    child{node[hollow node, label = left:{\tiny $m$}]{4}
        child{node[red node, label = below: {\tiny $a$}]{}}
        child{node[orange node, label = below: {\tiny $d$}]{}}
        child{node[green node, label = below:{\tiny $e$}]{}}
    }
    child{node[purple node, label = below: {\tiny $f$}]{}}
;
\end{tikzpicture}
~
\begin{tikzpicture}[scale = .6, level distance=1.4cm,
  level 1/.style={sibling distance=1.5cm},
  level 2/.style={sibling distance=1.0cm}]]
\tikzstyle{hollow node}=[circle,draw,inner sep=1.6]
\tikzstyle{solid node}=[circle,draw,inner sep=1.6,fill=black]
\tikzset{
red node/.style = {circle,draw=black,fill=red,inner sep=1.6},
blue node/.style= {circle,draw = black, fill= blue,inner sep=1.6}, 
purple node/.style= {circle,draw = black, fill= purple,inner sep=1.6}, 
orange node/.style= {circle,draw = black, fill= orange,inner sep=1.6},
yellow node/.style= {circle,draw = black, fill= yellow,inner sep=1.6},
green node/.style = {circle,draw=black,fill=green,inner sep=1.6}}
\node[hollow node, label = left: {\tiny $k$}] at (14,.4){2}
    child{node[purple node, label = below: {\tiny $f$}]{}}
    child{node[blue node, label = below:{\tiny $b$}]{}}
    child{node[hollow node, label = left: {\tiny $n$}]{5}
        child{node[orange node, label = below: {\tiny $d$}]{}}
        child{node[green node, label = below: {\tiny $e$}]{}}
        child{node[yellow node, label = below:{\tiny $c$}]{}}
    }
;
\end{tikzpicture}
~\hspace{1cm}
\begin{tikzpicture}[scale = .85]
    \tikzstyle{open node} = [circle, draw = black, inner sep = 2.3]
    \tikzstyle{empty node} = [inner sep=1,outer sep=0]
    \node[open node] (3) at (2.2, .1) {3};
    \node[open node] (2) at (5.8, .1) {2};
    \node[open node] (1) at (4,1.5) {1};
    \node[open node] (4) at (3.2, -1.6) {4};
    \node[open node] (5) at (4.8, -1.6) {5};

    \draw[->, thick] (3) -- node[empty node, midway, fill = white, draw = none]{$b$} (2);
    \draw[->, thick] (2) -- node[empty node, midway, fill = white, draw = none]{$f$} (1);
    \draw[->, thick] (3) -- node[empty node, midway, fill = white, draw = none]{$l$} (1);
    \draw[->, thick] (1) -- node[empty node, pos = .6, fill = white, draw = none]{$m$} (4);
    \draw[->, thick] (4.27) --node[empty node, midway, fill = white,draw = none]{$e$} (5.153);
    \draw[->, thick] (5.207) -- node[empty node, midway, fill = white, draw = none]{$d$} (4.333);
    \draw[->, thick] (5) -- node[empty node, pos = .45, fill = white, draw = none]{$c$} (3);
    \draw[->, thick] (4) -- node[empty node, pos = .45, fill = white, draw = none]{$a$} (3);
    \draw[->, thick] (2) -- node[empty node, midway, fill = white, draw = none]{$n$} (5);
\end{tikzpicture}
\caption{An example of two ternary trees which are paired (using colors and letters) to form a couple, as well as the corresponding molecule. }
\label{fig-eg-couple}
\end{figure}

\item The main step of the proof is to reduce it to a \textbf{counting problem} on couples (see Section \ref{sec-red}), where we bound the number of possible decorations with wave numbers. To capture the key parts of the counting problem, we turn the couple into a \textit{molecule}, a type of directed graph seen in Figure \ref{fig-eg-couple}, which is roughly performed by attaching the two trees of the couple by their paired leaves. The goal is to count small groups of wave numbers at a time using counting estimates. Unfortunately, counting estimates are much worse in ${d = 1}$. In \cite{WKE, WKE2023}, the higher dimensions allow the authors to combine several smaller counting estimates which are undesirable into a single larger and stronger estimate. We show that this is not possible in $d = 1$ and analytically prove the basic counting estimates encompassing all dispersion relations we are considering (see Section \ref{sec-counting}). These estimates are roughly split between those that are good, in the sense that they could allow us to reach $T_{\mathrm{kin}}$, or bad otherwise. 

\item We also rely on a \textbf{cancellation argument} between couples to minimize the number of bad counting estimates that we must do. Similar to \cite{WKE}, we identify \textit{irregular chains} as a structure within couples that can lead to a counting problem with an arbitrarily large number of bad counting estimates. However, we exhibit cancellations between irregular chains in similar couples, allowing us to remove the irregular chains from these couples before we reduce them to the counting problem. See Section \ref{sec-splice}. 

\item Once we reduce the result to a counting problem on a couple, or its corresponding molecule, we lay out an \textbf{algorithm} in Section \ref{sec-algorithm} using our counting estimates. The algorithm is suited to subcritical problems ($T \ll T_{\mathrm{kin}}$) in the sense that it minimizes (but cannot counteract) the number of bad counting estimates. Note that this is the case even considering our cancellation argument, which simply removes the very worst structures.

\item To bound the number of bad counting estimates, we perform an \textbf{operation map} in Section \ref{sec-opbound} by developing \textit{operation trees}, binary trees expressing the relationships between steps of the algorithm. These trees allow us to partially order each step of the algorithm and map the steps using bad counting estimates to steps using good ones in a well-defined manner. 

\item The \textbf{remainder} term $R_N$ is dealt with in Section \ref{sec-opL} and \ref{sec-thm}. We reduce bounding the remainder to inverting a linear operator $1 - \Ell$, allowing us to bound $R_N$ using contraction mapping. Powers of $\Ell^n$ can be written in terms of expressions related to a generalized version of ternary trees, which is what we end up bounding. The reason that we must consider the remainder separately is due to the divergence of the number of couples we would need to consider in higher iterates.
\end{itemize}

\section{Preparations}
\label{sec-prep}
\subsection{Preliminary reductions} For a solution $u$ to (\ref{DISP}), let $M = \fint |u|^2$ be the conserved mass (where $\fint$ takes the average on $\mathbb T_L$), and define $v:= e^{-2i\alpha Mt}\cdot u$. Then $v$ satisfies the Wick ordered equation 
\begin{equation}
\left(i\partial_t - \frac{1}{2\pi}|\nabla|^\sigma \right) v + \alpha\left(|v|^2 v - 2\fint |v|^2 \cdot v\right) = 0.
\end{equation}
Switching to Fourier space, rescaling in time, and conjugating by the linear flow, we define
\begin{equation}
a_k(t) = e^{-2\pi i \omega(k) Tt}\widehat{v}(Tt,k). 
\end{equation}
Then setting $a(t) = (a_k(t))$, $a_k$ satisfies 
\begin{equation}\label{eq-a}
\begin{cases}
\partial_t a_k = \mathcal{C}_+({ a}, { a}, {a})_k(t), \\
a_k(0) = \sqrt{n_{\mathrm{in}}(k)}g_k(\omega).
\end{cases}
\end{equation}
The nonlinearity $\C_{\zeta}$, for $\zeta \in \{\pm\}$ is
\begin{equation} \label{eq-C}
\C_{\zeta}(a, b, c)(t) := \left(\frac{\alpha T}{L} \right) (i \zeta) \sum\limits_{k_1 - k_2 + k_3 = k} \epsilon_{k_1k_2k_3}e^{\zeta 2\pi i Tt\Omega(k_1, k_2, k_3, k)} a_{k_1}(t) \overline{b_{k_2}(t)} c_{k_3}(t), 
\end{equation}
where 
\begin{equation} \label{eq-eps}
\epsilon_{k_1k_2k_3} := \begin{cases}
+1, & \text{if } k_2 \notin \{k_1,k_3\}; \\
-1 & \text{if } k_1 = k_2 = k_3; \\
0 & \text{otherwise},
\end{cases}
\end{equation}
and the resonance factor 
\begin{equation}\label{eq-omega}
\Omega = \Omega(k_1, k_2, k_3, k) := \omega(k_1) - \omega(k_2) + \omega(k_3) - \omega(k).
\end{equation}

The rest of the paper is focused on the system (\ref{eq-a}) for $a$, with the relevant terms defined in (\ref{eq-C})-(\ref{eq-omega}), in the time interval $t \in [0,1]$. 
\subsection{Parameters, notations, and norms} \label{sec-norms}
Throughout, let $C$ denote a large constant depending only on $(n_{\mathrm{in}}, \epsilon, \sigma)$, which may vary from line to line. Let $\theta $ denote any sufficiently small constant, depending on $(\epsilon, \sigma)$, and $\delta$ a fixed small constant depending only on $(\epsilon, \sigma)$. Define $N:= \frac{10^6}{\epsilon}$, the level up to which we will expand. 

For $t \in [0,1]$ and any function $F = F(t)$ defined on $[0,1]$, denote the Duhamel operator by
\begin{equation}\label{eq-duhamel}
\mathcal I F(t) = \int_0^t F(s) \diff s.
\end{equation}
Define the time Fourier transform (the use of $\widehat{\cdot}$ depends on the context) by
\begin{equation*}
\widehat{u}(\tau) = \int_\R u(t) e^{-2\pi i \tau t} \diff t, \hspace{1cm} u(t) = \int_\R \widehat{u}(\tau) e^{2\pi i \tau t} \diff \tau. 
\end{equation*}
For a function $a(t) = (a_k(t))_{k \in \Z_L}$, define the $Z$ norm to be
\begin{equation}
||a||_Z^2 = \sup_{0 \leq t \leq 1} L^{-1} \sum_{k \in \Z_L}\langle k \rangle^{10} |a_k(t)|^2.
\end{equation}
\subsection{Ansatz}
Note that by (\ref{eq-a}) and (\ref{eq-duhamel}), $a_k$ satisfies 
\begin{equation*}
a_k(t) = a_k(0) + \mathcal I \mathcal C_+({a}, {a}, {a})(t).
\end{equation*}
Define $(\mathcal J_n)_k(t)$ recursively as follows: 
\begin{align}
\left( \J_0\right)_k(t) &= a_k(0), \\
\left( \J_n\right)_k(t) &= \sum_{n_1 + n_2+ n_3 = n-1} \mathcal I \mathcal C_+(\mathcal J_{n_1}, \J_{n_2}, \J_{n_3})(t). 
\end{align}
We take the following ansatz for $a_k(t)$:
\begin{align*}
a_k(t) &= \sum_{n = 0}^N (\mathcal J_n)_k(t) + b_k(t),
\end{align*}
where $b$ is a remainder term and $N$ is defined in Section \ref{sec-norms}. Then, $b$ satisfies the equation
\begin{equation}\label{eq-b}
b = \mathscr R + \Ell b + \mathscr Q(b,b) + \mathscr C(b,b,b).
\end{equation}
The relevant terms are defined as
\begin{align}
\mathscr R &= \sum_{\substack{n_1 + n_2 + n_3 \geq N \\ 0 \leq n_1, n_2, n_3 < N}} \mathcal I \mathcal C_{+}(\mathcal J_{n_1}, \mathcal J_{n_2}, \mathcal J_{n_3}) \label{eq-r}, \\ 
\Ell b &= \sum_{0 \leq n_1, n_2 < N}^{\mathrm{cyc}} \mathcal I \mathcal C_{+}(\mathcal J_{n_1}, \mathcal J_{n_2}, b), \label{eq-L}\\
\mathscr Q(b,b) &= \sum_{0 \leq n < N}^{\mathrm{cyc}} \mathcal I \mathcal C_{+}(\J_{n}, b, b), \\
\mathscr C(b,b,b) &= \mathcal I \mathcal C_{+}(b, b, b),
\end{align}
where $\Sigma^\mathrm{cyc}$ indicates cyclic permutations of the entries of $\mathcal C_+$. Note that $\Ell, \mathscr Q, $ and $\mathscr C$ are multilinear operators. Therefore, (\ref{eq-b}) is equivalent to
\begin{equation} \label{eq-binv}
b = (1 - \Ell)^{-1}(\mathscr R + \mathscr Q(b,b) + \mathscr C(b,b,b)),
\end{equation}
provided that $1 - \mathscr L$ is invertible in a suitable space.
\subsection{Trees and couples}

As in  \cite{2019, WKE, WKE2023}, we break the terms $\J_n$ in our ansatz of $a_k$ down even further into ternary trees $\T$:
\begin{definition}{(Trees)}
A \textit{ternary tree} $\T$ is a rooted tree where each branching node has precisely three children. We say the tree is \textit{trivial} if it has no branching nodes, and denote it by $\boldsymbol \cdot$. Denote the set of branching nodes by $\mathcal N$, the set of leaves by $\mathcal L$, and the root by $\mathfrak r$. The \textit{order} of a tree, is denoted $n(\T) = |\mathcal N|$.  If $n(\T) = n$, $|\mathcal L| = 2n+1$. A tree may have a sign $\zeta \in \{\pm\}$, which sometimes may be added as a superscript. Then, each node (branching or leaf) also has a sign such that for a branching node $\mathfrak n$ with sign $\zeta_{\mathfrak n}$, its three children, from left to right, have sign $(+\zeta_{\mathfrak n}, - \zeta_{\mathfrak n}, +\zeta_{\mathfrak n})$. Define $\zeta(\T) = \prod_{\mathfrak n \in \mathcal N} (i\zeta_\mathfrak n)$. 
\end{definition}

\begin{definition}{(Decorations)} A decoration $\mathscr D$ of a tree $\T$ is a set of vectors $(k_{\mathfrak n})_{\mathfrak n \in \T}$ such that $k_{\mathfrak n} \in \Z_L$. Further, if $\mathfrak n$ is a branching node, 
\begin{equation*}
\zeta_{\mathfrak n} k_{\mathfrak n} = \zeta_{\mathfrak n_1} k_{\mathfrak n_1} + \zeta_{\mathfrak n_2} k_{\mathfrak n_2} +\zeta_{\mathfrak n_3} k_{\mathfrak n_3},
\end{equation*}
where $\mathfrak n_1, \mathfrak n_2, \mathfrak n_3$ are the three children of $\mathfrak n$ labelled from left to right. We say $\mathscr D$ is a $k$-decoration if $k_{\mathfrak r} = k$. Given a decoration $\mathscr D$, for each $\mathfrak n \in \mathcal N$ we define the resonance factors $\Omega_{\mathfrak n}$ by 
\begin{equation}
\Omega_{\mathfrak n} := \omega(k_{\mathfrak n_1}) - \omega(k_{\mathfrak n_2}) + \omega(k_{\mathfrak n_3}) - \omega(k_{\mathfrak n}).
\end{equation}
We also define 
\begin{equation}
\epsilon_{\mathscr D} := \prod_{\mathfrak n \in \mathcal N} \epsilon_{k_{\mathfrak n_1}k_{\mathfrak n_2}k_{\mathfrak n_3}}. 
\end{equation}
\end{definition}

\begin{definition}
For a tree $\T$ of order $n$, define 
\begin{equation} \label{eq-J}
(\J_\T)_k(t) := \left(\frac{\alpha T}{L}\right)^n \zeta(\T) \sum_{\mathscr D} \epsilon_{\mathscr D} \int_\mathcal{D} \prod_{\mathfrak n \in \mathcal N}e^{\zeta_{\mathfrak n} 2 \pi i T{t_\mathfrak n}\Omega_{\mathfrak n}} \diff t_{\mathfrak n} \prod_{\mathfrak l \in \mathcal L} \sqrt{n_{\mathrm{in}}(k_{\mathfrak l})}g_{k_{\mathfrak l}}^{\zeta_\mathfrak l}(\omega),
\end{equation}
where the sum is taken over all $k$-decorations $\mathscr D$ of $\T$, and the domain 
\begin{equation} \label{eq-D}
\mathcal D = \{t[\mathcal N] : 0 < t_{\mathfrak n'} < t_{\mathfrak n} < t \text{ whenever } \mathfrak n' \text{ is a child node of } \mathfrak n\}.
\end{equation}
\end{definition}

With this definition, for each $n$
\begin{equation}
(\J_n)_k(t) = \sum_{n(\mathcal{T^+}) = n}(\J_{\T^+})_k(t).
\end{equation}

Since we aim to estimate $\E|\widehat{u}(Tt,k)|^2 = \E |a_k(t)|^2$, it will therefore be useful to look at two trees $\T^+$ and $\T^-$ and estimate $\E \left[(\mathcal{J_{T^+}})_k(t)\overline{(\mathcal{J_{T^-}})_k(s)}\right] (0 \leq s,t \leq 1)$. For this, we recall the notion of couples introduced in \cite{WKE, WKE2023}: 

\begin{definition}{(Couples)} \label{def-couples}
A couple $\Q$ is a pair $\{\T^+, \T^-\}$ of two trees with opposite signs, together with a partition $\mathscr P$ of the set $\mathcal L^+ \cup \mathcal L^-$ into $(n+1)$ pairwise disjoint two-element subsets where $n = n(\T^+) + n(\T^-)$ is the \textit{order} of the couple. For the partition $\mathscr P$, we further impose that for $\{\mathfrak l, \mathfrak l'\} \in \mathscr P$, $\zeta_{\mathfrak l} = -\zeta_{\mathfrak l'}$. For a  couple $\Q = \{\T^+, \T^-, \mathscr P\}$, we denote the set of branching nodes by $\mathcal N = \mathcal N^+ \cup \mathcal N^-$ and the leaves by $\mathcal L = \mathcal L^+ \cup \mathcal L^-$. Define $\zeta(\Q) = \prod_{\mathfrak n \in \mathcal N}(i \zeta_{\mathfrak n})$. The trivial couple is the (only) order 0 couple, consisting of two trivial trees whose roots are paired. See Figure \ref{fig-eg-couple} for an example of a couple. 

A \textit{decoration} $\mathscr E$ of a couple is a decoration of each tree by $\mathscr D^+, \mathscr D^-$, along with the further restriction that $k_{\mathfrak l} = k_{\mathfrak l'}$ if $\{\mathfrak l, \mathfrak l'\} \in \mathscr P$. We define $\epsilon_{\mathscr E} = \epsilon_{\mathscr D^+}\epsilon_{\mathscr D^-}$. We say $\mathscr E$ is a $k$-decoration if $k_{\mathfrak r^+} = k_{\mathfrak r^-} = k$.
\end{definition}

\begin{definition}
For a couple $\Q$, define 
\begin{equation}\label{eq-KQ}
(\K_\Q)(t,s,k)  := \left( \frac{\alpha T}{L}\right)^n \zeta(\Q) \sum_{\mathscr E} \epsilon_{\mathscr E}\int_\mathcal{E} \prod_{\mathfrak n \in \mathcal N}e^{\zeta_{\mathfrak n} 2 \pi i Tt_{\mathfrak n}\Omega_{\mathfrak n}} \diff t_{\mathfrak n} \prod\limits^{+}_{\mathfrak l \in \mathcal L} n_{\mathrm{in}}(k_{\mathfrak l}),
\end{equation}
where the sum is taken over all $k$-decorations $\mathscr E$ of $\Q$, and the product $\prod^+_{\mathfrak l \in \mathcal L}$ is taken over leaves with $+$ signs, and the domain 
\begin{align} \label{eq-E}
\mathcal E &= \{t[\mathcal N] : 0 < t_{\mathfrak n'} < t_{\mathfrak n} \text{ whenever } \mathfrak n' \text{ is a child node of } \mathfrak n; \\
&\hspace{3.5cm}t_{\mathfrak n} < t \text{ for } \mathfrak n \in \mathcal N^+ \text{ and } t_{\mathfrak n} < s \text{ for } \mathfrak n \in \mathcal N^-\}. \nonumber
\end{align}
\end{definition}

Note that by Isserlis' Theorem, see Lemma A.2 of \cite{WKE},
\begin{equation}
\E \left[(\J_{\T^+})_k(t)\overline{(\J_{\T^-})_k(s)}\right] = \sum_{\mathcal Q} (\K_{\Q})(t,s,k), 
\end{equation}
where the summation is taken over all couples $\Q = \{\T^+, T^-\}$ with any partition $\mathscr P$. 
\subsection{The main estimates} The following are the main estimates of this paper. Their proofs will occupy up through Section \ref{sec-opL}, with the proof of Theorems \ref{mainthm} and \ref{mmtthm} in Section \ref{sec-thm}. Recall that $N$ large depending on $\epsilon$ is fixed in Section \ref{sec-norms}. 

\begin{proposition} {(Bound on Couples)} \label{prop-couples}
For each $1 \leq n \leq N^3$ and $0 \leq t \leq  1$, 
\begin{equation}\label{eq-couples}
\left|\sum_{\Q} (\K_\Q)(t,t,k)\right| \lesssim \langle k \rangle^{-20} T^{-\frac{3}{5}}(L^\theta \alpha T^{\frac{4}{5}})^n, 
\end{equation}
where the summation is taken over all couples $\Q = \{\T^+, \T^{-}\}$ such that $n(\T^+) + n(\T^-) = n$. 
\end{proposition}

\begin{proposition}{(Bound on Operator $\Ell$)} \label{prop-remainder}
With probability $\geq 1 - L^{-A}$, the linear operator $\Ell$ defined in (\ref{eq-L}) satisfies  
\begin{equation} \label{remains}
||\Ell^n||_{Z \to Z} \lesssim \left( L^\theta \alpha T^{4/5}\right)^{n/2}L^{50}
\end{equation}
for each $ 0 \leq n \leq N$ and some $A \geq 40$. 
\end{proposition}

\begin{remark}
To prove Theorems \ref{mainthm} and \ref{mmtthm}, we require Proposition \ref{prop-couples} over sums of certain subsets of couples of order $n$. However, the proof of Proposition \ref{prop-couples} clarifies for which subsets of couples we may sum over and requires the notion of congruence defined in Section \ref{subsec-cancellation}. 
\end{remark}
\subsection{Molecules} In order to prove the above Propositions, we will need the notion of molecules defined as in \cite{WKE, WKE2023}: 

\begin{definition}{(Molecules)} \label{def-molecules}
A \textit{molecule} $\mathbb M$ is a directed graph, formed by vertices (called \textit{atoms}) and edges (called \textit{bonds}), where multiple and self-connecting bonds are allowed, and each atom has out-degree at most 2 and in-degree at most 2. We write $v \in \mathbb M$ and $\ell \in \mathbb M$ for atoms $v$ and bonds $\ell$ in $\mathbb M$, and write $\ell \sim v$ if $v$ is an endpoint of $\ell$. For $\ell \sim v$, we may define $\zeta_{v,\ell}$ to be 1 if $\ell$ is outgoing from $v$, and -1 otherwise. We further require that $\mathbb M$ does not have any connected components of only degree 4 atoms, where connectivity is understood in terms of undirected graphs. For a molecule, define 
\begin{equation}
\chi := E - V + F,
\end{equation}
where $E$ is the number of bonds, $V$ is the number of atoms, and  $F$ is the number of components. An \textit{atomic group} in a molecule is a subset of atoms, together with all bonds between these atoms. A single bond $\ell$ is a \textit{bridge} if removing $\ell$ adds one new component. 
\end{definition}

\begin{definition}{(Molecule of couples)} \label{couple-molecule}
Let $\mathcal Q$ be a nontrivial couple. We define the corresponding \textit{molecule} $\mathbb M = \mathbb M(\mathcal Q)$ as follows. The atoms correspond to the branching nodes $\mathfrak n \in \mathcal N$. For branching nodes $\mathfrak n_1, \mathfrak n_2$, we connect the corresponding atoms denoted $v_1, v_2$ if (i) one of $\mathfrak n_1, \mathfrak n_2$ is a parent of the other, or (ii) a child of $\mathfrak n_1$ is paired to a child of $\mathfrak n_2$ as leaves. We may form a \textit{labelled molecule} by labeling each bond. In the case of (i) we label this bond as PC (Parent-Child) and place a P at the parent atom and C at the child atom  (from the corresponding couple), and in the case of (ii), we label the bond by LP (Leaf Pair). 

We fix the direction of each bond as follows. An LP bond should go away from the atom whose corresponding child in the couple has sign - towards the other atom whose corresponding child in the couple has a + sign. A PC bond should go from the P to the C if the child atom corresponds to a branching node with sign -, and the opposite if the child atom corresponds to a branching node with sign +. See Figure \ref{fig-eg-labeled-mol} to see the labelled molecule corresponding to the couple and molecule in Figure \ref{fig-eg-couple}. 

For any atom $v \in \mathbb M(\mathcal Q)$, let $\mathfrak n = \mathfrak n(v)$ be the corresponding branching node in $\mathcal Q$. For any bond $\ell \sim v$, define $\mathfrak m = \mathfrak m(v,l)$ such that (i) if $\ell$ is PC with $v$ labelled C, then $\mathfrak m = \mathfrak n$; (ii) if $\ell$ is PC with $v$ labelled P, them $\mathfrak m$ is the branching node corresponding to the other endpoint of $\ell$; (iii) if $\ell$ is LP them $\mathfrak m$ is the leaf in the leaf pair defining $\ell$ that is a child of $\mathfrak n$. 
\end{definition}

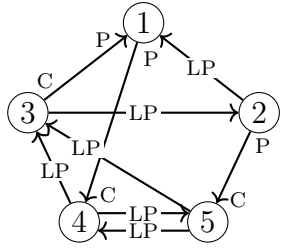
\begin{figure}
\begin{tikzpicture}[scale = .85]
    \tikzstyle{open node} = [circle, draw = black, inner sep = 2]
    \tikzstyle{empty node} = [inner sep=1,outer sep=0]
    \node[open node] (3) at (.2, .1) {3};
    \node[open node] (2) at (3.8, .1) {2};
    \node[open node] (1) at (2,1.5) {1};
    \node[open node] (4) at (1, -1.6) {4};
    \node[open node] (5) at (3, -1.6) {5};

    \draw[->, thick] (3) -- node[empty node, midway, fill = white, draw = none]{\tiny{LP}} (2);
    \draw[->, thick] (2) -- node[empty node, midway, fill = white, draw = none]{\tiny{LP}} (1);
    \draw[->, thick] (3) -- 
        node[above, pos = .01, draw = none]{\tiny{C}}
        node[left, pos = .94, draw = none]{\tiny{P}} (1);
    \draw[->, thick] (1) -- 
        node[right, pos = .1, draw = none]{\tiny{P}}
        node[right, pos = .95, draw = none]{\tiny{C}}(4);
    \draw[->, thick] (4.25) --node[empty node, midway, fill = white,draw = none]{\tiny{LP}} (5.155);
    \draw[->, thick] (5.205) -- node[empty node, midway, fill = white, draw = none]{\tiny{LP}} (4.335);
    \draw[->, thick] (5) -- node[empty node, pos = .72, fill = white, draw = none]{\tiny{LP}} (3);
    \draw[->, thick] (4) -- node[empty node, pos = .45, fill = white, draw = none]{\tiny{LP}} (3);
    \draw[->, thick] (2) -- 
        node[right, pos = .2, draw = none]{\tiny{P}}
        node[right, pos = .95, draw = none]{\tiny{C}}(5);
\end{tikzpicture}
\caption{An example of a labelled molecule.}
\label{fig-eg-labeled-mol}
\end{figure}

\begin{proposition}{(Correspondence between molecules and couples)}
We have the following relationship between molecules and couples: 
\begin{enumerate}
\item For any nontrivial couple $\mathcal Q$ of order $n$, the graph $\mathbb M = \mathbb M(\mathcal Q)$ defined in Definition \ref{couple-molecule} is a molecule. $\mathbb M$ is connected with $n$ atoms and $2n-1$ bonds and has either two atoms of degree 3 or one of degree 2, with all remaining atoms having degree 4. 
\item Given a molecule $\mathbb M$ with $n$ atoms as in Definition \ref{def-molecules}, the number of couples $\mathcal Q$ (if any) such that $\mathbb M(\mathcal Q) = \mathbb M$ is at most $C^n$.
\end{enumerate}
\end{proposition}
\begin{proof}
See Propositions 9.4 and 9.6 of \cite{WKE}. 
\end{proof}

\begin{definition}{(Decorations of Molecules)}
Given a molecule $\mathbb M$, suppose we fix $k_v \in \Z_L$ and $\beta_v \in \R$ for each $v \in \mathbb M$ such that $k_v = 0$ when $v$ has degree 4. Define a $(k_v, \beta_v)$-decoration of $\mathbb M$ to be the set $(k_\ell)$ for all bonds $\ell \in \mathbb M$, such that $k_\ell \in \Z_L$ and for each atom $v$,
\begin{align}\label{eq-mol-k}
\sum_{\ell \sim v} \zeta_{v,\ell} k_\ell &= k_v, \\
\left|\Gamma_v - \beta_v \right| &< T^{-1},
\end{align}
where $\Gamma_v$ is defined as 
\begin{align}\label{eq-mol-omega}
\Gamma_v &= \sum_{\ell \sim v} \zeta_{v, \ell} \omega(k_\ell).
\end{align}
Suppose $\mathbb M = \mathbb M(\mathcal Q)$ comes from a nontrivial couple $\mathcal Q$. For $k \in \Z_L$, we define a $k$-decoration of $\mathbb M$ to be a $(k_v, \beta_v)$-decoration where 
\begin{equation}
k_v = \begin{cases}
0 & \text{if } v \text{ has degree 2 or 4,} \\
+k & \text{if } v \text{ has out-degree 2 and in-degree 1,} \\
-k & \text{if } v \text{ has out-degree 1 and in-degree 2.} 
\end{cases}
\end{equation}
Given any $k$-decoration of $\mathcal Q$ in the sense of Definition \ref{def-couples}, define a $k$-decoration of $\mathbb M(\mathcal Q)$ such that $k_{\ell} = k_{\mathfrak m(v,\ell)}$ for an endpoint $v$ of $\ell$. It is easy to verify that $k_\ell$ is well-defined, and gives a one-to-one correspondence between $k$-decorations of $\mathcal Q$ and $k$-decorations of $\mathbb M(\mathcal Q)$. Moreover, for such decorations we have 
\begin{equation}
\Gamma_v = \begin{cases}
0 & \text{if } v \text{ has degree 2,} \\
-\zeta_{\mathfrak n(v)}\Omega_{\mathfrak n(v)} & \text{if } v \text{ has degree 4,} \\
-\zeta_{\mathfrak n(v)}\Omega_{\mathfrak n(v)} + \omega(k) & \text{if } v \text{ has out-degree 2 and in-degree 1,} \\
-\zeta_{\mathfrak n(v)}\Omega_{\mathfrak n(v)} - \omega(k) & \text{if } v \text{ has out-degree 1 and in-degree 2.} \\
\end{cases}
\end{equation}
\end{definition}

\begin{definition}{(Degenerate Atoms)}
Given a molecule $\mathbb M$ with a $(k_v, \beta_v)$-decoration, an atom $v$ is \textit{degenerate} if there are two bonds $\ell_1, \ell_2 \sim v$ of opposite direction such that $k_{\ell_1} = k_{\ell_2}$. Furthermore, $v$ is \textit{fully degenerate} if all bonds $\ell \sim v$ have the same value of $k_\ell$. 
\end{definition}

\subsection{Statement of the counting estimates} 
\label{subsec-counting}Another key part in proving the main estimates above will be to reduce them to counting estimates on the number of decorations of molecules. This will be done via an algorithm which will count smaller atomic groups, focusing on one or two atoms at a time. Therefore, we introduce notation for counting, coming from \cite{WKE2023}:

\begin{definition}\label{def-counting-notation}
Fix $r \leq 3$, $k \in \Z_L$, $\beta \in \R$, $a_1, \ldots, a_r \in \R$, and $\epsilon_1, \ldots, \epsilon_n \in \{\pm\}$. Then, a tuple $(j_1^{\epsilon_1}, \ldots, j_r^{\epsilon_r})$ denotes the set of variables $\{k_1, \ldots k_r\} \in \Z_L^r$ which satisfy $|k_j - a_j| \leq 1$ as well as
\begin{equation}
\sum_{i = 1}^r \epsilon_i k_{j_i} = k, \hspace{.5cm} \left| \sum_{i = 1}^r \epsilon_i |k_{j_i}|^\sigma - \beta \right| \leq T^{-1}.
\end{equation}
We also impose that among $\{k, a_1, \ldots, a_r\}$ at most two are $\geq D$ for some fixed $D \gtrsim 1$, possibly dependent on $L$ (this condition is not needed when $\sigma = 2$). We may sometimes omit the signs $\epsilon_j$, in which case it is assumed the signs can be arbitrary. Denote the corresponding number of solutions $\{k_1, \ldots, k_r\}$ by $\mathfrak C$ and we say the tuple is associated to an $r$-vector counting ($r$ v.c.).

Suppose a collection of $m$ tuples, each of which has its own fixed value of $r^{(i)}$, $k^{(i)}$, and $\beta^{(i)}$ $(i=1, \ldots, m)$, contains $n$ distinct values of $j$ and we fix $a_1, \ldots, a_n \in \R$ (eg. The collection of tuples $\{(1^+, 2^-, 3^+), (1^+, 4^-)\}$ has $n = 4$). Then, the collection of tuples denotes the set of variables $\{k_1, \ldots, k_n\}$ satisfying the joint system coming from each of the tuples. We denote the corresponding number of solutions $\{k_1, \ldots, k_n\}$ by $\mathfrak C$ and say the collection of tuples is associated to an $n$-vector counting ($n$ v.c.).
\end{definition}

\begin{proposition} \label{prop-vc}
Fix $T$ so that if $\sigma \neq 2$, $T \ll L$ and further if $0 < \sigma < 1$, $T \ll L^{\frac{1}{2 - \sigma}}$. Then, for any $\theta > 0$, we have the following bounds: 
\begin{enumerate}
\item For $(1,2)$, we have $\mathfrak C \lesssim L$.
\item For $(1^+, 2^-, 3^+)$, we have $\mathfrak C \lesssim L^{2 + \theta}T^{-1}D^{2-\sigma}$. 
\end{enumerate}
Furthermore, for $(1,2)$ if at least one of $a_1, a_2 \leq D$, then
\begin{enumerate}
\item For $(1^+, 2^+)$, we have $\mathfrak C \lesssim LT^{-\frac{1}{2}} D$.
\item For $(1^+, 2^-)$, we have $\mathfrak C \lesssim \min(L, LT^{-1}h^{-1} D^{2-\sigma})$, where $h = \min(1, |k|)$. 
\end{enumerate}
\end{proposition} 

\section{Discussion of proof and results}
\label{sec-overview}
\subsection{Overview of the proof} While the result covers subcritical timescales, our proof utilizes the strategy in \cite{WKE, WKE2023}, which dealt with the critical problem.  Proposition \ref{prop-couples} is used to bound higher iterates ($n \geq 3$), while Proposition \ref{prop-remainder} allows us to bound the remainder term, with the lowest iterates being dealt with explicitly. The bulk of the paper is devoted to the analysis of $\K_Q$, which can be generalized to the operator $\Ell$, defined in (\ref{eq-L}). 

The main step to analyzing $\K_\Q$ is to reduce it to a counting problem on molecules via Proposition \ref{prop-rigidity}, and then perform a counting algorithm. This algorithm counts decorations of bonds at an atom one at a time using the counting estimates in Proposition \ref{prop-vc}. The general counting estimate for $(1,2)$ of $\mathfrak C \lesssim L$ is the main hindrance in reaching $T_{\mathrm{kin}}$, so the algorithm is optimized to minimize the number of times this counting estimate is used. To further minimize the use of this counting estimate, we isolate irregular chains (Definition \ref{def-chain-couple}), structures in couples which solely depend on this counting estimate but luckily exhibit cancellation. This cancellation occurs between \textit{congruent} couples (Definition \ref{def-congruent}), whose irregular chains cancel with one another, allowing us to remove the chains via \textit{splicing} (Definition \ref{def-splicing}). This is why Proposition \ref{prop-couples} is stated for sums of $\K_Q$, which allows cancellation between congruent couples. After this splicing, the problem reduces to counting estimates, which we discuss further in Section \ref{subsec-counting-discussion}. 

Bounding $\Ell^n$ in Proposition \ref{prop-remainder} allows us to invert the operator $(1 - \Ell)$ and in turn bound the remainder $b$ using a contraction mapping with (\ref{eq-binv}). In order to prove Proposition \ref{prop-remainder}, we extend Proposition \ref{prop-couples} to \textit{flower trees} and \textit{flower couples} (Definition \ref{def-flower}) in which we fix the decoration at a leaf, arbitrarily large. Note that in Proposition \ref{prop-remainder}, no expectations are taken and rather we move to a set of overwhelming probability using large deviation estimates. 

To deal with the lowest iterates, we explicitly write the exact sums coming from these couples. For $n = 0$, we obtain the initial data $n_{\mathrm{in}}$, while the $n = 1$ iterate is 0 and $n = 2$ gives the collision kernel (\ref{KIN}) to within $o_{\ell_k^\infty}\left(\frac{t}{T_{\mathrm{kin}}}\right)$. A key tool when $n = 2$ is Corollary \ref{cor-iterates} which allows us to switch from a sum to integral(s) converging to the first iterate. 

\subsection{Counting in \texorpdfstring{$d = 1$}{d=1} and the algorithm} \label{subsec-counting-discussion} Here, we point out several subtleties of our counting estimates in Proposition \ref{prop-vc} and compare to the higher dimensional setting. 

\subsubsection{Dependence of $T$ on $\sigma$} As mentioned above, the counting estimate of $\mathfrak C \lesssim L$ for $(1,2)$ prevents us from reaching $T_{\mathrm{kin}}$ for any value of $\sigma$. However, the counting estimates additionally impose that $T \ll L$ if $\sigma \neq 2$ and further that $T \ll L^{\frac{1}{2-\sigma}}$ if $0 < \sigma < 1$. This prevents us from reaching $L^{-\epsilon}\alpha^{-\frac{5}{4}}$ for some scaling laws if $\sigma \neq 2$. 

\subsubsection{Algorithm on Molecules} Each step of our algorithm removes an atom and/or bond(s), reducing the value of $\chi$ by $\Delta \chi$ on the remaining molecule and performing the counting problem on the removed bonds to obtain the value $\mathfrak C$. In order to reach $T = T_{\mathrm{kin}}$, on average each step of the algorithm would need to satisfy $\left(\frac{\alpha T}{L} \right)^{-\Delta \chi} \mathfrak C \ll 1$, or $\mathfrak C \ll (LT^{-\frac{1}{2}})^{-\Delta \chi}$. The three-vector countings we perform in our algorithm have $\Delta \chi = -2$ and $\mathfrak C \lesssim L^2 T^{-1}$, while the two-vector ones have $\Delta \chi = -1$ and in the worst case $\mathfrak C \lesssim L$. So, while our three-vector countings would be sufficient to get close to $T_{\mathrm{kin}}$, the same cannot be said of all of our two-vector counting estimates. The two-vector countings which would be sufficient require additional assumptions on the size and/or gap of the decoration. 

Using only the counting estimates in Proposition \ref{prop-vc}, we encounter many molecules that must use the bad estimate of $\mathfrak C \lesssim L$ for two-vector counting. The smallest such molecule is in Figure \ref{fig-eg-couple}. Our algorithm allows us to bound the number of such two-vector counting steps in Proposition \ref{prop-2vc-bound}, after cancellation, by
\begin{equation} \label{eq-rough-bound}
\text{\# two-vector countings} \leq 3 \times (\text{\# three-vector countings} -1).
\end{equation}
This bound implies that we must consider $T < \alpha^{-\frac{5}{4}}$. In order to achieve it, we map the two-vector countings to the three-vector countings (Definition \ref{def-mapping}), for which we develop \textit{operation trees} (Definition \ref{def-op-tree}) to track the steps of the algorithm precisely. This method is completely new and well-adapted to the case of $d = 1$. The additional three-vector counting in (\ref{eq-rough-bound}) allows us to obtain the factor $o_{\ell_k^\infty}\left(\frac{t}{T_{\mathrm{kin}}}\right)$ in Theorems \ref{mainthm} and \ref{mmtthm}. 

The bound (\ref{eq-rough-bound}) is not quite sharp, in the sense that the factor of 3 in \eqref{eq-rough-bound} could be improved to a factor of 2. This was essentially shown in the more recent work \cite{1DWWWKE}. However, this bound cannot be improved as there is a molecule which saturates it, seen in Figure \ref{fig-eg-molecule}. We show the repeated pattern twice, although it may appear any number of times and will still saturate (\ref{eq-rough-bound}). Note that there are couples corresponding to this molecule after the splicing step (ie. the double bonds are not part of irregular chains and therefore do not exhibit cancellation). 

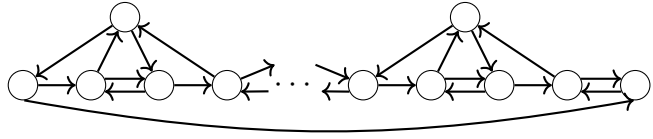
\begin{figure}
\begin{tikzpicture}[scale = .9]
    \tikzstyle{n node} = [circle, draw = black]
    \tikzstyle{o node} = [inner sep=2,outer sep=0]
    \node[n node] (2) at (-4,0) {};
    \node[n node] (3) at (-3,0) {};
    \node[n node] (4) at (-2,0) {};
    \node[n node] (5) at (-1,0) {};
    \node[o node] (13) at (0,0) {$\cdots$};
    \node[n node] (6) at (1,0) {};
    \node[n node] (7) at (2,0) {};
    \node[n node] (8) at (3,0) {};
    \node[n node] (9) at (4,0) {};
    \node[n node] (10) at (5,0) {};
    \node[n node] (11) at (-2.5,1) {};
    \node[n node] (12) at (2.5,1) {};

    \draw[->,  thick] (2) -- (3);
    \draw[->,  thick] (3.25) -- (4.155);
    \draw[->, thick] (4.205) -- (3.335);
    \draw[->,  thick] (4) -- (5);
    \draw[->,  thick] (5.25) -- (-.3,.3);
    \draw[->, thick] (13.193) -- (5.335);
    \draw[->,  thick] (.3, .3) -- (6.155);
    \draw[->, thick] (6.205) -- (13.347);
    \draw[->,  thick] (6) -- (7);
    \draw[->,  thick] (7.25) -- (8.155);
    \draw[->, thick] (8.205) -- (7.335);
    \draw[->,  thick] (8) -- (9);
    \draw[->,  thick] (9.25) -- (10.155);
    \draw[->, thick] (10.205) -- (9.335);
    \draw[->,  thick] (3) -- (11);
    \draw[->,  thick] (11) -- (2);    
    \draw[->,  thick] (5) -- (11);    
    \draw[->,  thick] (11) -- (4);
    \draw[->,  thick] (7) -- (12);
    \draw[->,  thick] (12) -- (6);    
    \draw[->,  thick] (9) -- (12);    
    \draw[->,  thick] (12) -- (8);
    \draw[->,   thick] (2.275) to [out=-10,in=-170] (10.265);
\end{tikzpicture}
\caption{A molecule demonstrating a restriction on the bound for two-vector countings.}
\label{fig-eg-molecule}
\end{figure}

\subsubsection{Lack of combined counting estimates} In higher dimensions, two-vector counting is also an issue. However, in \cite{WKE, WKE2023}, the authors notice that a gain occurs by combining a two-vector counting with a three-vector counting in a 5-vector counting estimate and utilize that to overcome two-vector counting losses. The counting problems for which they see this gain are represented by the tuples $\{(1^+, 2^-, 3^+), (1^+, 4^-, 5^+)\}$ and $\{(1^+, 2^-, 3^+), (1^+, 2^-, 4^+, 5^-)\}$, whose corresponding atomic groups are shown in Figure \ref{fig-5vc}. When performed in two steps, these can have an estimate of $L^3T^{-1}$, worse than the $L^3 T^{-\frac{3}{2}}$ needed to reach $T_{\mathrm{kin}}$. However, a five-vector counting in $d \geq 3$ gives an estimate of $L^3T^{-2}$. The following proposition proved in Section \ref{sec-counting} shows that this five-vector counting estimate is not generally true in $d  =1$. 

\begin{proposition}\label{prop-5vc}
Fix $\sigma = 2$. Then, there are choices of $a_i$ and $\beta^{(j)}, k^{(j)}$ such that \newline$\{(1^+, 2^-, 3^+), (1^+, 4^-, 5^+)\}$ has $\mathfrak C \gg L^3T^{-\frac{3}{2}}$, where $\mathfrak C$ denotes the solution to the joint counting problem and $\beta^{(j)}, k^{(j)}$ denote the fixed constants $\beta, k$ for the $j$th tuple. 
\end{proposition}

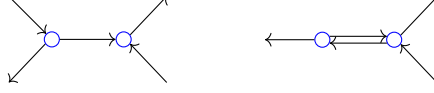
\begin{figure}[h]
\begin{tikzpicture}[scale = .95]
\tikzstyle{hollow node}=[circle,draw,inner sep=2]
\tikzset{edge/.style = {->,> = latex'}}
\hspace{-1cm}
\node[hollow node, blue] (C) at (-7,0) {};
\node[hollow node, blue] (D) at (-6,0) {};

\draw[->] (C.0) -> (D.180);

\draw[->] (C.215) to (-7.6,-.6);
\draw[->] (-7.6,.6) to (C.135);

\draw[->] (D.35) to (-5.4,.6);
\draw[->] (-5.4,-.6) to (D.325) ;
\end{tikzpicture}
~\hspace{1cm}
\begin{tikzpicture}[scale = .95]
\tikzstyle{hollow node}=[circle,draw,inner sep=2]
\tikzset{edge/.style = {->,> = latex'}}
\node[hollow node, blue] (A) at (-6,0) {};
\node[hollow node, blue] (B) at (-5,0) {};

\draw[->] (A.25) -> (B.155);
\draw[->] (B.205) -> (A.335);

\draw[->] (A.180) to (-6.8,0);

\draw[->] (B.35) to (-4.4,.6);
\draw[->] (-4.4,-.6) to (B.325);
\end{tikzpicture}
\caption{Atomic groups for five-vector counting.}
\label{fig-5vc}
\end{figure}

\subsubsection{The operator $\Ell$} Unlike the counting estimates in \cite{2019, WKE, WKE2023}, those in Proposition \ref{prop-vc} involve a parameter $D$, related to the size of the decorations. Due to the decay of $n_{\mathrm{in}}$, we typically may restrict the wave numbers to be $\lesssim L^{\theta}$. However, in the case of the remainder, when bounding $\Ell^n$, we fix a leaf to have decoration $\ell$ that can be arbitrarily large. This means that any nodes of the couple connecting the root to this leaf, which we refer to as the \textit{stem} (Definition \ref{def-flower}), can have arbitrarily large decoration. As each node on the stem has only one of its children along the stem, we need to allow for the possibility that two of the decorations can be large. Our improved two-vector counting estimates further require us to count at least one decoration off of the stem. To ensure that our algorithm for couples can be used for flower couples, we only use these improved estimates in the case of \textit{double bonds} (defined in Proposition \ref{prop-db}), which are guaranteed to not lie fully along the stem. Luckily, this still allows us to achieve the bound (\ref{eq-rough-bound}) on two-vector countings for both types of couples.

\subsection{Plan of the paper} In Section \ref{sec-counting} we prove the counting estimates stated in Propositions \ref{prop-vc} and \ref{prop-5vc}, as well as prove a related estimate on the iterates of (\ref{WKE}). In Section \ref{sec-splice}, we introduce relevant structures in couples and molecules, and demonstrate how to remove irregular chains via cancellation. In Section \ref{sec-red}, we reduce Proposition \ref{prop-couples} to Proposition \ref{prop-rigidity}, a bound on the number of decorations of molecules. In Section \ref{sec-algorithm} we lay out an algorithm to count decorations of molecules, reducing Proposition \ref{prop-rigidity} to Proposition \ref{prop-2vc-bound}, a bound on the number of two-vector countings. In Section \ref{sec-opbound}, we prove Proposition \ref{prop-2vc-bound} by developing a mapping of two-vector countings to three-vector countings. In Section \ref{sec-opL}, we control the kernels of $\Ell^n$ to prove Proposition \ref{prop-remainder}. Finally, in Section \ref{sec-thm}, we prove Theorems \ref{mainthm} and \ref{mmtthm}. 

\section{Counting estimates}
\label{sec-counting}
Here, we start with the proof of Proposition \ref{prop-vc} through Lemmas \ref{lem-sg2vc}-\ref{lem-3vc} when $\sigma \neq 2$. The case of $\sigma = 2$ is established in \cite{WKE2023}. We also prove a related estimate on iterates in Corollary \ref{cor-iterates}. Finally, we prove Proposition \ref{prop-5vc}.

\begin{lemma}{(Uni-directional 2 v.c.)}\label{lem-sg2vc}
Consider the counting problem corresponding to the tuple $(1^+, 2^+)$ with one of $a_1, a_2 \leq D$. Then 
\begin{equation}
\mathfrak C \lesssim LT^{-\frac{1}{2}}D.
\end{equation}
\end{lemma}
\begin{proof}
Let $X = Lk_1$ and note that we are looking for integers $X$ such that the function $F(X) = |X|^\sigma + |Lk-X|^\sigma$ satisfies 
\begin{equation}\label{eq-Fx}
F(X) = L^\sigma \beta + O(L^\sigma T^{-1}).
\end{equation}
Note that we may split the solutions $X$ into $O(1)$ intervals such that each of $X$ and $Lk - X$ do not change sign throughout the interval. So, it suffices to show that each of these intervals has length $\lesssim LT^{-{\frac{1}{2}}}D$. Suppose $X$ lies in one of these intervals and that $LD \gtrsim |\nu| \gtrsim LT^{-\frac{1}{2}}D$. We aim to show $X + \nu$ can no longer lie in this interval, where we may assume that $\nu$ and $X$ have the same sign. We consider the relative signs of $X$ and $Lk - X$.  
\begin{enumerate}
\item If $\mathrm{sgn}(X) = -\mathrm{sgn}(Lk - X)$, note that the signs of $|X + \nu| - |X|$ and ${|Lk - X - \nu| - |Lk - X|}$ are the same. Due to symmetry, we may assume that $|X| \lesssim LD$, and therefore it is enough to observe by the mean value theorem 
\begin{align*}
|X + \epsilon|^\sigma - |X|^\sigma &\gtrsim |\nu| \min (|X|^{\sigma - 1}, |\nu|^{\sigma - 1}) \\
&\gtrsim L^\sigma T^{-\frac{1}{2}}D^\sigma \min(1, T^{\frac{1-\sigma}{2}}) \\
&\gtrsim L^\sigma T^{-1}.
\end{align*}
\item If $\mathrm{sgn}(X) = \mathrm{sgn}(Lk -X )$, due to symmetry, we may assume that $|X| \geq |Lk - X|$ and $LT^{-\frac{1}{2}}D \lesssim |Lk - X| \lesssim LD$. Therefore, 
\begin{align*}
|X + \nu|^\sigma - |X|^\sigma + |Lk - X - \nu|^\sigma - |Lk - X|^\sigma & \gtrsim|Lk - X + \nu|^\sigma + |Lk - X - \nu|^\sigma - 2|Lk - X|^\sigma \\
& \gtrsim \nu^2 \min(|LD|^{\sigma - 2}, |\nu|^{\sigma - 2}) \\
& \gtrsim L^\sigma T^{-1}. 
\end{align*}
\end{enumerate}
\end{proof}

\begin{lemma}{(Large gap 2 v.c.)} \label{lem-lg2vc}Consider the counting problems corresponding to the tuples $(1^+, 2^-)$ and $(1^+, 2^+)$ for $\sigma \neq 2$, setting $h = \min(1, |k|)$. If $h \gtrsim T^{-1}$ and at least one of $a_1, a_2 \leq D$, then
\begin{equation}
\mathfrak C \lesssim LT^{-1}h^{-1} D^{2 - \sigma}. 
\end{equation}
\end{lemma}
\begin{proof}
Throughout we assume that $|a_1 \pm a_2 - k| \leq 4$, depending on the tuple, as if not, we trivially have $\mathfrak C = 0$. It is enough to determine $k_1$ as this fixes $k_2$. So, let $W$ be a bump function localizing $k_1$ near $a_1$. Note also that if any $|k_i| \lesssim T^{-1}$ , we automatically achieve the desired bound, so we may assume that all $|k_1|, |k_2| \gtrsim T^{-1} $ and reflect this in $W$. Namely, if $\chi$ is a function of sufficient decay supported in a ball of radius $O(1)$, we choose
\begin{equation}
W(u) = \chi(u - a_1) \chi(u - k \pm a_2) (1 - \chi(T u))(1 - \chi(T(u - k))). 
\end{equation}
Additionally, let $\psi \geq 0$ be a bump function localized around $\beta$ such that $\psi \geq 1$ on $B(\beta, 1)$ and $\hat{\psi}$ is supported on a ball of radius $O(1)$. Let
\begin{equation}
\Omega(u) = |u|^\sigma \pm |u - k|^\sigma,
\end{equation} 
so that 
\begin{equation*}
\nabla \Omega (u) = \sigma (\mathrm{\mathrm{sgn}}(u)|u|^{\sigma - 1} \pm \mathrm{\mathrm{sgn}}(u - k)|u - k|^{\sigma - 1}). 
\end{equation*}
It is enough to estimate
\begin{align}
\sum_{k_1 \in \Z_L} W(k_1) \psi(T\Omega(k_1)) &= \sum_{k_1 \in \Z_L} W(k_1) T^{-1}\int_\tau\hat{\psi}\left( \frac{\tau}{T}\right) e(\tau \Omega(k_1)) \diff \tau\label{eq-psift} \\
&=LT^{-1}\sum_{f \in \Z} \int_\tau \hat{\psi}\left( \frac{\tau}{T}\right) \int_u W(u) e(\tau \Omega(u) - Luf)\diff u \diff \tau,\label{eq-ps}
\end{align}
where we are rewriting $\psi$ in terms of its time Fourier transform in (\ref{eq-psift}) and then using Poisson Summation in (\ref{eq-ps}). Letting $\Phi(u) = \tau \Omega(u) - Luf$, note that $\nabla \Omega$ is localized in a ball of radius at most $O(1)$ for $1 < \sigma < 2$ and in a ball of radius at most $O(T^{1 - \sigma})$ for $0 < \sigma < 1$. Therefore, as $\tau \lesssim T$ and using our choice of $T$, for all but $O(1)$ values of $f$, we have 
\begin{equation*}
\left| \nabla \Phi(u) \right| \gtrsim L|f|.
\end{equation*}
For the remaining terms, we may integrate by parts sufficiently many times, noting that we may also lose a factor of $T$ at each step. However, since $\frac{T}{L} \lesssim L^{-\epsilon}$, we may integrate at least $M$ times, where $\left( \frac{T}{L}\right)^M \ll T^{-1}$. 

For the $O(1)$ terms, we rewrite them by reversing the Fourier Transform in $\psi$ and absorbing the phase into $W$ via $\widetilde W$ as: 
\begin{equation} \label{eq-O1terms}
L\int_u \widetilde{W}(u)\psi(T\Omega(u)) \diff u.
\end{equation}
Performing a change of variables by replacing $u$ with $s = \Omega(u)$, (\ref{eq-O1terms}) becomes
\begin{equation}\label{eq-O1integral}
L \int_s \widetilde{W}(u(s))\psi(Ts) \frac{1}{\sigma} \frac{\diff s}{\mathrm{\mathrm{sgn}}(u)|u|^{\sigma - 1} \pm \mathrm{\mathrm{sgn}}(u - k)|u-k|^{\sigma - 1}}.
\end{equation}
Note that by the mean value theorem, 
\begin{equation}
A:= \left| \frac{1}{\mathrm{sgn}(u)|u|^{\sigma - 1} \pm \mathrm{sgn}(u - k)|u-k|^{\sigma - 1}}\right| \lesssim \begin{cases} \frac{|u|^{1 - \sigma}|u-k|^{1 - \sigma}\max(|u|, |u-k|)^{\sigma}}{|k|} & 0 < \sigma < 1, \\ \frac{\max(|u|, |u-k|)^{2 - \sigma}}{|k|}& 1 < \sigma < 2.\end{cases} 
\end{equation}
Now we consider several values of $k$:
\begin{enumerate}
\item If $|k|\leq 10D$, $A \lesssim \frac{D^{2 - \sigma}}{h}$. 
\item If $|k| \geq 10D$, $A \lesssim D^{2 - \sigma} $.
\end{enumerate}
Therefore, when integrating in $s$, we may bound (\ref{eq-O1integral}) by $LT^{-1}h^{-1} D^{2-\sigma}$.
\end{proof}

\begin{lemma}{(3 v.c.)}\label{lem-3vc}
Consider the counting problem corresponding to the tuple $(1^+, 2^-, 3^+)$ for $\sigma \neq 2$. If at least two of $k, a_1, a_2, a_3 \leq D$, then
\begin{equation}
\mathfrak C \lesssim L^2T^{-1}\log L D^{2 - \sigma}. 
\end{equation}
\end{lemma}
\begin{proof}
Throughout we assume that $|a_1 - a_2 + a_3 - k| \leq 6$, as if it is not, then trivially we have that $\mathfrak C = 0$.  It is enough to determine $(k_1, k_3)$ as this fixes $k_2$. Similar to Lemma \ref{lem-lg2vc}, we may also assume that $|k_i| \gtrsim T^{-1}$ and $|k - k_i| \gtrsim T^{-1}$. So, we choose
\begin{align}
W(u_1, u_3) = &\chi(u_1 - a_1) \chi(u_3-a_3) \chi(u_1 + u_3 - k - a_2) \nonumber\\
& \times (1 - \chi(Tu_1))(1 - \chi(Tu_3))(1 - \chi(T(u_1 + u_3 - k))) \\
& \times (1 - \chi(T(u_1-k)))(1 - \chi(T(u_3-k)))(1 - \chi(T(u_1+k_3 - 2k))). \nonumber
\end{align}
Let 
\begin{equation}\label{eq-Omega-13}
\Omega(u_1, u_3) = |u_1|^\sigma - |u_1 + u_3 - k|^\sigma + |u_3|^\sigma - |k|^\sigma,
\end{equation}
so that 
\begin{align*}
\nabla \Omega(u_1, u_3) = \sigma (&\mathrm{sgn}(u_1)|u_1|^{\sigma - 1} - \mathrm{sgn}(u_1 + u_3 - k)|u_1 + u_3 - k|^{\sigma - 1} , \\
&\mathrm{sgn}(u_3)|u_3|^{\sigma - 1} - \mathrm{sgn}(u_1 + u_3 - k)|u_1 + u_3 - k|^{\sigma - 1}).
\end{align*}
It is enough to estimate, where $\psi$ is as in Lemma \ref{lem-lg2vc},
\begin{align}\label{eq-sum-to-int}
\sum\limits_{k_1, k_3 \in \Z_L}&W(k_1, k_3) \psi(T\Omega(k_1, k_3)) \nonumber  \\
&=L^2T^{-1} \sum\limits_{f_1, f_3 \in \Z} \int_\tau\hat{\psi}\left(\frac{\tau}{T}\right) \int_{u_1, u_3} W(u_1, u_3) e(\tau \Omega(u_1, u_3) - Lu\cdot f) \diff u \diff \tau. 
\end{align}
Letting $\Phi(u) = \tau \Omega(u) - Lu\cdot f$,
note that as in Lemma \ref{lem-lg2vc}, $\nabla \Omega$ is localized in a ball of radius at most $O(1)$ for $1 < \sigma < 2$ and in a ball of radius at most $O(T^{1-\sigma})$ for $0 < \sigma < 1$. So, again we have that for all but $O(1)$ values of $(f_1, f_3)$,
\begin{equation}\label{eq-nabla-bound}
\nabla_{u_j} \Phi (u) \gtrsim L|f_j|.
\end{equation}
These terms we may treat as in Lemma \ref{lem-lg2vc} so that collectively, they are $\ll T^{-1}$. For the $O(1)$ terms, we may rewrite them by reversing the Fourier Transform in $\psi$ and absorbing the phase into $W$ via $\widetilde{W}$ as:
\begin{equation}\label{O1terms}
L^2\int_{u_1, u_3} \widetilde{W}(u_1, u_3) \psi(T\Omega(u)) \diff u.
\end{equation}
Performing a change of variables replacing $u_3$ by $s = \Omega(u_1, u_3)$, (\ref{O1terms}) becomes
\begin{equation*}
L^2\int_s \int_{u_1} \widetilde{W}(u_1, u_3(s, u_1)) \psi(Ts) \frac{1}{\sigma}\frac{\diff u_1 \diff s}{\mathrm{sgn}(u_3)|u_3|^{\sigma - 1} + \mathrm{sgn}(u_1 + u_3 - k)|u_1 + u_3 - k|^{\sigma - 1}}.
\end{equation*}
We bound this integral by first integrating in $u_1$, where we lose at most $\log L$, independent of $s$ and then we integrate in $s$ and gain a factor of $T^{-1}$. To see the $\log L$ bound when integrating in $u_1$, set 
\begin{align*}
A &:= \left|\frac{1}{\mathrm{sgn}(u_3) |u_3|^{\sigma - 1} + \mathrm{sgn}(u_1 + u_3 - k)|u_1 + u_3 - k|^{\sigma - 1}}\right| \\
&\lesssim \begin{cases} \frac{|u_3|^{1-\sigma} |u_1 + u_3 - k|^{1-\sigma}\max(|u_3|, |u_1 + u_3 - k|)^\sigma}{|u_1 - k|} & 0 < \sigma < 1 \\ \frac{\max(|u_3|, |u_1 + u_3 - k|)^{2 - \sigma}}{|u_1 - k|} & 1 < \sigma < 2.\end{cases}
\end{align*}
Now, we consider various values of $(a_2, a_3)$: 
\begin{enumerate}
\item If $a_2, a_3 \leq D$, $A \lesssim \frac{D^{2 - \sigma}}{|u_1 - k|}$, which is integrable in $u_1$ with at most $\log L$ losses as $|u_1 - k| \gtrsim T^{-1}$. 
\item If at most one $\geq D$, suppose $a_3$, then either $a_3 \leq 10D$, which reduces to the above case or $a_3 \geq 10D$, in which case $|u_1 - k| \gtrsim |a_3|$ and $A \lesssim  D^{2 - \sigma} $. 
\item If both $|a_2|, |a_3| \geq D$, we could replace $u_1$ by $s$ rather than $u_3$. 
\end{enumerate}
\end{proof}

\begin{corollary}{(Convergence of Iterates)}\label{cor-iterates}
Fix $k \in \Z_L$ and let $W \in \mathcal{S}(\R)$ and $\psi \in \mathcal{S}(\R)$ with $\hat{\psi}$ supported in a ball of radius $O(1)$. Then, there is $\delta \ll 1$ such that 
\begin{align}
 \sum_{k_1, k_3 \in \Z_L} &W(k_1, k_3)\psi(T\Omega(k_1, k_3)) = \nonumber\\ 
 &\sum_{\substack{(f_1, f_3) \in \Z^2 \\ |f_1|, |f_3| \lesssim R}}L^2\int_{u_1, u_3 \in \R} W(u_1, u_3)\psi(T\Omega(u_1, u_3))e(-Lu\cdot f)\diff u + O(L^{2 - \delta}T^{-1}), \label{eq-iterate-conv}
\end{align}
where $\Omega(k_1, k_3) = |k_1|^\sigma - |k - k_1 - k_3|^\sigma + |k_3|^\sigma - |k|^\sigma$ and $R = \max(T, T^{2 - \sigma})L^{-1 + 3\delta}$. 
\end{corollary}
\begin{proof}
Although the proof is similar to the proof of Lemma \ref{lem-3vc}, there are a few subtleties. First note that the inputs to $W$ may be arbitrarily large. However, for any $\delta > 0$, if $|k_1|, |k_3| \gtrsim L^{\delta}$, both terms are $\lesssim L^{2-\delta}T^{-1}$ by the decay of $W$. Similarly, for $0 < \sigma < 1$, if one of $|k_1|, |k - k_1 - k_3|, |k_3| \lesssim (L^{2\delta}T)^{-1}$, then again both terms are $\lesssim L^{2 - \delta}T^{-1}$, as $T^{2 - \sigma} \ll 1$. So, for $\chi$ a bump function supported in a ball of radius $O(1)$ around 0, we may replace $W$ with 
\begin{align}
W(k_1, k_3) = &\chi\left( \frac{k_1}{L^{1 + \delta}}\right)\chi\left( \frac{k_3}{L^{1 + \delta}}\right) \times \nonumber\\
&(1 - \chi((TL^{2\delta})k_1))(1 - \chi((TL^{2\delta})k_3))(1 - \chi((TL^{2\delta})(k - k_1 - k_3)))W(k_1, k_3),
\end{align}
where we omit the $1 - \chi(\cdot)$ terms if $\sigma > 1$. At this point, we may rewrite the sum as in (\ref{eq-sum-to-int}). Note that if $|f_1|, |f_3| \lesssim R$, we reverse the Fourier Transform in $\psi$ and obtain precisely the sum in $(f_1, f_3)$ in (\ref{eq-iterate-conv}). So, it remains to bound the remaining terms by $L^{2 - \delta}T^{-1}$. Note that similarly to above, $\nabla \Omega$ is in a ball of radius at most $L^{\delta(\sigma - 1)}$ if $1 < \sigma \leq 2$ and $(TL^{2\delta})^{1 - \sigma}$ if $0 < \sigma < 1$, where now these balls are centered at 0. So, as long as $\delta < \frac{\epsilon}{10}$, we again have (\ref{eq-nabla-bound}) whenever $|f_j| \gtrsim R$. As above, we integrate by parts sufficiently many times so that the sum of the remaining terms is $\lesssim L^{2 - \delta}T^{-1}$. Note that we lose a factor of $TL^{2\delta}$ when taking derivatives at each step for $0 < \sigma < 1$, in which case $TL^{2\delta} \ll L$. 
\end{proof}

\begin{proof}[Proof of Proposition \ref{prop-5vc}]
Let $\beta^{(1)} = \beta^{(2)} = 0$ and $k: = k^{(1)} = k^{(2)}$. Also, let all $a_i = k$. Then note that for $\Omega_j$ denoting the value of the resonance factor $\Omega$ for the $j$th tuple, 
\begin{align*}
\Omega_1 &= k_1^2 - k_2^2 + k_3^2 - k^2 \sim (k - k_1)(k - k_3),\\
\Omega_2 &= k_1^2 - k_4^2 - k_5^2 - k^2 \sim (k - k_1)(k - k_5),
\end{align*}
where we are rewriting $\Omega$ as in Lemma 3.2 of \cite{2019}. Therefore, let us set 
\begin{align}
X &= L(k - k_1), \nonumber \\
Y &= L(k - k_3), \\
Z &= L(k - k_5), \nonumber
\end{align}
so that we are looking for $(X,Y,Z) \in \Z^3$ with $|X|, |Y|, |Z| \leq L$ with 
\begin{align*}
XY & \leq L^2 T^{-1}, \\
XZ & \leq L^2T^{-1}.
\end{align*}
If $T < L$, we note that the following triples for $(X,Y,Z)$ satisfy the above equations and allow all $|k_i - a_i| \leq 1$ for all $i = 1, \ldots, 5$:  
\begin{equation*}
\# \{(X,Y,Z) \in \Z^3 | \hspace{.1cm} 0 < X \leq \frac{1}{2}LT^{-1}, 0 < Y < \frac{1}{2}L, 0 < Z< \frac{1}{2}L \} \sim L^3T^{-1}. 
\end{equation*}
Note that $L^3T^{-1} \gg L^3T^{-\frac{3}{2}}$. Similarly, if $T \geq L$, we may fix $X = 1$ and note that then any choice of $Y,Z$ with $|Y|, |Z| \leq L^2{T^{-1}}$ is sufficient and $L^4T^{-2} \gg L^3T^{-\frac{3}{2}}$ as $T \ll L^2$ when $\gamma < 1$, even if $T = T_{\mathrm{kin}}$. 
\end{proof}

\section{Stage 1: Splicing} \label{sec-splice}
We begin what we refer to as Stage 1 of the proof of Proposition \ref{prop-couples}. Here, we splice irregular chains from couples to yield a cancellation between them. 
\subsection{Double bonds and chains} \label{subsec-chains}
We start by introducing the necessary structures in couples and molecules for cancellation, adapting those in \cite{WKE, WKE2023} to suit our purposes.

\begin{proposition}\label{prop-db}
Let $\Q$ be a couple such that $\mathbb{M}(\Q)$ has nodes $v_1$ and  $v_2$ connected by precisely two edges with opposite directions (resp. same direction). Let $\mathfrak n_j = \mathfrak n(v_j)$. Then, up to symmetry, one of the following scenarios happens, as illustrated in Figure \ref{double-bond}: 
\begin{enumerate}[(a)]
\item \textbf{Cancellation (CL) Double Bond}: There is a child $\mathfrak n_{12}$ of $\mathfrak n_1$ which is paired to a child $\mathfrak n_{21}$ of $\mathfrak n_2$. The node $\mathfrak n_2$ is a child of $\mathfrak n_1$ such that $\mathfrak n_{12}$ and $\mathfrak n_2$ have opposite signs (resp. same signs). All other children of $\mathfrak n_1$ and $\mathfrak n_2$ which are leaves are not paired. In the molecule, this corresponds to a double bond with one LP and one PC bond.
\item \textbf{Connectivity (CN) Double Bond}: There are children $\mathfrak n_{11}$ and $\mathfrak n_{12}$ of $\mathfrak n_1$ with opposite sign (resp. same sign) and children $\mathfrak n_{21}$ and $\mathfrak n_{22}$ of $\mathfrak n_2$ with opposite sign (resp. same sign) so that $\mathfrak n_{11}, \mathfrak n_{12}$ are paired with $\mathfrak n_{21}, \mathfrak n_{22}$ (the exact pairings are dictated by sign) and the remaining children of $\mathfrak n_{1}$ and $\mathfrak n_{2}$ are not paired if they are leaves. Also, $\mathfrak n_2$ is not a child of $\mathfrak n_1$ or vice versa. In the molecule, this corresponds to a double bond with both being LP.
\end{enumerate}
We say that nodes $v_1$ and $v_2$ of the $\M(\Q)$ are connected via a CL, or CN, double bond. 
\end{proposition}
\begin{proof}
Each bond in the molecule is either a LP bond or a PC bond. It is not possible for both bonds in the double bond to be PC bonds, so at most one is a PC bond. 
\end{proof}

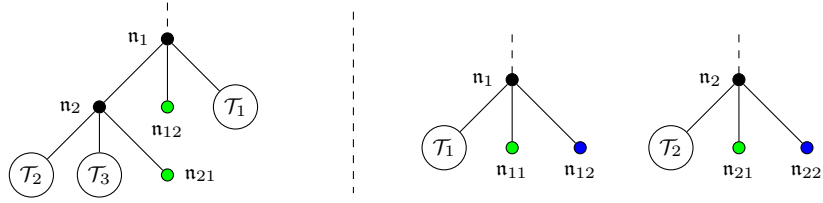
\begin{figure}
\begin{tikzpicture}[scale = .6]
\tikzstyle{hollow node}=[circle,draw,inner sep=2.0]
\tikzstyle{solid node}=[circle,draw,inner sep=1.6,fill=black]
\tikzset{
red node/.style = {circle,draw=black,fill=red,inner sep=1.6},
blue node/.style= {circle,draw = black, fill= blue,inner sep=1.6}, green node/.style = {circle,draw=black,fill=green,inner sep=1.6}}
\node[solid node, label = left: {\tiny $\mathfrak n_1$}] at (4.4,.4){}
    child{node[solid node, label = left: {\tiny $\mathfrak n_2$}]{}
        child{node[hollow node]{\tiny{$\mathcal T_2$}}}
        child{node[hollow node]{\tiny{$\mathcal T_3$}}}
        child{node[green node, label = right: {\tiny$\mathfrak n_{21}$}]{}}
    }
    child{node[green node, label = below: {\tiny$\mathfrak n_{12}$}]{}}
    child{node[hollow node]{\tiny{$\mathcal T_1$}}}
;
\draw[dashed] (4.4,1.2) -- (4.4,.4);
\draw[dashed] (8.5,1) -- (8.5,-3);

\node[solid node, label = left: {\tiny $\mathfrak n_1$}] at (12,-.5){}
child{node[hollow node]{\tiny{$\mathcal T_1$}}}
child{node[green node, label = below: {\tiny$\mathfrak n_{11}$}]{}}
child{node[blue node, label = below: {\tiny$\mathfrak n_{12}$}]{}}
;
\draw[dashed] (12,.5) -- (12,-.5);
\draw[dashed] (17,.5) -- (17,-.5);
\node[solid node, label = left: {\tiny $\mathfrak n_2$}] at (17,-.5){}
child{node[hollow node]{\tiny{$\mathcal T_2$}}}
child{node[green node, label = below: {\tiny$\mathfrak n_{21}$}]{}}
child{node[blue node, label = below: {\tiny$\mathfrak n_{22}$}]{}}
;
\end{tikzpicture}
\caption{CL and CN double bonds, viewed in the couple.}
\label{double-bond}
\end{figure}

\begin{definition}\label{def-chain} 
Given a couple $\Q$ and molecule $\M(\Q)$, a \textit{chain} is a sequence of atoms $(v_0, \ldots, v_{q})$ such that for $i \in \{0, \ldots, q-1\}$, $v_i$ and $v_{i+1}$ are connected via a CN or CL double bond. We call $q$ the \textit{length} of the chain. There are several chain variations, depicted in Figure \ref{fig-chain} at the level of the molecule, which we isolate:
\begin{itemize}
\item A \textit{hyperchain}  is a chain where additionally $v_0$ and $v_{q}$ are connected via a single bond.
\item A \textit{pseudo-hyperchain} is a chain where additionally $v_0$ and $v_{q}$ are each connected to an atom $v$, distinct from $v_0, \ldots, v_q$, via single bonds.
\item A \textit{CL chain} is a chain with only CL double bonds. 
\item A \textit{wide ladder} a collection of chains $\{(v_0^{(1)}, \ldots, v_{q^{(1)}}^{(1)}), \ldots,  (v_0^{(m)}, \ldots, v_{q^{(m)}}^{(m)})\}$ such that for each $1 \leq i \leq m-1$, there is a bond connecting $v_0^{(i)}$ to $v_0^{(i + 1)}$ and $v_{q^{(i)}}^{(i)}$ to $v_{q^{(i + 1)}}^{(i + 1)}$. Each chain in the wide ladder is called a \textit{rung}. See Figure \ref{fig-ladder}. 
\end{itemize}
A \textit{maximal chain} is a chain such that $v_0$ or $v_q$ are not connected to any other atoms not in the chain via CL or CN double bond. We may also refer to \textit{negative chains}, \textit{negative hyperchains}, \textit{negative pseudo-hyperchains}, and \textit{maximal negative chains} which are defined similarly with the further restriction that the two bonds in all of the double bonds have opposite directions. Furthermore, an \textit{irregular chain} is a CL chain which is a negative chain.  

A \textit{negative wide ladder} requires all chains in the collection to be \textit{negative chains}. A \textit{maximal (negative) wide ladder} of a set $\mathscr C$ of maximal (negative) chains is a wide ladder with each rung a maximal (negative) chain for which no (negative) chain in $\mathscr C$ can be added retaining the wide ladder structure.
\end{definition}

\begin{remark}
Note that a wide ladder where each rung has length one is a structure in \cite{WKE2023} called a ladder. 
\end{remark}

\begin{definition}\label{def-chain-couple}
Given a couple $\Q$, a sequence of nodes $(\mathfrak n_0, \ldots, \mathfrak n_q)$ is called a \textit{CL chain} if for $0 \leq j \leq q - 1$, (i) $\mathfrak n_{j+1}$ is a child of $\mathfrak n_j$, and (ii) $\mathfrak n_{j}$ has child $\mathfrak m_{j+1}$ paired to $\mathfrak p_{j + 1}$, a child of $\mathfrak n_{j+1}$. We define $\mathfrak n_0^1$ to be the remaining child of $\mathfrak n_0$ and $\mathfrak n_0^{11}$ and $\mathfrak n_0^{12}$ to be the remaining children of $\mathfrak n_q$. The chain $(\mathfrak n_0, \ldots, \mathfrak n_q)$ is an \textit{irregular chain} if additionally $\mathfrak n_i$ and $\mathfrak m_i$ have opposite sign for $1 \leq i \leq q$. Then, $\mathfrak n_0^{11}$ and $\mathfrak n_0^{12}$ have opposite sign and we will refer to $\mathfrak n_0^{11}$ as the child of $\mathfrak n_q$ having the same sign as $\mathfrak n_0$. Due to our correspondence between couples and labelled molecules, the corresponding sequence $(v_0, \ldots, v_q)$ in $\M(\Q)$ is a CL chain (resp. irregular chain) if and only if $(\mathfrak n_0, \ldots, \mathfrak n_q)$ is.  See Figure \ref{fig-irregular-chain}.
\end{definition}

\begin{figure}
\begin{subfigure}[t]{0.5\linewidth}
\hspace{-1cm}
\begin{tikzpicture}[scale = .9]
    \tikzstyle{every node} = [circle, draw = black]
    \node (1) at (-10,0) {};
    \node (2) at (-9,0) {};
    \node (3) at (-8,0) {};
    \node (4) at (-7,0) {};
    \node (5) at (-6,0) {};
    \node (6) at (-5,0) {};
    \node (7) at (-4,0) {};

    \draw[thick] (1.25) -- (2.155);
    \draw[thick] (2.205) -- (1.335);
    \draw[thick] (2.25) -- (3.155);
    \draw[thick] (3.205) -- (2.335);
    \draw[thick] (3.25) -- (4.155);
    \draw[thick] (4.205) -- (3.335);
    \draw[thick] (4.25) -- (5.155);
    \draw[thick] (5.205) -- (4.335);
    \draw[thick] (5.25) -- (6.155);
    \draw[thick] (6.205) -- (5.335);
    \draw[thick] (6.25) -- (7.155);
    \draw[thick] (7.205) -- (6.335);

    \draw[dashed] (1) -- (-10.6,.5);
    \draw[dashed] (1) -- (-10.6,-.5);

    \draw[dashed] (7) -- (-3.4,.5);
    \draw[dashed] (7) -- (-3.4,-.5);

    \draw[dashed, ->] (-2.6,0) -- node[above, pos = .45,draw = none]{\tiny{splicing}} (-1.4,0);

    \node (8) at (0,0) {};
    \draw[dashed] (8) -- (-.6,.5);
    \draw[dashed] (8) -- (-.6,-.5);
    \draw[dashed] (8) -- (.6,.5);
    \draw[dashed] (8) -- (.6,-.5);
\end{tikzpicture}
    \caption{A general chain, depicted after splicing if the chain is also a CL chain.}
\end{subfigure}
\begin{subfigure}[t]{0.5\linewidth}
\begin{tikzpicture}[scale = .9]
    \tikzstyle{every node} = [circle, draw = black]
    \node (1) at (-5,0) {};
    \node (2) at (-4,0) {};
    \node (3) at (-3,0) {};
    \node (4) at (-2,0) {};
    \node (5) at (-1,0) {};
    \node (6) at (0,0) {};
    \node (7) at (1,0) {};

    \draw[thick] (1.25) -- (2.155);
    \draw[thick] (2.205) -- (1.335);
    \draw[thick] (2.25) -- (3.155);
    \draw[thick] (3.205) -- (2.335);
    \draw[thick] (3.25) -- (4.155);
    \draw[thick] (4.205) -- (3.335);
    \draw[thick] (4.25) -- (5.155);
    \draw[thick] (5.205) -- (4.335);
    \draw[thick] (5.25) -- (6.155);
    \draw[thick] (6.205) -- (5.335);
    \draw[thick] (6.25) -- (7.155);
    \draw[thick] (7.205) -- (6.335);

    \draw[thick] (1.275) to [out=-10,in=-170] (7.265);

    \draw[dashed] (1) -- (-5.7,0);
    \draw[dashed] (7) -- (1.7,0);
    
\end{tikzpicture}
\caption{A hyperchain.}
\label{pic-chain-hyper}
\end{subfigure}
~
\begin{subfigure}[t]{0.5\linewidth}
\begin{tikzpicture}[scale = .9]
    \tikzstyle{every node} = [circle, draw = black]
    \node (1) at (-5,0) {};
    \node (2) at (-4,0) {};
    \node (3) at (-3,0) {};
    \node (4) at (-2,0) {};
    \node (5) at (-1,0) {};
    \node (6) at (0,0) {};
    \node (7) at (1,0) {};
    \node (8) at (-2,.8) {\tiny$v$};

    \draw[thick] (1.25) -- (2.155);
    \draw[thick] (2.205) -- (1.335);
    \draw[thick] (2.25) -- (3.155);
    \draw[thick] (3.205) -- (2.335);
    \draw[thick] (3.25) -- (4.155);
    \draw[thick] (4.205) -- (3.335);
    \draw[thick] (4.25) -- (5.155);
    \draw[thick] (5.205) -- (4.335);
    \draw[thick] (5.25) -- (6.155);
    \draw[thick] (6.205) -- (5.335);
    \draw[thick] (6.25) -- (7.155);
    \draw[thick] (7.205) -- (6.335);
    \draw[thick] (1.80) -- (8);
    \draw[thick] (8) -- (7.100);

    \draw[dashed] (1) -- (-5.7,0);
    \draw[dashed] (7) -- (1.7,0);
\end{tikzpicture}
\caption{A pseudo-hyperchain.}
\label{pic-chain-pseudohyper}
\end{subfigure}
\caption{An example of several types of chains with $q = 6$.}
\label{fig-chain}
\end{figure}
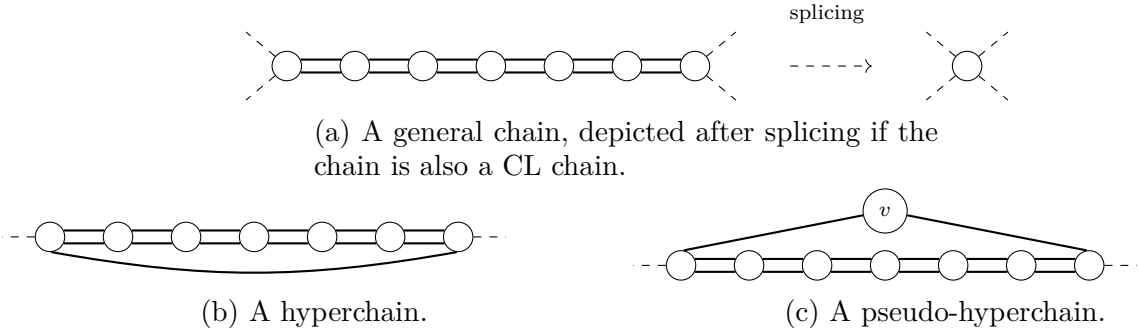

\begin{figure}[t]
\centering
\begin{tikzpicture}[scale = 1]
    \tikzstyle{n node} = [circle, draw = black]
    \tikzstyle{o node} = [inner sep=1,outer sep=0]

    \node[n node] (1a) at (-2,0) {};
    \node[n node] (2a) at (-1,0) {};
    \node[o node] (3a) at (0,0) {$\cdots$};
    \node[n node] (4a) at (1,0) {};
    \node[n node] (5a) at (2,0) {};

    \node[n node] (1b) at (-2,-1) {};
    \node[n node] (2b) at (-1,-1) {};
    \node[o node] (3b) at (0,-1) {$\cdots$};
    \node[n node] (4b) at (1,-1) {};
    \node[n node] (5b) at (2,-1) {};

    \node[n node] (1c) at (-2,-2) {};
    \node[n node] (2c) at (-1,-2) {};
    \node[o node] (3c) at (0,-2) {$\cdots$};
    \node[n node] (4c) at (1,-2) {};
    \node[n node] (5c) at (2,-2) {};

    \node[o node]  at (-2,-1.5) {$\cdots$};
    \node[o node]  at (2,-1.5) {$\cdots$};

    \draw[thick] (1a.25) -- (2a.155);
    \draw[thick] (2a.205) -- (1a.335);
    \draw[thick] (2a.25) -- (3a.165);
    \draw[thick] (3a.195) -- (2a.335);
    \draw[thick] (3a.15) -- (4a.155);
    \draw[thick] (4a.205) -- (3a.345);
    \draw[thick] (4a.25) -- (5a.155);
    \draw[thick] (5a.205) -- (4a.335);

    \draw[thick] (1b.25) -- (2b.155);
    \draw[thick] (2b.205) -- (1b.335);
    \draw[thick] (2b.25) -- (3b.165);
    \draw[thick] (3b.195) -- (2b.335);
    \draw[thick] (3b.15) -- (4b.155);
    \draw[thick] (4b.205) -- (3b.345);
    \draw[thick] (4b.25) -- (5b.155);
    \draw[thick] (5b.205) -- (4b.335);

    \draw[thick] (1c.25) -- (2c.155);
    \draw[thick] (2c.205) -- (1c.335);
    \draw[thick] (2c.25) -- (3c.165);
    \draw[thick] (3c.195) -- (2c.335);
    \draw[thick] (3c.15) -- (4c.155);
    \draw[thick] (4c.205) -- (3c.345);
    \draw[thick] (4c.25) -- (5c.155);
    \draw[thick] (5c.205) -- (4c.335);

    \draw[thick] (1a) -- (1b);
    \draw[thick] (5a) -- (5b);

    \draw[dashed] (1a) -- (-2,.5);
    \draw[dashed] (5a) -- (2,.5);
    \draw[dashed] (1c) -- (-2,-2.5);
    \draw[dashed] (5c) -- (2,-2.5);

    \draw[dashed] (1b) -- (-2,-1.4);
    \draw[dashed] (5b) -- (2,-1.4);
    \draw[dashed] (1c) -- (-2,-1.6);
    \draw[dashed] (5c) -- (2,-1.6);
    
\end{tikzpicture}
\caption{A wide ladder, with 3 rungs shown. Note that each rung may have a different length $q$.}
\label{fig-ladder}
\end{figure}
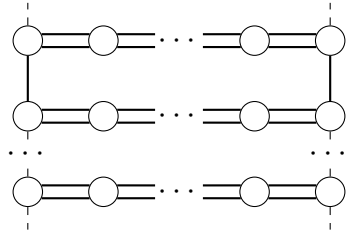

\begin{figure}
\centering
\begin{tikzpicture}[scale = .55,
level distance=1.5cm,
  level 1/.style={sibling distance=1.5cm},
  level 2/.style={sibling distance=1.5cm}]
\tikzstyle{hollow node}=[circle,draw,inner sep=1.6]
\tikzstyle{solid node}=[circle,draw,inner sep=1.6,fill=black]
\tikzset{
red node/.style = {circle,draw=black,fill=red,inner sep=1.6},
blue node/.style= {circle,draw = black, fill= blue,inner sep=1.6}, green node/.style = {circle,draw=black,fill=green,inner sep=1.6}}, 
[edge from parent/.style={draw,dashed}]
\node[solid node, label = left: {\tiny $\mathfrak n_0$}] at (14,.4){}
    child{node[solid node, label = left: {\tiny $\mathfrak n_1$}]{}
        child{node[solid node, label = left: {\tiny $\mathfrak n_2$}]{}
            child[dashed]{node[solid node, label = left: {\tiny $\mathfrak n_{q-1}$}]{}
                child[solid]{node[solid node, label = left: {\tiny $\mathfrak n_{q}$}]{}
                    child{node[solid node, label = left: {\tiny $\mathfrak n_0^{11}$}]{}}
                    child{node[solid node, label={[label distance=-1.5mm] 300:\tiny $\mathfrak n_0^{12}$}]{}}
                    child{node[red node, label={[label distance=-1.5mm] 300:\tiny $\mathfrak p_q$}]{}}
                    }
                child[solid]{node[red node, label={[label distance=-1.5mm] 300:\tiny $\mathfrak m_{q}$}]{}}
                child[solid]{node[hollow node, label={[label distance=-1.5mm] 300:\tiny $\mathfrak p_{q-1}$}]{}}
                }
            child{node[hollow node, label={[label distance=-1.5mm] 300:\tiny $\mathfrak m_3$}]{}}
            child{node[blue node, label={[label distance=-1.5mm] 300:\tiny $\mathfrak p_2$}]{}}
            }
        child{node[blue node, label={[label distance=-1.5mm] 300:\tiny $\mathfrak m_2$}]{}}
        child{node[green node, label={[label distance=-1.5mm] 300:\tiny $\mathfrak p_1$}]{}}
    }
    child{node[green node, label={[label distance=-1.5mm] 300:\tiny $\mathfrak m_1$}]{}}
    child{node[solid node, label={[label distance=-1.5mm] 300:\tiny $\mathfrak n_0^1$}]{}}
;
\draw[dashed] (14,1.2) -- (14,.4);
\draw[dashed] (15.5,-2) -- (15.5,-1.2);
\draw[dashed] (8,-7.8) -- (8,-7);
\draw[dashed] (6.5,-7.8) -- (6.5,-7);

\draw[dashed, ->] (17.4,-3) -- node[above, pos = .45,draw = none]{\tiny{splicing}}(19.6, -3);

\node[solid node, label = left: {\tiny $\mathfrak n_0$}] at (23,.4){}
    child{node[solid node, label = left: {\tiny $\mathfrak n_0^{11}$}]{}}
    child{node[solid node, label={[label distance=-1.5mm] 300:\tiny $\mathfrak n_0^{12}$}]{}}
    child{node[solid node, label={[label distance=-1.5mm] 300:\tiny $\mathfrak n_0^1$}]{}}
;
\draw[dashed] (23,1.2) -- (23,.4);
\draw[dashed] (24.5,-2) -- (24.5,-1.2);
\draw[dashed] (23,-2) -- (23,-1.2);
\draw[dashed] (21.5,-2) -- (21.5,-1.2);
\end{tikzpicture}
\caption{A CL chain (to be spliced) which is also an irregular chain, as in Definition \ref{def-chain-couple}. The white leaves are paired with leaves in the omitted part rather than each other, so that $\mathfrak m_i$ is paired with $\mathfrak p_{i}$ for $i \in \{1, \ldots, q\}$.} 
\label{fig-irregular-chain}
\end{figure}

\begin{definition}\label{def-splicing}
Suppose $\Q$ is a couple with a CL chain $(\mathfrak n_0, \ldots, \mathfrak n_q)$. We can define a new couple $\Q^{\mathrm{sp}}$, and corresponding new molecule $\mathbb{M}^{\mathrm{sp}}$, by removing the nodes and leaves $\mathfrak n_i, \mathfrak m_i, \mathfrak p_i$ for  $(i = 1, \ldots, q)$. The children of $\mathfrak n_0$ become $\mathfrak n_0^1$, $\mathfrak n_0^{11}$, and $\mathfrak n_0^{12}$, with their position determined by sign or by their relative position as children of $\mathfrak n_q$. We call this \textit{splicing} the couple at nodes $\mathfrak n_1, \ldots, \mathfrak n_q$, or at the chain $(\mathfrak n_0, \ldots, \mathfrak n_q)$. 
\end{definition}

\begin{definition}
Suppose $\M(\Q)$ is a molecule containing a double bond with the edges having opposite direction. For a decoration of $\M(\Q)$, denote $k$ and $\ell$ the decoration of the edges in the double bond, and let $h = k-\ell$ be the \textit{gap} of the double bond. We refer to the double bond as a \textit{large gap (LG) double bond} if $|h| \gtrsim T^{-\frac{1}{2}}$ and a \textit{small gap (SG) double bond} if $|h| \lesssim T^{-\frac{1}{2}}$. By Proposition 8.3 of \cite{WKE}, all gaps in a chain are identical, so we may also refer to a \textit{SG or LG negative chain}, or negative chain-like object. Similarly, an irregular chain $(\mathfrak n_0, \ldots, \mathfrak n_q)$ of $\Q$ can be SG or LG, with the gap $h := k_{\mathfrak n_0} - k_{\mathfrak n_0^1} = k_{\mathfrak n_j} - k_{\mathfrak m_j}$ ($1 \leq j \leq q$).
\end{definition}

\begin{proposition}\label{prop-CN}
Given a molecule $\M(\mathcal Q)$ and chain $(v_0, \ldots, v_{q})$, there can be at most one CN double bond. 
\end{proposition}

\begin{figure}
\begin{tikzpicture}
    \tikzstyle{every node} = [circle, draw = black, inner sep = 2.3, scale = .9]
    \node (1) at (-5,0) {\tiny $v_0$};
    \node (2) at (-3.5,0) {\tiny $v_1$};
    \node (3) at (-2,0) {\tiny $v_2$};
    \node (4) at (-.5,0) {\tiny $v_3$};
    \node (5) at (1,0) {\tiny $v_{q-1}$};
    \node (6) at (2.5,0) {\tiny $v_{q}$};

    \draw[thick] (1.25) -- (2.155) 
        node[above, pos = .5,draw = none]{\tiny{LP}};
    \draw[thick] (2.205) -- (1.335)
        node[below, pos = .5,draw = none]{\tiny{LP}};
    \draw[thick] (2.25) -- node[above, draw = none]{\tiny{LP}}(3.155);
    \draw[thick] (3.205) -- (2.335)
        node[below, pos = .1,draw = none]{\tiny{P}}
        node[below, pos = .8,draw = none]{\tiny{C}};
    \draw[thick] (3.25) -- node[above, draw = none]{\tiny{LP}}(4.155);
    \draw[thick] (4.205) -- (3.335)
        node[below, pos = .1,draw = none]{\tiny{P}}
        node[below, pos = .8,draw = none]{\tiny{C}};
    \draw[dashed] (4.25) -- (5.163);
    \draw[dashed] (5.197) -- (4.335);
    \draw[thick] (5.17) -- node[above, draw = none]{\tiny{LP}}(6.155);
    \draw[thick] (6.205) -- (5.343)
        node[below, pos = .1,draw = none]{\tiny{P}}
        node[below, pos = .8,draw = none]{\tiny{C}};

    \draw[dashed] (1) -- (-5.6,.5);
    \draw[dashed] (1) -- (-5.6,-.5);

    \draw[dashed] (6) -- (3.1,.5);
    \draw[dashed] (6) -- (3.1,-.5);
\end{tikzpicture}
\caption{Labelling of a chain with a CN double bond.}
\label{fig-CN}
\end{figure}

\begin{proof}
Suppose WLOG that $v_0$ and $v_1$ are connected via a CN double bond. Then, we may attempt to label the molecule starting at $v_0$. For each double bond in the chain, at least one edge is an LP edge, and both edges connecting $v_0$ and $v_1$ are LP edges. So, $v_1$ has all 3 children accounted for. If a vertex $v_i$ has all 3 children accounted for, the remaining edge must be a PC edge with $v_{i+1}$ the parent, leaving $v_{i+1}$ with all its children accounted for. So, $v_{i+1}$ is the parent of $v_{i}$ for $i \in \{1, \ldots, q-1\}$, making all other bonds in the chain CL double bonds. See Figure \ref{fig-CN}.
\end{proof}

\subsection{Cancellation of an irregular chain} \label{subsec-cancellation}

\begin{definition}\label{def-twist} Given a couple $\Q$, we call node $\mathfrak n_2 \in \mathcal N$ \textit{admissible} if it has parent $\mathfrak n_1$ such that the corresponding atoms $v_1$ and $v_2$ in $\M(\Q)$ are connected via a CL double bond. We additionally call $\mathfrak n_2$ \textit{twist-admissible} if the two bonds have opposite directions. Denote the paired children of $\mathfrak n_1$ and $\mathfrak n_2$ by $\mathfrak n_{12}$ and $\mathfrak n_{21}$. 

For $\mathfrak n_2$ twist-admissible, a \textit{unit twist} of $\mathcal Q$ at node $\mathfrak n_2$ is performed, as shown in Figure \ref{pic-twist}, by swapping $\mathfrak n_{12}$ with $\mathfrak n_2$ and swapping the two children (and the trees below them) of $\mathfrak n_2$ which are not $\mathfrak n_{21}$, so that all other parent-child relationships in the couple and their signs remain the same. When there may be ambiguity between a couple and its unit twist, we use $\tilde{\mathfrak m}$ to refer to a node or leaf in the unit twist coming from node or leaf $\mathfrak m$ in $\mathcal Q$. 
\end{definition}

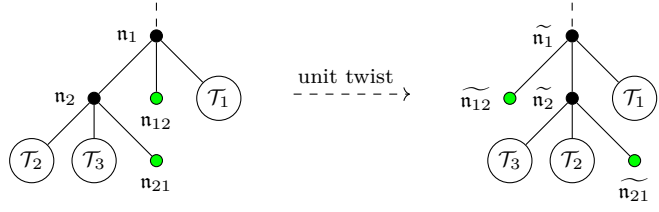
\begin{figure}[t!]
        \centering
\begin{tikzpicture}[scale = .55]
\tikzstyle{hollow node}=[circle,draw,inner sep=2.0]
\tikzstyle{solid node}=[circle,draw,inner sep=1.6,fill=black]
\tikzset{
red node/.style = {circle,draw=black,fill=red,inner sep=1.6},
blue node/.style= {circle,draw = black, fill= blue,inner sep=1.6}, green node/.style = {circle,draw=black,fill=green,inner sep=1.6}}
\node[solid node, label = left: {\tiny $\mathfrak n_1$}] at (14,.4){}
    child{node[solid node, label = left: {\tiny $\mathfrak n_2$}]{}
        child{node[hollow node]{\tiny{$\mathcal T_2$}}}
        child{node[hollow node]{\tiny{$\mathcal T_3$}}}
        child{node[green node, label = below:{\tiny $\mathfrak n_{21}$}]{}}
    }
    child{node[green node, label = below:{\tiny $\mathfrak n_{12}$}]{}}
    child{node[hollow node]{\tiny{$\mathcal T_1$}}}
;
\draw[dashed] (14,1.2) -- (14,.4);

\draw[dashed, ->] (17.3,-1) -- node[above, pos = .45,draw = none]{\tiny{unit twist}}(20.1, -1);

\node[solid node, label = left: {\tiny $\widetilde{\mathfrak n_1}$}] at (24,.4){}
    child{node[green node, label = left:{\tiny $\widetilde{\mathfrak n_{12}}$}]{}}
    child{node[solid node, label = left: {\tiny $\widetilde{\mathfrak n_2}$}]{}
        child{node[hollow node]{\tiny{$\mathcal T_3$}}}
        child{node[hollow node]{\tiny{$\mathcal T_2$}}}
        child{node[green node, label = below:{\tiny $\widetilde{\mathfrak n_{21}}$}]{}}
    }
    child{node[hollow node]{\tiny{$\mathcal T_1$}}}
;
\draw[dashed] (24,1.2) -- (24,.4);

\end{tikzpicture}
    \caption{Unit twist of a couple at node $\mathfrak n_2$.}
    \label{pic-twist}
\end{figure}

\begin{remark}\label{twist-dec}
If $\mathcal Q$ is a decorated couple and we perform a unit twist on it at node $\mathfrak n_2$ to obtain a couple $\mathcal Q'$ we obtain a corresponding decoration of $\mathcal Q'$ by leaving all values of ${k_{\tilde{\mathfrak m}}}$ the same other than ${k_{\widetilde{\mathfrak n_2}}}, {k_{\widetilde{\mathfrak n_{12}}}}$, and ${k_{\widetilde{\mathfrak n_{21}}}}$. We set ${k_{\widetilde{\mathfrak n_2}}} = k_{\mathfrak n_{12}} = k_{\mathfrak n_{21}}$ and ${k_{\widetilde{\mathfrak n_{12}}}} = {k_{\widetilde{\mathfrak n_{21}}}} = k_{\mathfrak n_2}$. As a result, if we splice both $\mathcal Q$ or $\mathcal Q'$ at $\mathfrak n_2$, we obtain the same couple and corresponding molecule for each with the same decoration. See Figure \ref{pic-twistdec}.
\end{remark}

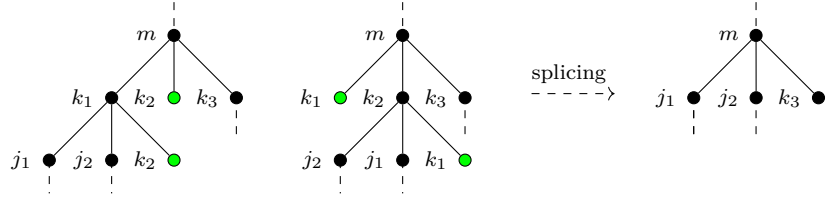
\begin{figure}[t!]
\centering
\begin{tikzpicture}[scale = .55]
\tikzstyle{hollow node}=[circle,draw,inner sep=2.0]
\tikzstyle{solid node}=[circle,draw,inner sep=1.6,fill=black]
\tikzset{
red node/.style = {circle,draw=black,fill=red,inner sep=1.6},
blue node/.style= {circle,draw = black, fill= blue,inner sep=1.6}, green node/.style = {circle,draw=black,fill=green,inner sep=1.6}}
\node[solid node, label = left: {\tiny $m$}] at (11.5,.4){}
    child{node[solid node, label = left: {\tiny $k_1$}]{}
        child{node[solid node, label = left: {\tiny $j_1$}]{}}
        child{node[solid node, label = left: {\tiny $j_2$}]{}}
        child{node[green node, label = left:{\tiny $k_2$}]{}}
    }
    child{node[green node, label = left:{\tiny $k_2$}]{}}
    child{node[solid node, label = left: {\tiny $k_3$}]{}}
;
\draw[dashed] (11.5,1.2) -- (11.5,.4);
\draw[dashed] (8.5,-2.6) -- (8.5,-3.4);
\draw[dashed] (10,-2.6) -- (10,-3.4);
\draw[dashed] (13,-1.4) -- (13,-2);

\node[solid node, label = left: {\tiny $m$}] at (17,.4){}
    child{node[green node, label = left: {\tiny $k_1$}]{}}
    child{node[solid node, label = left: {\tiny $k_2$}]{}
        child{node[solid node, label = left: {\tiny $j_2$}]{}}
        child{node[solid node, label = left: {\tiny $j_1$}]{}}
        child{node[green node, label = left: {\tiny $k_1$}]{}}
    }
    child{node[solid node, label = left: {\tiny $k_3$}]{}}
;
\draw[dashed] (17,1.2) -- (17,.4);
\draw[dashed] (17,-2.6) -- (17,-3.4);
\draw[dashed] (15.5,-2.6) -- (15.5,-3.4);
\draw[dashed] (18.5,-1.4) -- (18.5,-2);

\draw[dashed, ->] (20.1,-1) -- node[above, pos = .45,draw = none]{\tiny{splicing}} (22.1,-1);

\node[solid node, label = left: {\tiny $m$}] at (25.5,.4){}
    child{node[solid node, label = left: {\tiny $j_1$}]{}}
    child{node[solid node, label = left: {\tiny $j_2$}]{}}
    child{node[solid node, label = left: {\tiny $k_3$}]{}}
;
\draw[dashed] (25.5,1.2) -- (25.5,.4);
\draw[dashed] (25.5,-1.4) -- (25.5,-2);
\draw[dashed] (24,-1.4) -- (24,-2);
\draw[dashed] (24,-1.4) -- (24,-2);

\end{tikzpicture}
    \caption{Decorations of the couples from Figure \ref{pic-twist} and the decoration of the couple upon splicing them both at $\mathfrak n_2$ and $\widetilde{\mathfrak n_2}$.}
    \label{pic-twistdec}
\end{figure}

\begin{definition} \label{def-congruent}
We say couples $\mathcal Q_1$ and $\mathcal Q_2$ are \textit{congruent} ($\mathcal Q_1 \sim \mathcal Q_2$) if any number of unit twists of $\mathcal Q_1$ yields $\mathcal Q_2$ and vice versa. Via Remark \ref{twist-dec}, the decorations of $\mathcal Q_1$ and $\mathcal Q_2$ are in one-to-one correspondence. 
\end{definition}

\begin{remark}
The unit twist operation is commutative, so we may specify how two couples are congruent by specifying a set of nodes at which a unit twist was performed. For a couple $\mathcal Q$ and set $\mathcal M \subset \mathcal N$ of twist-admissible nodes, denote by $\textgoth{Q}_{\mathcal M}$: 
\begin{equation*}
\mathcal {\textgoth{Q}_\mathcal M} = \{\mathcal Q' | \mathcal Q' \sim \mathcal Q \text{ via } \mathcal M' \in 2^{\mathcal M}\}.
\end{equation*}
Similarly, set 
\begin{equation*}
(\K_{\textgoth{Q}_{\mathcal M}}) (t,s,k) = \sum_{\Q' \in \textgoth{Q}_\mathcal M}\mathcal{K_{Q'}}(t,s,k).
\end{equation*}
\end{remark}

\begin{lemma}\label{lem-splice}
For a couple $\mathcal Q$ and SG irregular chain $(\mathfrak n_0, \ldots, \mathfrak n_{q})$ of length $q$, 
\begin{align*}
\K_{\textgoth{Q}_\mathcal M} (t,s,k) &= \left( \frac{\alpha T}{L}\right)^{n_{\mathrm{sp}}} \zeta(\mathcal Q^{\mathrm{sp}}) \int _0^1 \diff \varsigma\sum_{\mathscr E^{\mathrm{sp}}} \tilde \epsilon_{\mathscr E^{\mathrm{sp}}} \int_{\mathcal{E}^{\mathrm{sp}}_{\varsigma}} \left( \prod_{\mathfrak n \in \mathcal N^{\mathrm{sp}}}e^{\zeta_{\mathfrak n} 2 \pi i Tt_{\mathfrak n}\Omega_{\mathfrak n}} \diff t_{\mathfrak n} \right) \\
& \hspace{5cm}\times P_q(\varsigma, k[\mathcal N^{\mathrm{sp}}]) \prod^+_{\mathfrak l \in \mathcal L^{\mathrm{sp}}} n_{\mathrm{in}}(k_{\mathfrak l}),
\end{align*}
where $P_q$ satisfies \begin{equation}\label{eq-pestimate}
\sup_{|k_{\mathfrak n_0} - k_{\mathfrak n^1_0}| \leq T^{-1/2}} \|P_q(\varsigma, k[\mathcal N^{\mathrm{sp}}])\|_{\ell^\infty} \lesssim \left(\alpha T^{1/2}\right)^q,
\end{equation} and 
\begin{equation}\label{eq-sig-setE}
\mathcal E^{\mathrm{sp}}_{\varsigma} = \mathcal E^{\mathrm{sp}} \cap \{ t_{\mathfrak n_0} \geq t_{\mathfrak n_0^{11}} + \varsigma, t_{\mathfrak n_0} \geq t_{\mathfrak n_0^{12}} + \varsigma \}.
\end{equation}
In the above, $\mathcal Q^{\mathrm{sp}}$ is the couple $\mathcal Q$ with the irregular chain spliced out, $k[\mathcal N^{\mathrm{sp}}]$ are the decorations of the nodes in the spliced couple, and $\mathcal M = \{\mathfrak n_1, \ldots, \mathfrak n_{q}\}$. Additionally, $\tilde \epsilon_{\mathscr E^{\mathrm{sp}}}$ allows for degeneracies at $\mathfrak n_0$. 
\end{lemma}

\begin{proof}
For simplicity throughout, let us use the notation of Figure \ref{fig-irregular-chain} and assume that $\Q$ is the element of $\textgoth{Q}_{\mathcal M}$ where all $\zeta_{\mathfrak n_j} (j = 0, \ldots, q)$ have the same sign. We fix $k_{\mathfrak n}$ and $t_{\mathfrak n}$ for $\mathfrak n \notin \{\mathfrak n_i, \mathfrak p_i, \mathfrak m_i\} (1 \leq i \leq q)$ and  sum and integrate in the remaining values, setting $k_{\mathfrak n_i} = k_i$, $k_{\mathfrak p_i} = k_{\mathfrak m_i} = \ell_i$, and $t_i = t_{\mathfrak n_i} (0 \leq i \leq q)$. For convenience, we set $t_{q+1} = \text{max}(t_{\mathfrak n_0^{11}}, t_{\mathfrak n_0^{12}})$ in addition to $k_{q+1} = k_{\mathfrak n_0^{11}},$ and $\ell_{q+1} = k_{\mathfrak n_0^{12}}$. Throughout, we set (for $0 \leq j \leq q$)
\begin{align*}
\Omega_j = |k_{j+1}|^\sigma - |\ell_{j+1}|^\sigma + |\ell_{j}|^\sigma - |k_{j}|^\sigma
\end{align*}
and 
\begin{align*}
\Omega = \sum_{j = 0}^q \Omega_{j} = |k_{q+1}|^\sigma - |\ell_{q+1}|^\sigma + |\ell_0|^\sigma - |k_0|^\sigma.
\end{align*}
Also, note that $k_0, \ell_0, k_{q+1}, \ell_{q+1}$ are fixed with $$k_0 - \ell_0 = k_1 - \ell_1 = \ldots = k_q - \ell_q = k_{q+1} - \ell_{q+1} := h.$$ Each $\mathcal Q' \in \textgoth{Q}_{\mathcal M}$ will have a similar corresponding expression for $\mathcal K_{Q'}$, with the only difference occurring in the variables which are not fixed. This part we isolate as 
\begin{align*}
& \left( \frac{\alpha T}{L}\right)^q\sum\limits_{\substack{k_i, \ell_i \in \Z_L \\ k_{i} - \ell_{i} + \ell_{i-1} = k_{i-1} \\ i = 1, \ldots q}} \int\limits_{t_{q+1} < t_q < \ldots < t_1 < t_0} \left(\prod_{j = 1}^q i\zeta_{\mathfrak n_j}\right) \left(\prod_{j = 0}^q\epsilon_{k_{j+1}\ell_{j+1}m_{j}}\right)  \nonumber \\
& \hspace{4.5cm}\times \left( \prod_{j = 0}^q e^{\zeta_{\mathfrak n_j} 2 \pi i T t_j \Omega_j}\right)\left( \prod_{j = 1}^q n_{\mathrm{in}}(m_j)\right) \mathrm d t_1 \ldots \mathrm d t_q,
\end{align*}
where $m_j = \ell_j$ if no unit twist has been performed at $\mathfrak n_j$ to get from $\mathcal Q$ to $\mathcal Q'$ and $m_j = k_j$ otherwise. We may assume that  $\epsilon_{k_{j+1}\ell_{j+1}m_j}= 1$ as $h = 0$ admits much better estimates. A unit twist at $\mathfrak n_j$ changes the sign of $\zeta_{\mathfrak n_j}$ 
but not the expression for $\zeta_{\mathfrak n_j}\Omega_j$, so summing over $\mathcal Q' \in \textgoth Q_{\mathcal M}$ amounts to summing over all choices of $\zeta_{\mathfrak n_j}$ for $j = 1, \ldots, q$, yielding the expression 
\begin{align*}
& \hspace{.4cm}\left( \frac{\alpha T}{L}\right)^q\sum\limits_{\substack{\substack{k_i \in \Z_L, \\ i = 1, \ldots q}}} \hspace{.2cm}\int\limits_{t_{q+1} < t_q < \ldots < t_1 < t_0}  \left( \prod_{j = 0}^q e^{2 \pi i T t_j \Omega_j}\right)\left( \prod_{j = 1}^q n_{\mathrm{in}}(k_j - h) - n_{\mathrm{in}}(k_j)\right) \mathrm d t_1 \ldots \diff t_q\\
&= e^{2\pi i T t_0 \Omega} \sum\limits_{\substack{\substack{k_i \in \Z_L, \\ i = 1, \ldots q}}} \hspace{.2cm}\int\limits_{t_{q+1} < t_q < \ldots < t_1 < t_0} \left( \frac{\alpha T}{L}\right)^q \left( \prod_{j = 1}^q e^{2\pi i T(t_j - t_0)\Omega_j} ( n_{\mathrm{in}}(k_j - h) - n_{\mathrm{in}}(k_j))\right) \mathrm d t_1 \ldots \diff t_q \\
&= e^{2\pi i T t_0 \Omega} \int\limits_{0 < s_1 < \ldots < s_q < t_0 - t_{q+1}}  \left( \frac{\alpha T}{L}\right)^q  \sum\limits_{\substack{\substack{k_i \in \Z_L, \\ i = 1, \ldots q}}} \left( \prod_{j = 1}^q e^{-2\pi i Ts_j \Omega_j}\left(n_{\mathrm{in}}(k_j - h) - n_{\mathrm{in}}(k_j)\right) \diff s_j\right) \\
& = e^{2\pi i T t_0 \Omega} \int\limits_{\varsigma \in \left[0, t_0 - t_{q+1}\right]} P_q(\varsigma, k[\mathcal N^{\mathrm{sp}}]) \diff\varsigma,
\end{align*}
where 
\begin{align}
P_q(\varsigma, k[\mathcal N^{\mathrm{sp}}]) &:= \left( \frac{\alpha T}{L}\right)^q \sum\limits_{\substack{\substack{k_i \in \Z_L, \\ i = 1, \ldots q}}} \hspace{.2cm} \int\limits_{0 < s_1 < \ldots < s_{q-1} < \varsigma}   \left( \prod_{j = 1}^q e^{-2\pi i Ts_j \Omega_j}\right)  \\
&\hspace{3cm}\times \prod_{j = 1}^q\left(n_{\mathrm{in}}(k_j - h) - n_{\mathrm{in}}(k_j)\right)\diff s_1 \ldots \diff s_{q-1}. \nonumber
\end{align}
To analyze $P_q(\varsigma, k[\mathcal N^{\mathrm{sp}}])$, note that for any $\theta > 0$,
\begin{align*}
\left(\frac{\alpha T}{L}\right) \sum_{k \in \Z_L} \left| \left(n_{\mathrm{in}}(k - h) - n_{\mathrm{in}}(k) \right)\right| &\lesssim L^\theta \alpha T h, 
\end{align*}
using the decay of $n_{\mathrm{in}}$, establishing (\ref{eq-pestimate}). 

Returning to the expression of $\K_{\textgoth Q_{\mathcal M}}(t,s,k)$ (now no values of $k_{\mathfrak n}$ or $t_{\mathfrak n}$ are fixed), we will write it in terms of the spliced couple $\mathcal Q^{\mathrm{sp}}$ and its corresponding decorations ($\mathscr E^{\mathrm{sp}}$), domain of integration $\mathcal E^{\mathrm{sp}}$ and its nodes and leaves ($\mathcal N^{\mathrm{sp}}$ and $\mathcal L^{\mathrm{sp}}$): 
\begin{align*}
\K_{\textgoth{Q}_\mathcal M} (t,s,k) &= \left( \frac{\alpha T}{L}\right)^{n_{\mathrm{sp}}} \zeta(\mathcal Q^{\mathrm{sp}}) \sum_{\mathscr E^{\mathrm{sp}}} \tilde\epsilon_{\mathscr E^{\mathrm{sp}}} \int_{\mathcal{E}^{\mathrm{sp}}} \left( \prod_{\mathfrak n \in \mathcal N^{\mathrm{sp}}}e^{\zeta_{\mathfrak n} 2 \pi i Tt_{\mathfrak n}\Omega_{\mathfrak n}} \diff t_{\mathfrak n} \right) \\
& \hspace{3cm} \times \int\limits_{0 \leq \varsigma \leq (t_{\mathfrak n_0} - \max(t_{\mathfrak n_0^{11}}, t_{\mathfrak n_0^{12}}))} P_q(\varsigma, k[\mathcal N^{\mathrm{sp}}]) \diff \varsigma \\
&=\left( \frac{\alpha T}{L}\right)^{n_{\mathrm{sp}}} \zeta(\mathcal Q^{\mathrm{sp}})  \int _0^1 \diff \varsigma \sum_{\mathscr E^{\mathrm{sp}}} \tilde \epsilon_{\mathscr E^{\mathrm{sp}}} \int_{\mathcal{E}^{\mathrm{sp}}_\sigma} \left( \prod_{\mathfrak n \in \mathcal N^{\mathrm{sp}}}e^{\zeta_{\mathfrak n} 2 \pi i Tt_{\mathfrak n}\Omega_{\mathfrak n}} \diff t_{\mathfrak n} \right) \\
& \hspace{5cm}\times P_q(\varsigma, k[\mathcal N^{\mathrm{sp}}]) \prod^+_{\mathfrak l \in \mathcal L^{\mathrm{sp}}} n_{\mathrm{in}}(k_{\mathfrak l}),
\end{align*}
for $\mathcal{E}^{\mathrm{sp}}_\varsigma$ the same as $\mathcal E^{\mathrm{sp}}$ except now also $t_{\mathfrak n_0} \geq t_{\mathfrak n_0^{11}} + \varsigma, t_{\mathfrak n_0^{12}} + \varsigma$. Note also that when we splice, it may be the case that $\{k_0, \ell_{q_1}\} = \{k_{q+1}, \ell_0\}$, so we modify $\epsilon_{\mathscr E^{\mathrm{sp}}}$ so that $\epsilon_{k_{q+1} \ell_{q+1} \ell_0}$ is also 1 if $\ell_0 = \ell_{q+1}$. 
\end{proof}
\subsection{Splicing specifications}
Now that Lemma \ref{lem-splice} allows us to splice out one SG irregular chain from a couple $\Q$, we generalize this to multiple irregular chains. We also specify a set $\mathcal M$ of twist-admissible nodes (and corresponding irregular chains) we intend to splice, and write the expression for $\K_{\textgoth{Q}_{\mathcal M}}$. 

\begin{lemma}\label{lem-chains}
For any molecule $\M$, there is a unique collection $\mathscr C$ of disjoint atomic groups, such that each atomic group in $\mathscr C$ is a negative chain, negative hyperchain, or negative pseudo-hyperchain and any negative chain $\C$ of $\M$ is a subset of precisely one atomic group in $\mathscr C$.

Furthermore, for any subset $\mathscr C_1 \subset \mathscr C$, there is a unique collection $\mathscr D_{\mathscr C_1}$ of disjoint atomic groups, such that each atomic group in $\mathscr D_{\mathscr C_1}$ is a maximal negative wide ladder of $\mathscr C_1$ and each chain $\C \in \mathscr C_1$ is a subset of precisely one atomic group in $\mathscr D_{\mathscr C_1}$.
\end{lemma}
\begin{proof}
Consider all maximal negative chains in $\M$. If they are also hyperchains or pseudo-hyperchains, then we consider them as such. This collection is $\mathscr C$. Suppose that $\C_1 \neq \C_2 \in \mathscr C$ and their intersection is nontrivial. If at least one is a hyperchain or pseudo-hyperchain, $ \mathcal D = \C_1  \cup \C_2$ cannot be a chain, hyperchain, or pseudo-hyperchain. If each are chains, then $\C_1$ is connected to an atom in $\C_2$ via double bond, and is thus not maximal.

Now let us consider $\mathscr D_{\mathscr C_1}$ to be the set of all maximal negative wide ladders in $\mathscr C_1$. Note that a maximal wide ladder may contain a single rung. Suppose that $\mathcal L_1 \neq \mathcal L_2 \in \mathscr D_{\mathscr C_1}$ have nontrivial intersection. Therefore, there is a chain $\mathcal C =  (v_0, \ldots, v_{q}) \in \mathcal L_1, \mathcal L_2$ such that $v_0, v_q$ are connected to a chain in $\mathcal L_2$ and $\mathcal L_2$, so that neither is a maximal negative wide ladder of $\mathscr C_1$. Similarly any $\mathcal C \in \mathscr C_1$ is maximal and therefore  part of a single maximal wide ladder. 
\end{proof}

\begin{definition}\label{def-M}
Consider a couple $\mathcal Q$ and let $\mathscr C$ be defined as in Lemma \ref{lem-chains} for $\M(\Q)$. Suppose also that we have chosen $\mathscr C_1 \subset \mathscr C$ of all SG negative chain-like objects. We define the set $\mathcal M_{\mathscr C_1}$ below. Consider the set $\mathscr D_{\mathscr C_1}$ defined in Lemma \ref{lem-chains}. If $\mathcal L \in \mathscr D_{\mathscr C_1}$ consists of a single chain $\C = (v_0, \ldots, v_q) \in \mathscr C_1$:
\begin{enumerate}
\item If $\C$  is a chain, we include in $\mathcal M$ all $\mathfrak n(v_i)$ which are admissible.
\item If $\C$ is a hyperchain or pseudo-hyperchain with a CN double bond, we include in $\mathcal M$ all $\mathfrak n(v_i)$ which are admissible. If there is no CN double bond, we exclude one admissible $\mathfrak n(v_i)$. 
\end{enumerate}

Otherwise if $\mathcal L = \{\mathcal C_1, \ldots, \mathcal C_m\}$ with chain $j$ having length $q^{(j)}$, perform the following:
\begin{enumerate}[(a)]
\item Starting with $\mathcal C_1$ determine the nodes to be added to $\mathcal M$ by performing the above, considering it as a hyperchain or pseudo-hyperchain if it is.
\item For $\mathcal C_{j + 1}$ and $j \leq m - 2$, if there are $q^{(j)}$ nodes of $\mathcal C_j$ contained in $\mathcal M$, treat $\mathcal C_{j + 1}$ as to point (2) above in determining which nodes to add to $\mathcal M$. Otherwise, treat it according to (1). 
\item For $\mathcal C_{m}$, if it is a hyperchain or pseudo-hyperchain or there are $q^{(m-1)}$ nodes of $\mathcal C_{m-1}$ in $\mathcal M$, treat $\mathcal C_m$ according to (2). Otherwise, treat it according to (1). 
\end{enumerate}
\end{definition}

\begin{proposition}\label{prop-stage1}
For a couple $\mathcal Q$, and choice of SG negative chain-like objects $\mathscr C_1$ in $\M(\Q)$, consider the resulting couple $\mathcal Q^{\mathrm{sp}}$ obtained by splicing $\mathcal Q$ at the nodes in $\mathcal M = \mathcal M_{\mathscr C_1}$. Compared to $\mathbb M = \mathbb M(\mathcal Q)$, $\mathbb M^{\mathrm{sp}} = \mathbb M(\mathcal Q^{\mathrm{sp}})$ has the following properties: 
\begin{enumerate}[(a)]
\item The molecule $\mathbb M^{\mathrm{sp}}$ is connected and has either 2 atoms of degree 3 or one atom of degree 2, with the rest having degree 4.  
\item For elements of $\mathscr C_1$, each SG chain is reduced to a single atom or a single CN double bond in $\mathbb M^{\mathrm{sp}}$. Similarly, each SG hyperchain is reduced to a triple bond and each SG pseudo-hyperchain is reduced to a single double bond (either CL or CN) pseudo-hyperchain. 
\item The only SG CL double bonds in $\mathbb M^{\mathrm{sp}}$ are part of pseudo-hyperchains. 
\item Any degenerate atom of $\mathbb M^{\mathrm{sp}}$ with self-connections is fully degenerate. 
\end{enumerate}
Also, 
\begin{align} \label{eq-splice-exp}
\begin{split}
\K_{\textgoth{Q}_\mathcal M} (t,s,k) &= \left( \frac{\alpha T}{L}\right)^{n_{\mathrm{sp}}} \zeta(\mathcal Q^{\mathrm{sp}}) \int _{\left[ 0,1\right]^{\mathcal N_0^{\mathrm{sp}}}} \mathrm d \pmb{\varsigma} \sum_{\mathscr E^{\mathrm{sp}}} \tilde \epsilon_{\mathscr E^{\mathrm{sp}}} \int_{\mathcal{E}^{\mathrm{sp}}_{\pmb{\varsigma}}} \left( \prod_{\mathfrak n \in \mathcal N^{\mathrm{sp}}}e^{\zeta_{\mathfrak n} 2 \pi i Tt_{\mathfrak n}\Omega_{\mathfrak n}} \diff t_{\mathfrak n} \right)  \\
& \hspace{5cm}\times \prod_{\mathfrak n \in \mathcal N^{\mathrm{sp}}_0} P_{q_{\mathfrak n}}(\varsigma_{\mathfrak n}, k[\mathcal N^{\mathrm{sp}}]) \prod^+_{\mathfrak l \in \mathcal L^{\mathrm{sp}}} n_{\mathrm{in}}(k_{\mathfrak l}),
\end{split}
\end{align}
where $\mathcal N_0^{\mathrm{sp}}$ are the nodes at which an irregular chain was spliced out and $P_{q_{\mathfrak n}}$ is given in Lemma \ref{lem-splice} and $q_{\mathfrak n}$ is the length or irregular chain spliced out below $\mathfrak n$. Additionally, $\mathcal{E}^{\mathrm{sp}}_{\pmb{\varsigma}} = \mathcal E \cap \{t_{\mathfrak n} \geq t_{\mathfrak n^{11}} + \varsigma_{\mathfrak n}, t_{\mathfrak n^{12}} + \varsigma_{\mathfrak n}\}_{\mathfrak n \in \mathcal N_0^{\mathrm{sp}}}$. 
\end{proposition}
\begin{proof}
First, (a) directly follows from the fact that splicing retains the couple with a tree structure. Also, (b) and (c) follow from Definition \ref{def-M} using the fact that splicing a rung of a wide ladder can cause another chain to become a pseudo-hyperchain. To see (d), note that Lemma \ref{lem-splice} alters $\epsilon_{\mathscr E^{\mathrm{sp}}}$ to allow degeneracies at atoms in $\mathcal N_0^{\mathrm{sp}}$. However, an atom $\mathfrak n \in \mathcal N_0^{\mathrm{sp}}$ may not have a self-connection in $\M$ as two of its children are involved in an irregular chain. So, a self-connection at $\mathfrak n$ would need to be created as part of splicing. This would imply that the last node in the irregular chain has an additional child paired to a child of $\mathfrak n$ and thus that it was a hyperchain in $\M(\Q)$. However, we do not splice out all double bonds in a hyperchain and leave one creating a triple bond. Finally, the nodes in $\mathcal M$ form groups of disjoint irregular chains with each element of $\mathscr C_1$ corresponding to at most 2 disjoint irregular chains. As these do not effect each other, we may apply Lemma \ref{lem-splice} to obtain (\ref{eq-splice-exp}). 
\end{proof}

\section{Reduction to Stage 2}
\label{sec-red}
Here, we reduce Proposition \ref{prop-couples} to the following:  

\begin{proposition}\label{prop-rigidity}
Consider a labelled molecule $\mathbb M = \mathbb M(\mathcal Q)$ of order $n\leq N^3$. Suppose we fix $k \in \Z_L$ as well as $k_\ell^0 \in \Z_L$ and $\beta_v \in \R$ for each bond $\ell$ and atom $v$ of $\mathbb M$ such that if a bond $\ell$ is LP, $|k_\ell^0| \leq D$, and for each $v$, there are at least two $\ell \sim v$ such that $|k_\ell^0| \leq D$. Consider all $k$-decorations $(k_\ell)$ of $\M$ such that 
\begin{enumerate}[(i)]
\item The $k$-decoration $(k_\ell)$ is inherited from a $k$-decoration $\mathscr E$of $\mathcal Q$ that satisfies the SG and LG assumptions in Proposition \ref{prop-stage1} as well as non-degeneracy conditions $\tilde \epsilon_{\mathscr E}.$
\item The decoration is restricted by $(\beta_v)$ and $(k_{\ell}^0)$ in the sense that $|\Omega_v - \beta_v| \leq T^{-1}$ and $|k_\ell - k_\ell^0| \leq 1$. 
\end{enumerate}
Then, the number $\mathfrak D$ of such $k$-decorations is bounded by 
\begin{equation}\label{eq-counting}
\mathfrak D \leq C^n L^{n(1 + \theta)}T^{-\frac{n}{5} - \frac{3}{5}}D^{n(2 - \sigma)}.
\end{equation}
\end{proposition}

We first record the following Lemma, proved in Lemma 10.2 of \cite{WKE}.

\begin{lemma}\label{lem-osc-int}
Let 
\begin{equation}
B(t,s,\lambda[\mathcal N^{\mathrm{sp}}], \pmb \varsigma):= \int_{\mathcal{E}^{\mathrm{sp}}_{\pmb{\varsigma}}} \prod_{\mathfrak n \in \mathcal N^{\mathrm{sp}}}e^{\zeta_{\mathfrak n} 2 \pi i t_{\mathfrak n}\lambda_{\mathfrak n}} \diff t_{\mathfrak n},
\end{equation}
where $\mathcal{E}^{\mathrm{sp}}_{\pmb{\varsigma}}$ is defined in Proposition \ref{prop-stage1} and $\lambda[\mathcal N^{\mathrm{sp}}] = \{\lambda_{\mathfrak n}\}_{\mathfrak n \in \mathcal N^{\mathrm{sp}}}$. Then, 
\begin{align} \label{eq-osc-int}
\left|B(t,s,\lambda[\mathcal N^{\mathrm{sp}}], \pmb \varsigma) \right|\leq C^{n_{\mathrm{sp}}} \sum_{d_\mathfrak n} \prod_{\mathfrak n \in \mathcal N^{\mathrm{sp}}} \frac{1}{\langle q_{\mathfrak n}\rangle}, 
\end{align}
for $d_{\mathfrak n} \in \{0,1\}$ and $q_{\mathfrak n}$ defined inductively as $q_{\mathfrak n} = \zeta_{\mathfrak n}\lambda_{\mathfrak n} + \zeta_{\mathfrak n_1}d_{\mathfrak n_1}q_{\mathfrak n_1} + \zeta_{\mathfrak n_2}d_{\mathfrak n_2}q_{\mathfrak n_2} + \zeta_{\mathfrak n_3}d_{\mathfrak n_3}q_{\mathfrak n_3}$, where $\mathfrak n_1, \mathfrak n_2, \mathfrak n_3$ are the three children of $\mathfrak n$ and $q_{\mathfrak l} = 0$ for leaf $\mathfrak l$.
\end{lemma}

\begin{proof}[Proof of Proposition \ref{prop-couples} assuming Proposition \ref{prop-rigidity}]
The order of our couples is bounded by $N^3$, so there are at most $O(C^{N^3} (N^3)!)$ couples of order $n$, independent of $L$. Similarly, given a couple $\Q$ and the corresponding set $\mathscr C$, defined in Lemma \ref{lem-chains}, there are $O(C^N)$ choices for the collection of SG negative-chain like objects, each of which corresponds to a set $\mathscr C_1$ and thus $\mathcal M_{\mathscr C_1}$ defined in Lemma \ref{lem-chains} and Definition \ref{def-M}. So, it suffices to fix a couple $\Q$ along with a choice $\mathscr C_1$ and $\mathcal M_{\mathscr C_1}$ and bound the corresponding expression for  $\K_{\textgoth{Q}_{\mathcal M_{\mathscr C_1}}}$ in (\ref{eq-splice-exp}). 

For each leaf, we get decay from $n_{\mathrm{in}}$ being Schwartz, so we may restrict  $|k_\ell| \leq L^\theta$ (see Remark \ref{rmk-how-large} for clarification). Then, we may fix some value of $k_\ell^0$ for each leaf and impose  $|k_\ell - k_\ell^0| \leq 1$. This allows us to restrict each $k_\ell$, not just leaves, to an interval of length 2 up to a loss of $C^n$ (a choice of $k_\ell^0$ for each child of a node yields an interval of length 6 for their parent). So, fixing each $k_\ell^0$ with $|k_\ell^0| \lesssim L^{\theta}$ and expressing each $n_{\mathrm{in}}(k_\ell) = \langle k_\ell\rangle^{-50}\mathcal X(k_\ell)$ with $\mathcal X$ a bounded function, 
\begin{align*}
\text(\ref{eq-splice-exp}) &\leq C^n \left(\alpha T^{1/2} \right)^{n_0} \sup_{\pmb \varsigma \in [0,1]^{\mathcal N_0^{\mathrm{sp}}}} \prod\limits_{\mathcal L^{\mathrm{sp}}}^+ \langle k_\ell^0 \rangle^{-30} \sum_{\lambda[\mathcal N^{\mathrm{sp}}]}\left( \frac{\alpha T}{L}\right)^{n_{\mathrm{sp}}} \sum_{\tilde{\mathscr E_{\lambda[\mathcal N^{\mathrm{sp}}]}^{\mathrm{sp}}}} \left| \int_{\mathcal{E}^{\mathrm{sp}}_{\pmb{\varsigma}}} \prod_{\mathfrak n \in \mathcal N^{\mathrm{sp}}}e^{\zeta_{\mathfrak n} 2 \pi i Tt_{\mathfrak n}\Omega_{\mathfrak n}} \diff t_{\mathfrak n} \right| \\
&\lesssim C^n \left(\alpha T^{1/2} \right)^{n_0} \langle k \rangle^{-20} \left( \sup_{\lambda[\mathcal N^{\mathrm{sp}}]} \sum_{\tilde{\mathscr E_{\lambda[\mathcal N^{\mathrm{sp}}]}^{\mathrm{sp}}}} \left( \frac{\alpha T}{L}\right)^{n_{\mathrm{sp}}}\right) \left( \sum_{\lambda[\mathcal N^{\mathrm{sp}}]}\sup_{\substack{\varsigma \in [0,1]^{\mathcal N_0^{\mathrm{sp}}} }}\left|\int_{\mathcal{E}^{\mathrm{sp}}_{\boldsymbol{\varsigma}}} \prod_{\mathfrak n \in \mathcal N^{\mathrm{sp}}}e^{\zeta_{\mathfrak n} 2 \pi i t_{\mathfrak n}\lambda_{\mathfrak n}} \diff t_{\mathfrak n}\right|\right),
\end{align*}
where $\lambda[\mathcal N^{\mathrm{sp}}] \in \Z^{|\mathcal N^{\mathrm{sp}}|}$ and $\tilde{\mathscr E^{\mathrm{sp}}_{\lambda[\mathcal N^{\mathrm{sp}}]}}$ denotes decorations of the couple satisfying $|T\Omega_{\mathfrak n} - \lambda_{\mathfrak n}| \leq 1$ for each $\mathfrak n \in \mathcal N^{\mathrm{sp}}$, in addition to the degeneracy conditions and choices of $k_\ell^0$ for each leaf and node of $\mathcal N^{\mathrm{sp}}$. Here, $n_0 + n_{\mathrm{sp}} = n$, where $n_0$ is number of nodes spliced out. 

Since we have fixed each $k_\ell^0$, there are at most $L^{4}$ options for each $\lambda_{\mathfrak n}$ and we use Lemma \ref{lem-osc-int} to conclude
\begin{align*}
\sum_{\lambda[\mathcal N^{\mathrm{sp}}]}\sup_{\substack{\varsigma \in [0,1]^{\mathcal N_0^{\mathrm{sp}}} }}\left|\int_{\mathcal{E}^{\mathrm{sp}}_{\boldsymbol{\varsigma}}} \prod_{\mathfrak n \in \mathcal N^{\mathrm{sp}}}e^{\zeta_{\mathfrak n} 2 \pi i t_{\mathfrak n}\lambda_{\mathfrak n}} \diff t_{\mathfrak n}\right| & \lesssim \sum_{\lambda[\mathcal N^{\mathrm{sp}}]} \sum_{d_{\mathfrak n} \in \{0,1\}}\frac{1}{\langle \lambda_{\mathfrak n} + \widetilde{q}_{\mathfrak n}\rangle} \\
&\lesssim C^n (\log L)^{n_{\mathrm{sp}}}.
\end{align*}
Therefore by Proposition \ref{prop-rigidity} with $D = L^\theta$, 
\begin{align*}
\text(\ref{eq-splice-exp}) &\lesssim \langle k \rangle^{-20} C^n \left(\alpha T^{1/2} \right)^{n_0} (\log L)^{n_{\mathrm{sp}}} \left( \sup_{\lambda[\mathcal N^{\mathrm{sp}}]} \sum_{\tilde{\mathscr E_{\lambda[\mathcal N^{\mathrm{sp}}]}^{\mathrm{sp}}}} \left( \frac{\alpha T}{L}\right)^{n_{\mathrm{sp}}}\right) \\
&\lesssim \langle k \rangle^{-20} (\log L)^n T^{-\frac{3}{5}}\left(L^\theta\alpha T^{4/5}\right)^n,
\end{align*}
establishing (\ref{eq-couples}). 
\end{proof}

\begin{remark} \label{rmk-how-large}
The fact that we can restrict leaves $\lesssim L^\theta$ is a bit more subtle than in \cite{2019, WKE, WKE2023} as Proposition \ref{prop-rigidity} depends on the sizes of $k_\ell^0$. So, if a single leaf has $|k_\ell| \sim D$, we lose up to $D^{2N}$ in \eqref{eq-counting} (up to a maximum of $T^{2N}$), requiring a corresponding decay $\gg \langle k_\ell \rangle ^{-2N}$ in $n_{\mathrm{in}}$. As $N$ independent of $L$, this is sufficient, although implies that the smaller we take $\epsilon$ in Theorems \ref{mainthm} and \ref{mmtthm}, the more decay we require on $n_{\mathrm{in}}$. 
\end{remark}

\section{Stage 2: Algorithm}
\label{sec-algorithm}
Here, we lay out the algorithm to count decorations of molecules and reduce Proposition \ref{prop-rigidity} to Proposition \ref{prop-2vc-bound}, a bound on the number of two-vector countings used in the algorithm. 

\subsection{Description of molecules}\label{subsec-assumptions} Before we begin Stage 2, let us first perform the following pre-processing step on couples:

\vspace{.1cm}
\noindent \textbf{Pre-processing Step:} For a couple $\Q$ and corresponding molecule $\M(\Q)$, consider the set $\mathscr C$ and corresponding set $\mathcal M_{\mathscr C}$ described in Lemma \ref{lem-chains} and Definition \ref{def-M}, for all chains, not just negative ones. Splice the couple $\Q$ at the nodes in $\mathcal M_{\mathscr C}$.
\vspace{.1cm}

This allows us to use the description of molecules from Proposition \ref{prop-stage1} without the restriction that chains must be negative or SG (notably we have conditions (b) and (c) for all chains). Further, let us make the following assumptions on the molecules ${\M = \M(\Q)}$:
\begin{enumerate}
\item The molecule $\M$ has 2 degree 3 atoms, with the rest degree 4. 
\item The molecule $\M$ contains no degenerate atoms.
\end{enumerate}
We deal with the pre-processing step and the simplifying assumptions in Remark \ref{rmk-remove-assumptions}. For such molecules, we may rule out the following structures:

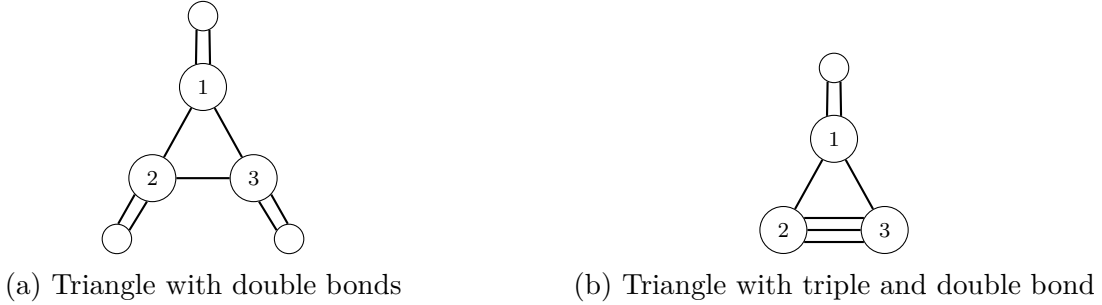
\begin{figure}[h]
    \centering
    \begin{subfigure}[t]{0.5\linewidth}
        \centering
        \begin{tikzpicture}[scale = .67]
        \tikzstyle{every node} = [circle, draw = black]
        \node (1) at (0, 0) {\tiny{1}};
        \node (2) at (-1, -1.8) {\tiny{2}};
        \node (3) at (1,-1.8) {\tiny{3}};
        \node (4) at (0,1.4) {};
        \node (5) at (-1.7, -3) {};
        \node (6) at (1.7, -3) {};

        \draw[thick] (1) -- (2);
        \draw[thick] (2) -- (3);
        \draw[thick] (3) -- (1);
        \draw[thick] (1.73) -- (4.295);
        \draw[thick] (4.245) -- (1.107);
        \draw[thick] (2.223) -- (5.85);
        \draw[thick] (5.35) -- (2.257);
        \draw[thick] (3.283) -- (6.145);
        \draw[thick] (6.95) -- (3.317);
        \end{tikzpicture}
        \subcaption{Triangle with double bonds}
        \label{fig-triangle-td}
    \end{subfigure}
    ~
    \begin{subfigure}[t]{0.5\linewidth}
        \centering
        \begin{tikzpicture}[scale = .67]
        \tikzstyle{every node} = [circle, draw = black]
        \node (1) at (0, 0) {\tiny{1}};
        \node (2) at (-1, -1.8) {\tiny{2}};
        \node (3) at (1,-1.8) {\tiny{3}};
        \node (4) at (0,1.4) {};

        \draw[thick] (1) -- (2);
        \draw[thick] (2.30) -- (3.150);
        \draw[thick] (3.180) -- (2);
        \draw[thick] (2.330) -- (3.210);
        \draw[thick] (3) -- (1);
        \draw[thick] (1.73) -- (4.295);
        \draw[thick] (4.245) -- (1.107);

        \end{tikzpicture}
        \subcaption{Triangle with triple and double bond}
        \label{fig-triangle-tt1}
    \end{subfigure}
    \caption{Two non-allowed atomic groups in Proposition \ref{prop-triangle}.}
    \label{fig-triangle}
\end{figure}

\begin{proposition} \label{prop-triangle}
Molecules assumed above cannot contain the atomic groups in {Figure \ref{fig-triangle}}.
\end{proposition}

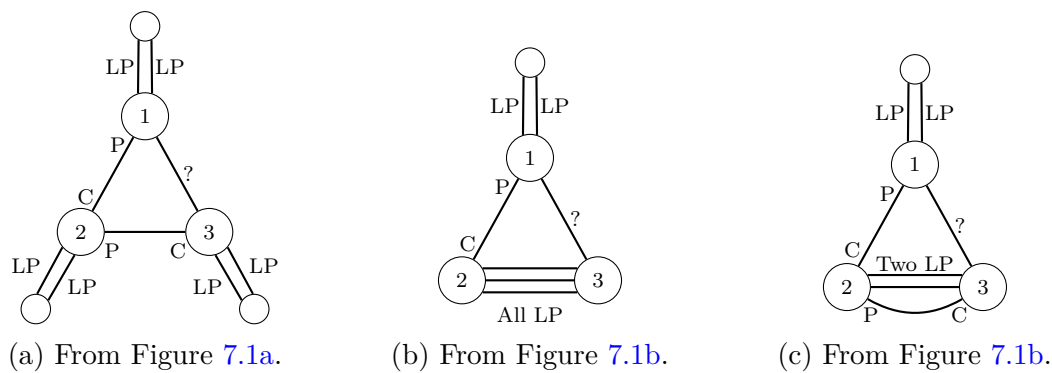
\begin{figure}[t]
    \centering
    \begin{subfigure}[t]{0.3\linewidth}
        \centering
        \begin{tikzpicture}[scale = .85]
        \tikzstyle{a node} = [circle, draw = black]
        \tikzstyle{b node} = [inner sep=1,outer sep=0]
    \node[a node] (1) at (0, 0) {\tiny{1}};
    \node[a node] (2) at (-1, -1.8) {\tiny{2}};
    \node[a node] (3) at (1,-1.8) {\tiny{3}};
    \node[a node] (4) at (0,1.4) {};
    \node[a node] (5) at (-1.7, -3) {};
    \node[a node] (6) at (1.7, -3) {};

    \draw[thick] (1) -- (2)
        node[pos = .1,draw = none]{\tiny{P \ \ \ }}
        node[pos = .8, draw = none]{\tiny{C \ \ \ \ }};
    \draw[thick] (2)  -- (3)
        node[below, pos = .09,draw = none]{\tiny{P}}
        node[below, pos = .91,draw = none]{\tiny{C}};
    \draw[thick] (3) -- (1)
            node[pos = .5, draw = none]{\tiny{\ \ \ ?}};
    \draw[thick] (1.73) -- node[a node, left, draw = none]{\tiny{LP}}(4.295);
    \draw[thick] (4.245) -- node[a node, right, draw = none]{\tiny{LP}}(1.107);
    \draw[thick] (2.220) -- node[left, pos = .4, draw = none]{\tiny{LP}}(5.85);
    \draw[thick] (5.35) -- node[right, pos = .3, draw = none]{\tiny{LP}}(2.255);
    \draw[thick] (3.285) -- node[left, pos = .7, draw = none]{\tiny{LP}} (6.145);
    \draw[thick] (6.95) -- node[right, pos = .6, draw = none]{\tiny{LP}} (3.320);
        \end{tikzpicture}
        \subcaption{From Figure \ref{fig-triangle-td}.}
        \label{fig-triangle-label-td}
    \end{subfigure}
    ~
    \begin{subfigure}[t]{0.3\linewidth}
    \centering
    \begin{tikzpicture}[scale = .9]
    \tikzstyle{a node} = [circle, draw = black]
    \tikzstyle{b node} = [inner sep=1,outer sep=0]
    \node[a node] (1) at (0, 0) {\tiny{1}};
    \node[a node] (2) at (-1, -1.8) {\tiny{2}};
    \node[a node] (3) at (1,-1.8) {\tiny{3}};
    \node[a node] (4) at (0,1.4) {};
    \node[b node] (5) at (0,-2.3) {\tiny{All LP}};

    \draw[thick] (1) -- (2)
        node[a node, pos = .1,draw = none]{\tiny{P \ \ \ }}
        node[a node, pos = .8, draw = none]{\tiny{C \ \ \ \ }};
    \draw[thick] (2.30) -- (3.150);
    \draw[thick] (3.180) -- (2);
    \draw[thick] (2.330) -- (3.210); 
    \draw[thick] (3) -- (1) node[a node, pos = .5, draw = none]{\tiny{\ \ \ ?}};
    \draw[thick] (1.73) -- node[a node, left, draw = none]{\tiny{LP}}(4.295);
    \draw[thick] (4.245) -- node[a node, right, draw = none]{\tiny{LP}}(1.107);

        \end{tikzpicture}
        \subcaption{From Figure \ref{fig-triangle-tt1}. }
        \label{fig-triangle-label-tt1}
    \end{subfigure}
    ~
    \begin{subfigure}[t]{0.3\linewidth}
        \centering
        \begin{tikzpicture}[scale = .9]
        \tikzstyle{a node} = [circle, draw = black]
        \tikzstyle{b node} = [inner sep=1,outer sep=0]
    \node[a node] (1) at (0, 0) {\tiny{1}};
    \node[a node] (2) at (-1, -1.8) {\tiny{2}};
    \node[a node] (3) at (1,-1.8) {\tiny{3}};
    \node[a node] (4) at (0,1.4) {};
    \node[b node] (5) at (0,-1.45) {\tiny{Two LP}};
    \node[b node] (6) at (-.65,-2.2) {\tiny{P}};
    \node[b node] (7) at (.65,-2.2) {\tiny{C}};

    \draw[thick] (1) -- (2)
        node[a node, pos = .1,draw = none]{\tiny{P \ \ \ }}
        node[a node, pos = .8, draw = none]{\tiny{C \ \ \ \ }};
    \draw[thick] (2.30) -- (3.150);
    \draw[thick] (3.180) -- (2);
    \draw[thick] (2.330) to [out=-30,in=-150]  (3.210);
    \draw[thick] (3) -- (1)
            node[a node, pos = .5, draw = none]{\tiny{\ \ \ ?}};
    \draw[thick] (1.73) -- node[a node, left, draw = none]{\tiny{LP}}(4.295);
    \draw[thick] (4.245) -- node[a node, right, draw = none]{\tiny{LP}}(1.107);

        \end{tikzpicture}
        \subcaption{From Figure \ref{fig-triangle-tt1}.}
        \label{fig-triangle-label-tt2}
    \end{subfigure}
    \caption{Labelling of atomic groups in Figure \ref{fig-triangle}.}
    \label{fig-triangle-label}
\end{figure}

\begin{proof}
The molecule is of the form $\M = \M(\Q)$, so we attempt to label each bond as LP or PC in Figure \ref{fig-triangle-label} in accordance with the couple structure.
Each double bond must be a CN double bond as they cannot be part of a pseudo-hyperchain, so are labelled with LP.

Focusing on Figure \ref{fig-triangle-td}, the single bonds connecting the atoms labelled 1,2, and 3 cannot all be LP, so let us assume that the bond connecting atoms 2 and 3 is PC, with atom 2 labelled as P. This forces the bond between atom 1 and 2 to be PC with atom 1 labelled as P since all of atom 2's children in the couple $\Q$ have been accounted for. Then, there is no label we can put for the bond connecting atoms 1 and 3, as indicated in Figure \ref{fig-triangle-label-td}.

Similarly, if we focus on Figure \ref{fig-triangle-tt1}, at least two of the bonds in the triple bond connecting atoms 2 and 3 must be LP, indicated in Figures \ref{fig-triangle-label-tt1} and \ref{fig-triangle-label-tt2}, where we arbitrarily chose to label 2 with a P if one of the bonds is PC. In either case, the bond connecting atoms 1 and 2 must be PC with atom 1 labelled as P. Then, there is no label we can put for the bond connecting atoms 1 and 3. 
\end{proof}

\subsection{Algorithm}
On a molecule as in Section \ref{subsec-assumptions}, we perform the algorithm below. This allows us to bound the decorations of a molecule one step at a time. Suppose we have a molecule $\M^{\mathrm{pre}}$ at some stage of the algorithm and we perform an operation of the algorithm to yield $\M^{\mathrm{post}}$. Correspondingly let the number of decorations of each be denoted by $\mathfrak D^{\mathrm{pre}}$ and $\mathfrak D^{\mathrm{post}}$. 
For the operation, let $\mathfrak C$ denote a corresponding counting estimate so that roughly, 
\begin{equation}
\mathfrak D^{\mathrm{pre}} \leq \mathfrak C \cdot \mathfrak D^{\mathrm{post}}.
\end{equation}
Start the following at Operation 1, where each time we remove an atom, we also remove all bonds connected to it:  

\begin{enumerate}\settozero
    \item[(0)] If possible, remove an atom of degree 2 with a double bond. Go to (0). 
    \item Otherwise, if possible, remove a bridge (recall Definition \ref{def-molecules}). Go to (0).
    \item Otherwise, if possible, remove an atom of degree 3 with a triple bond. Go to (1).
    \item Otherwise, if possible, remove an atom of degree 3 with a double bond. Go to (0).
    \item Otherwise, if possible, remove an atom of degree 3 with only single bonds. Go to (1).
    \item Otherwise, if possible, remove an atom of degree 2 with only single bonds which satisfy at least one of the following: 
        \begin{enumerate}[(a)]
        \item The atom is connected to two atoms which are themselves connected by a double bond.
        \item The atom is connected to two atoms which are themselves connected by a single bond, which becomes a bridge if the operation were to be performed.
        \item The atom is connected to two atoms which are themselves connected by a single bond and there is no double bond at either. 
        \item The atom is connected to two atoms which are themselves not connected.
        \end{enumerate}
    Go to (1).
    \item Otherwise, if possible, remove an atom of degree 2 with only single bonds connected to two atoms which are themselves connected by a single bond, and there is a double bond at precisely one of them. Go to (1).
    \item Otherwise, if possible, remove an atom of degree 2 with only single bonds which is connected to two atoms which are themselves connected by a triple bond. Go to (2).
    \item Otherwise, if possible, remove an atom of degree 2 with only single bonds connected to two atoms which are themselves connected by a single bond, and there is a double bond at both of them. Go to (1).
    \item Otherwise, remove a sole atom (degree 0) with no edges. Repeat.
\end{enumerate}

\subsection{Invariant properties}
We state the following invariant properties describing the molecule after a given operation is performed.

\begin{proposition}\label{prop-inv}
After Operation 2 or 4 - 9 is performed, the molecule contains no degree 2 atoms with a double bond. Similarly, Operation 3 creates at most 1 such atom, and Operation 1 creates at most 2. After performing Operation 0, there is at most one such atom left (in a different component).
\end{proposition}
\begin{proof}
To create a degree 2 atom with a double bond while performing the algorithm, the atom must originally be degree 3 (degree 4 isn't possible as chains of double bonds are not allowed) and have its degree reduced to 2. Operations 4 - 9 cannot create such an atom since the removal of degree 3 atoms with double bonds are prioritized.
Operation 2 cannot create such an atom as it reduces the degree of two atoms by 3. Operation 3 removes an atom of degree 3 with a double bond, so this atom must be connected to two distinct atoms. The atom it is connected to via double bond cannot be part of an additional double bond due to splicing, leaving one atom which may become degree 2 with a double bond. Operation 1 removes a bridge, of which each of the atoms connected to it could become degree 2. Lastly, Operation 0 cannot create new degree 2 atoms with double bonds. Via other operations, the most that exist at any time is 2, so after Operation 0, at most 1 remains.
\end{proof}

\begin{proposition} \label{prop-deg2}
After each of Operation 5-8 are performed, the component of the molecule in which it was performed has at least 2 atoms of degree 1 or 3.
\end{proposition}
\begin{proof}
If one of Operations 5-8 is being performed, each atom in the component is even degree of degree 2 or 4. We remove two bonds, so the number of edges is still even degree. Since the bonds removed were not double bonds, two atoms now have odd degree. 
\end{proof}

\begin{proposition} \label{prop-sole}
Only Operations 1 and 2 may create a sole atom. 
\end{proposition}
\begin{proof}
Let us show the other operations may not create a sole atom. If Operation 0 created a sole atom, the molecule must have contained a component of two degree 2 atoms connected by a double bond. As chains of double bonds are not allowed and Operation 0 is prioritized, a single operation must have removed precisely one other bond at each atom simultaneously, causing them to be degree 2. Due to Proposition \ref{prop-inv} and the lack of chains, no such operation can do this. If Operation 3-8 created a sole atom, then the atom removed as part of this operation is connected to an atom of degree 1 (Operation 1 would have priority) or an atom of degree 2 with a double bond (Operation 0 would have priority). 
\end{proof}

\subsection{Operation types and Proposition \ref{prop-2vc-bound}}
We sort each operation of the algorithm with the same value of $\mathfrak C$ (and $\Delta \chi$) into the named types below, where we will keep track of the number of operations performed of each of these types. 
\begin{enumerate}
\item \textbf{Bridge Operations} consist of Operation 1. In this case, $\Delta \chi = 0$ and $\mathfrak C = 1$ by summing (\ref{eq-mol-k}) for all atoms on one side of the bridge and noting that all bonds but the bridge appear precisely twice with opposite signs. Let $m_0$ denote the number of bridge operations. 
\item \textbf{Sole Atom Operations} consist of Operation 9. Here $\Delta \chi = 0$ and as no bonds are being removed, $\mathfrak C = 1$. Let $m_1$ denote the number of sole atom operations. 
\item \textbf{Two-Vector Counting Operations} consist of Operations 0 and 5 - 8. In this case, $\Delta \chi = -1$ and $\mathfrak C = L$ by Proposition \ref{prop-vc}. Let $m_2$ denote the number of two-vector counting operations. 
\item \textbf{Three-Vector Counting Operations} consist of Operations 2-4. In this case, ${\Delta \chi = -2}$ and ${\mathfrak C = L^{2 + \theta}T^{-1}D^{2-\sigma}}$ by Proposition \ref{prop-vc}. Let $m_3$ denote the number of three-vector counting operations. 
\end{enumerate}

We may obtain the following initial estimate on $m_0$:

\begin{proposition}\label{prop-bridge-op}
(Bound on bridge operations) For a molecule $\mathbb{M}$ as in Section \ref{subsec-assumptions}, ${m_0 < m_3}$.
\end{proposition}
\begin{proof}
Writing equations for edges and $\chi$, we get
\begin{align}
3m_3 + 2m_2 + m_0 = 2n-1 \\
2m_3 + m_2 = n \label{eq-chi}
\end{align}
By subtracting the two equations, we may deduce that 
\begin{equation}\label{eq-all-op}
m_3 + m_2 + m_0 = n - 1. 
\end{equation}
So, if $m_0 \geq m_3$, (\ref{eq-all-op}) becomes $$2m_3 + m_2 \leq n-1,$$ which contradicts (\ref{eq-chi}). Therefore, $m_0 < m_3$.
\end{proof}

Now, we may reduce Proposition \ref{prop-rigidity} to obtaining a bound on the number of two-vector counting operations performed.

\begin{proposition}{(Bound on 2 v.c.)}\label{prop-2vc-bound} For any molecule $\M(\Q)$ as assumed in Section \ref{subsec-assumptions}, we must have 
\begin{equation}\label{eq-bound}
m_2 \leq 3(m_3 - 1).
\end{equation}
\end{proposition}

\begin{remark}\label{rmk-bad-molecule}
The bound (\ref{eq-bound}) is the primary estimate allowing us to reach times close to $\alpha^{-\frac{5}{4}}$, but is also the reason that the techniques in this paper do not allow us to reach $T_{\mathrm{kin}}$ as it is not possible to improve on the bound on $m_2$ for the molecule in Figure \ref{fig-eg-molecule}. Note that Figure \ref{fig-bad-mol} contains a molecule that theoretically saturates \eqref{eq-bound}, although this molecule is not allowed at this stage as the double bonds which are not pseudo-hyperchains cannot all be CL. However, we use this molecule throughout this section in Example \ref{ex-counting} and \ref{ex-mapping} to illustrate various aspects of the algorithm as well as provide clarity on what improvements need to be made to exclude such a molecule explicitly.
\end{remark}

\begin{proof}[Proof of Proposition \ref{prop-rigidity} assuming Proposition \ref{prop-2vc-bound}]
The equation for $\chi$ is $2m_3 + m_2 = n$ and by Proposition \ref{prop-2vc-bound}, $5m_3 - 3 \geq n$. So, 
\begin{align*}
\mathfrak D &\lesssim C^n (L^{2 + \theta}T^{-1}D^{2-\sigma})^{m_3}L^{m_2} \\
&\lesssim C^n L^n D^{n(2 - \sigma)} T^{-m_3} \\
& \lesssim C^n L^n D^{n(2-\sigma)}T^{-\frac{n}{5} - \frac{3}{5}}.
\end{align*}
Note that Proposition \ref{prop-2vc-bound} does not cover all molecules considered in Proposition \ref{prop-rigidity}. However, we handle these cases in Remark \ref{rmk-remove-assumptions}. 
\end{proof}

\section{Proof of Proposition \ref{prop-2vc-bound}} \label{sec-opbound}
Here, we prove Proposition \ref{prop-2vc-bound} by mapping each two-vector counting operation to a three-vector counting operation or bridge operation in a way which is at most 1-1 in three-vector counting operations and 3-1 in bridge operations.

\subsection{Operation trees} We first build up the necessary structure on operations.
\begin{definition}{(Operation trees)}\label{def-op-tree} In order to keep track of the order of operations in a meaningful way, we will use a binary tree. We construct the tree and label its nodes in conjunction with performing the algorithm in the following way: 
\begin{enumerate}
\item Before starting the algorithm, create a root $r$, setting $i_r = 1$.
\item If you are performing Operation $o$, assume that the tree contains a leaf corresponding to each component of the molecule, each of which has a counter $i_n$. Record Operation $o$ as $o_{i_n}$ in the leaf corresponding to its component following all other operations at the leaf. Then, if $o \neq 9$, increase $i_n$ by 1. 
\item If $o = 1$, record the operation as above and then from that leaf, create two separate child branches, $n_1$ and $n_2$, for each newly created component. In each of these children, record $1_0$. Then, set $i_{n_1} = i_{n_2} = 1$. 
\end{enumerate}
For clarity in referring to specific nodes later on, we number each node. Set $r = 1$ for the root. For a node $n$ with children $n_1$ and $n_2$ from left to right, set $n_1 = 2n$ and $n_2 = 2n+1$.
\end{definition}

\begin{example} \label{ex-counting}
A molecule and its corresponding operation tree can be seen in Figure \ref{fig-bad-mol}. Each node is  labelled with its node number for clarity and contains the operations performed in order as a list. 
\end{example}

\begin{figure}
\begin{tikzpicture}[scale = .9]
    \tikzstyle{every node} = [circle, draw = black]
    \node (1) at (-5,0) {};
    \node (2) at (-4,0) {};
    \node (3) at (-3,0) {};
    \node (4) at (-2,0) {};
    \node (5) at (-1,0) {};
    \node (6) at (0,0) {};
    \node (7) at (1,0) {};
    \node (8) at (2,0) {};
    \node (9) at (3,0) {};
    \node (10) at (4,0) {};
    \node (11) at (-2.5,1) {};
    \node (12) at (1.5,1) {};

    \draw[->,  thick] (1.25) -- (2.155);
    \draw[->, thick] (2.205) -- (1.335);
    \draw[->,  thick] (2) -- (3);
    \draw[->,  thick] (3.25) -- (4.155);
    \draw[->, thick] (4.205) -- (3.335);
    \draw[->,  thick] (4) -- (5);
    \draw[->,  thick] (5.25) -- (6.155);
    \draw[->, thick] (6.205) -- (5.335);
    \draw[->,  thick] (6) -- (7);
    \draw[->,  thick] (7.25) -- (8.155);
    \draw[->, thick] (8.205) -- (7.335);
    \draw[->,  thick] (8) -- (9);
    \draw[->,  thick] (9.25) -- (10.155);
    \draw[->, thick] (10.205) -- (9.335);
    \draw[->,  thick] (3) -- (11);
    \draw[->,  thick] (11) -- (2);    
    \draw[->,  thick] (5) -- (11);    
    \draw[->,  thick] (11) -- (4);
    \draw[->,  thick] (7) -- (12);
    \draw[->,  thick] (12) -- (6);    
    \draw[->,  thick] (9) -- (12);    
    \draw[->,  thick] (12) -- (8);
    \draw[->,   thick] (1.270) to [out=-15,in=-165] (10.270);
\end{tikzpicture}
\hspace{.5cm}
\begin{tikzpicture}[level distance=1.5cm,
  level 1/.style={sibling distance=3.7cm},
  level 2/.style={sibling distance=2cm}]
\tikzstyle{hollow node}=[draw,inner sep=2]
\node[hollow node]{1: $3_1$, $0_2$, $6_3$, $3_4$, $5_5$, $1_6$}
    child{node[hollow node]{2: $1_0$, $9_1$}}
    child{node[hollow node]{3: $1_0$, $0_1$, $6_2$, $3_3$, $5_4$, $1_5$}
        child{node[hollow node]{6: $1_0$, $9_1$}}
        child{node[hollow node]{7: $1_0$, $9_1$}}
    };
\end{tikzpicture}
\caption{Operation tree for a molecule}
\label{fig-bad-mol}
\end{figure}
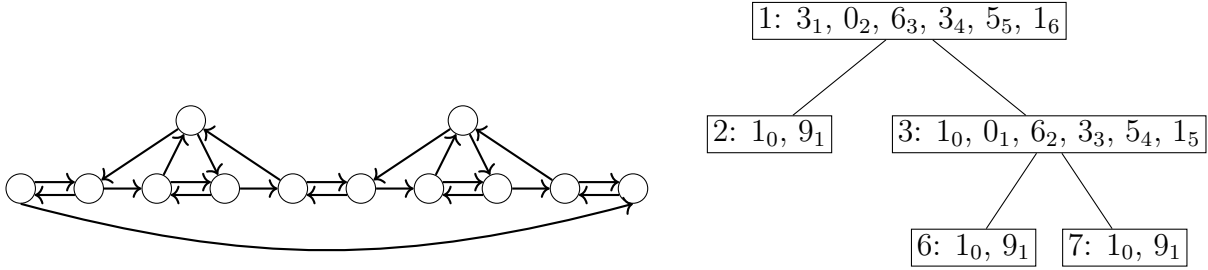

\begin{remark}\label{rmk-op-tree}
Each instance of an operation other than Operation 1 corresponds to one record listed inside a node of the operation tree. Each instance of Operation 1 corresponds to three records listed in three separate nodes of the operation tree.
\end{remark}

\begin{proposition}{(Leaf property)} \label{prop-leaf}
Each leaf of an operation tree is one of the following: 
\begin{enumerate}
\item A node containing $1_0, 9_1$ corresponding to only Operation 9 being performed in the component. 
\item A node $n$ containing at least 3 operations, the last of which are $2_{i_{n-1}}, 9_{i_n}$.
\end{enumerate}
\end{proposition}

\begin{proof}
A component of the molecule is completely removed once no atoms remain. So, the last operation must be Operation 9 as all other operations necessitate multiple atoms and remove at most one. Due to Proposition \ref{prop-sole}, only Operation 1 or 2 can create a sole atom, which is necessarily in its own component. 
\end{proof}

\begin{definition}{(Operation tuple)}
An \textit{operation tuple} $(n, k, o)$ of an operation tree is a triple of a node $n$ in the operation tree as well as a pair $(k,o)$ such that Operation $o_k$ is listed in node $n$ of the operation tree. An operation tuple containing Operation $o$ (respectively, one of Operations $o, \ldots, o + n$), we denote as an ($o$)-tuple (respectively, $(o :(o + n))$-tuple).

Denote $T$ the space of all such tuples corresponding to an operation tree. Let $T_3$ be the set of such tuples where $o \in \{1, 2, 3, 4\}$ corresponding to 3-vector counting operations and bridge operations, and $T_2$ to be the set of tuples where $o \in \{0, 5, 6, 7,8\}$ corresponding 2-vector counting operations. Note that $T_3$ contains $m_3 + 3m_0$ elements due to Remark \ref{rmk-op-tree}.
\end{definition}

\begin{definition}{(Partial ordering of tuples)}
Define a partial order on operation tuples of an operation tree so that $(n_1, k_1, o_1) < (n_2, k_2, o_2)$ if 
\begin{itemize}
    \item The node $n_2$ is a descendant of the node $n_1$ in the operation tree, or
    \item $n_2 = n_1$ and $k_1 < k_2$. 
\end{itemize}
We say that $(n_1, k_1, o_1)$ is \textit{before} $(n_2, k_2, o_2)$ in the operation tree, and similarly refer to $(n_2, k_2, o_2)$ as being \textit{after} $(n_1, k_1, o_1)$ in the operation tree. If $o_1 \neq 1$, $(n_2, k_2, o_2)$ is the \textit{next tuple} to $(n_1, k_1, o_1)$ if there is no other distinct operation tuple $(n_3, k_3, o_3)$ such that {$(n_1, k_1, o_1) < (n_3, k_3, o_3) < (n_2, k_2, o_2)$}. Similarly, we can say that $(n_2, k_2, o_2)$ is \textit{immediately preceded} by the Operation tuple $(n_1, k_1, o_1)$, now without the restriction that $o_1 \neq 1$. 
\end{definition}

\begin{remark} Due to our node-numbering convention, we can reconstruct the operation tree from the operation tuples. 
\end{remark}

\subsection{Operation map} Now we are prepared to define the operation map. 
\begin{definition}{(Mapping of 2 v.c.)}\label{def-mapping}
Define $\Phi: T_2 \rightarrow T_3$ by mapping $(n,k,o) \in T_2$ using the following steps: 
\begin{enumerate}
\item  If $o \in \{5, 6, 7,8\}$, map $(n,k,o)$ to the most recent tuple in $T_3$ before it. 
\item If $o = 0$, the mapping is more complex: 
\begin{enumerate}
\item If the next tuple is not a (5:8)-tuple, then we map $(n,k,o)$ to the most recent tuple in $T_3$ before it. 
\item Otherwise, if the next tuple is a (5)-tuple, we map $(n,k,o)$ to the next tuple in $T_3$ after it. 
\item Otherwise, the next tuple is a (6)-tuple. We map $(n,k,o)$ to the next tuple in $T_3$ which is not immediately preceded by a (6)-tuple.
\end{enumerate}
\end{enumerate}
\end{definition}

\begin{example}\label{ex-mapping}
Taking the molecule and operation tree in Example \ref{ex-counting}, $\Phi$ is as follows: 
\begin{align*}
(1, 2, 0) &\mapsto (1, 6, 1), && (3, 1, 0) \mapsto (3, 5, 1), \\
(1, 3, 6) &\mapsto (1, 1, 3), && (3, 2, 6) \mapsto (3, 0, 1),\\
(1, 5, 5) &\mapsto (1, 4, 3), && (3, 4, 5) \mapsto (3, 3, 3).
\end{align*}
\end{example}

\begin{proposition}{(Well-defined and Injective)} \label{prop-defined-map} The map $\Phi$ defined in Definition \ref{def-mapping} is a well-defined injective map of $T_2$ into $T_3$. The map is not surjective, with there being an operation tuple for each leaf of the operation tree not in the image of $\Phi$.
\end{proposition}

\begin{proof}
To see that $\Phi$ is well-defined, note that a (7:8)-tuple cannot be immediately preceded by a (0)-tuple due to Proposition \ref{prop-triangle}. Additionally, Proposition \ref{prop-leaf} and the fact that the algorithm must begin with a three-vector counting implies that each element of $T_2$ will be mapped to an element of $T_3$. In fact, each element of $T_2$ must be mapped to an element of $T_3$ with the same node $n$. This is because each non-leaf node in the operation tree has its smallest and largest tuple being elements of $T_3$ (for a leaf the largest tuple is a (9)-tuple immediately preceded by a tuple in $T_3$). Note also that Step 2c of the mapping cannot map a (6)-tuple to a tuple at a different node since Operation 6 does not create a bridge, so a (6)-tuple cannot immediately precede a (1)-tuple.

So, we only need to show injectivity at the level of each node of the operation tree. We look at each step of the $\Phi$ map. Starting with Steps 1 and 2a, note that Proposition \ref{prop-inv} ensures that a (0)-tuple is immediately preceded by a (1) or (3)-tuple. Also, a (5:8)-tuple must be immediately preceded by either a (0)-tuple or element of $T_3$. So, these steps themselves cannot map multiple elements of $T_2$ to the same element of $T_3$. 

Now, we consider Step 2b of the operation map and split into each case of Operation 5: 
\begin{itemize}
\item \textbf{Operation 5a or 5d}: As seen in Figure \ref{fig-op5ad}, a (5)-tuple originating from Operation 5a (respectively, 5d), must be immediately followed by an (3)-tuple (respectively (1) or (4)-tuple). After the corresponding operations are performed, the molecule cannot contain atoms of degree 2 with double bonds and must contain at least one atom of odd degree. So, in this case the relevant sequence of tuples is a (0)-tuple followed by a (5)-tuple followed by at least two elements of $T_3$. So the (0)-tuple is mapped to a distinct element of $T_3$ from those mapped to above.
\begin{figure}[h]
\begin{tikzpicture}[scale = .7]
    \hspace{-2cm}
    \tikzstyle{hollow node}=[circle,draw = black]
    \tikzset{square node/.style = {rectangle,draw=black,inner sep=20}}
    \node[hollow node] (2) at (5.3, -2) {};
    \node[hollow node] (3) at (6.3, -1) {};
    \node[hollow node] (4) at (6.3, -3) {};

    \draw[thick] (2) -- (3);
    \draw[thick] (2) -- (4);
    \draw[thick] (4.65) -- (3.295);
    \draw[thick] (3.245) -- (4.115);
    \draw[thick] (3) -- (7.3, -1);
    \draw[thick] (4) -- (7.3, -3);

    \hspace{4cm}

    \tikzstyle{hollow node}=[circle,draw = black]
    \tikzset{square node/.style = {rectangle,draw=black,inner sep=20}}
    \node[hollow node] (2) at (5.3, -2) {};
    \node[hollow node] (3) at (6.3, -1) {};
    \node[hollow node] (4) at (6.3, -3) {};

    \draw[thick] (2) -- (3);
    \draw[thick] (2) -- (4);
    \draw[thick] (3) -- (7, -.2);
    \draw[thick, dotted] (3) -- (7.5, -1);
    \draw[thick, dotted] (3) -- (7, -1.8);
    \draw[thick] (4) -- (7, -2.2);
    \draw[thick, dotted] (4) -- (7.5, -3);
    \draw[thick, dotted] (4) -- (7, -3.8);
    
\end{tikzpicture}
\caption{Operation 5a or 5d is performed.}
\label{fig-op5ad}
\end{figure}
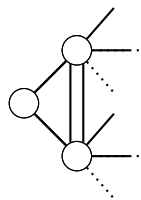

\item \textbf{Operation 5b}: Operation 5b creates a bridge, so the following tuple is a (1)-tuple, which ends a node of the operation tree and therefore is a distinct element of $T_3$ from those mapped to above.

\item \textbf{Operation 5c}: As seen in Figure \ref{fig-op5c}, as part of Operation 5c, two atoms (denoted b and c) become odd degree (in fact, they must be degree 3, or else we can classify the operation as 5b rather than 5c). So, Operation 5c is immediately followed by Operation 4 at atom b or c, suppose atom b. In this case, atom b must be connected to two atoms, which become odd degree. So, the relevant sequence of tuples is a (0)-tuple followed by a (5)-tuple followed by at least two elements of $T_3$, and the (0)-tuple is mapped to a distinct element of $T_3$ from those mapped to above.
\begin{figure}[h!]
\begin{tikzpicture}[scale = .7]

    \tikzstyle{hollow node}=[circle,draw = black]
    \tikzset{square node/.style = {rectangle,draw=black,inner sep=20}}

    \node[hollow node] (2) at (5.3, -2) {};
    \node[hollow node] (3) at (6.3, -1) {\tiny{b}};
    \node[hollow node] (4) at (6.3, -3) {\tiny{c}};
    
    \draw[thick] (2) -- (3);
    \draw[thick] (2) -- (4);
    \draw[thick] (3) -- (4);
    \draw[thick, dotted] (3) -- (7.3, -.7);
    \draw[thick, dotted] (3) -- (7.3, -1.3);
    \draw[thick, dotted] (4) -- (7.3, -2.7);
    \draw[thick, dotted] (4) -- (7.3, -3.3);
    
\end{tikzpicture}
\caption{Operation 5c is performed.}
\label{fig-op5c}
\end{figure}
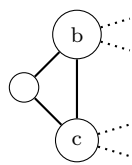
\end{itemize}

Lastly, we consider Step 2c of the operation map. At the point at which Operation 6 is performed in the algorithm, the relevant atomic group is represented in Figure \ref{fig-op6}. Once Operation 6 is performed, we are forced to perform Operation 3 on atom b, at which point we must perform a two-vector counting operation, specifically Operation 5 or 6 due to Proposition \ref{prop-triangle} and the double bond between atoms b and d. If Operation 6 is then performed, we are in a loop of performing Operation 6 and 3, which eventually must terminate in Operation 5. So, the relevant sequence of tuples is a (0)-tuple followed by repeating instances of a (6)-tuple and (3)-tuple followed by a (5)-tuple, which is followed by at least two elements of $T_3$ using the analysis above. So, the (0)-tuple is mapped to a distinct element of $T_3$ from those mapped to above. 

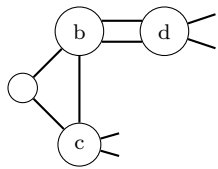
\begin{figure}[h]
\begin{tikzpicture}[scale = .75]
    \tikzstyle{hollow node}=[circle,draw = black]
    \tikzset{square node/.style = {rectangle,draw=black,inner sep=20}}

    \node[hollow node] (2) at (5.3, -2) {};
    \node[hollow node] (3) at (6.3, -1) {\tiny{b}};
    \node[hollow node] (5) at (7.8, -1) {\tiny{d}};
    \node[hollow node] (4) at (6.3, -3) {\tiny{c}};
    
    \draw[thick] (2) -- (3);
    \draw[thick] (2) -- (4);
    \draw[thick] (3) -- (4);
    \draw[thick] (5) -- (8.7, -.7);
    \draw[thick] (5) -- (8.7, -1.3);
    \draw[thick] (3.25) -- (5.155);
    \draw[thick] (5.205) -- (3.335);
    \draw[thick] (4) -- (7, -2.8);
    \draw[thick] (4) -- (7, -3.2);
    
\end{tikzpicture}
\caption{Operation 6 is performed.}
\label{fig-op6}
\end{figure}

To establish that the map is not surjective, recall that each element of $T_2$ is mapped to an element of $T_3$ with the same node. If a leaf of an operation tree is $1_0, 9_1$, there are no corresponding operation tuples in $T_2$, so the operation tuple corresponding to $1_0$ will have nothing mapped to it. Otherwise a leaf of an operation tree ends with $2_{i_{n-1}}, 9_{i_n}$, in which case at the point at which Operation 2 is performed, the component of the molecule is two atoms connected with a triple bond. Due to the priority of operations in the algorithm, the (2)-tuple must be immediately preceded by a (1), (4), or (7)-tuple. Note that a (7)-tuple is always mapped to a tuple preceding it and the other types of tuples are elements of $T_3$. By the above analysis, the (2)-tuple is not in the image of $\Phi$. 
\end{proof}

Now, we may prove Proposition \ref{prop-2vc-bound}.
\begin{proof}[Proof of Proposition \ref{prop-2vc-bound}]
Note that as established in Proposition \ref{prop-bridge-op}, we have that ${m_0 \leq m_3-1}$. By Proposition \ref{prop-defined-map}, there are $m_0 + 1$ elements of $T_3$ not in the image of $\Phi$. So, given that $T_3$ contains $m_3 + 3m_0$ elements, 
\begin{equation} \label{eq-counting-bound}
m_2 \leq m_3 + 3m_0 - (m_0 + 1) = m_3 + 2m_0 -1\leq 3 m_3-3.
\end{equation}
\end{proof}

\begin{remark}\label{rmk-remove-assumptions}
Our algorithm does not take into account all relevant molecules in Proposition \ref{prop-rigidity} due to the simplifying assumptions and pre-processing step in Section \ref{subsec-assumptions}. However, we can easily modify the algorithm to take these molecules into account: 
\begin{enumerate}
\item Assumption 1: The molecule $\M(\Q)$ may contain 1 atom of degree 2 rather than 2 of degree 3. However, then $\Q$ has one root the trivial tree paired with a child of the other root. If the root is fully degenerate, we may begin the algorithm by removing this atom and its two bonds, with $\mathfrak C = 1$, creating a molecule with two degree three atoms (unless the bond was a double bond, in which case all atoms in the chain are fully degenerate and can similarly be removed). Otherwise, Stage 1 or the pre-processing step changed $\M(\Q)$ from having 2 degree 3 atoms to 1 degree 2 atom. Then, we may instead splice all but one of the double bonds and begin the algorithm by removing the two degree 3 atoms and the molecule still has 2 atoms of degree 3.
\item Assumption 2: The molecule may contain degenerate atoms: 
\begin{enumerate}
\item An atom could be fully degenerate with a self-connecting bond (note that splicing does not create degenerate atoms with self-connecting bonds that are not fully degenerate). In this case, we add Operation 10 to the algorithm which removes an atom with a single self-connecting bond, corresponding to $\Delta \chi = -1$ and $\mathfrak C = 1$. When performing the algorithm, each atom with a self-connecting bond will eventually become degree 3 (CN double bonds cannot occur at an atom with a self-connecting bond and CL double bonds are part of pseudo-hyperchains), at which point the non self-connecting bond is a bridge removed using Operation 1 and then we perform Operation 10. The result of Proposition \ref{prop-bridge-op} in this case will change to $m_0 < m_3 + m_4$, where $m_4$ is the number of atoms with self-connecting edges. The operation tree structure is left unchanged other than that some leaves will contain Operation 10 rather than Operation 9. Therefore, we now have $m_2 \leq 3(m_3 - 1) + 2m_4$. So, now the proof of Proposition \ref{prop-rigidity} assuming Proposition \ref{prop-2vc-bound} relies on the following estimate, using $\gamma < 1$:
\begin{align*}
L^{2m_3 + m_2}T^{-m_3} &= L^{n - m_4} T^{-\frac{n}{5} - \frac{3}{5} + \frac{3}{5}m_4}\\
&\lesssim L^n T^{-\frac{n}{5} - \frac{3}{5}} \left(L^{m_4(-1 + \frac{3}{4} \gamma)}\right) \\
& \lesssim L^n T^{-\frac{n}{5} - \frac{3}{5}}.
\end{align*}
\item An atom could be (fully) degenerate with no self-connecting bonds as a result of Stage 1 or the pre-processing step, in which case $\{k_1, k_3\} = \{k_2, k\}$. If a bridge or two-vector counting operation is performed involving these bonds, $\mathfrak C$ is at worst unchanged. If any of these bonds are counted via a three-vector counting operation, then now $\mathfrak C = L$. If Stage 1 made the atom degenerate, we will have gained at least $L^\theta \alpha T^{\frac{1}{2}} \leq T^{-\frac{3}{10}}$ from splicing. So, it suffices to note that $LT^{-\frac{3}{10}} \ll L^2T^{-1}$ as $T \ll L^{\frac{5}{4}}$. If instead the pre-processing step made the atom degenerate, we cannot use cancellation but can modify the pre-processing step to remove all but one double bond and replace the three-vector counting with a single five-vector counting with $\mathfrak C = L^2T^{-1}D^{2 - \sigma}$, exploiting the degeneracy. 
\end{enumerate}
\item Pre-processing step: The molecule may contain double bonds that were removed as part of the pre-processing step. Note that for a decoration of the rest of the molecule, we may count each of these double bonds with $\mathfrak C = LT^{-\frac{1}{2}}D^{2 - \sigma}$ from Proposition \ref{prop-vc}. The reason we may appeal to these counting estimates is that a double bond must contain at least one LP bond, which we may restrict to have decoration $\leq D$, as we did when we reduced the proof of Proposition \ref{prop-couples} to Proposition \ref{prop-rigidity}. This improved counting estimate is as good as our three vector-counting estimate, and therefore only improves our bound on $\mathfrak D$ for the molecule. 
\end{enumerate}
\end{remark}

\section{The Operator \texorpdfstring{$\Ell$}{L}} \label{sec-opL}
The main goal of this section is to establish Proposition \ref{prop-remainder}. 

\subsection{Kernels of \texorpdfstring{$\Ell$}{L}} In order to define the kernels of $\Ell$, we will need a generalization of trees and couples, repeated from \cite{WKE}:

\begin{definition}\label{def-flower}
A \textit{flower tree} is a tree $\mathcal T$ with one leaf $\mathfrak f$ specified, called the \textit{flower}; different choices of $\mathfrak f$ for the same tree $\mathcal T$ leads to different flower trees. There is a unique path from the root $\mathfrak r$ to the flower $\mathfrak f$, which we call the \textit{stem}. A \textit{flower couple} is a couple formed by two flower trees, such that the two flowers are paired (in particular, they have opposite signs). 

The \textit{height} of the flower tree $\mathcal T$ is the number of branching nodes in the stem of $\mathcal T$. Clearly a flower tree of height $n$ is formed by attaching two sub-trees each time and repeating $n$ times, starting from a single node. We say a flower tree is \textit{admissible} if all these sub-trees have scale at most $N$.
\end{definition}

\begin{proposition} \label{prop-Lm}
Let $\mathscr L$ be defined as in (\ref{eq-L}). Note that $\mathscr L^n$ is an $\R$-linear operator for $n \geq 0$. Define its kernels $(\mathscr L^n)_{k \ell}^\zeta(t,s)$ for $\zeta \in \{\pm\}$ by 
\begin{equation*}
(\mathscr L^n b)_k (t) = \sum_{\zeta \in \{\pm\}} \sum_{\ell} \int_{\R} (\mathscr L^n)_{k \ell}^{\zeta} (t,s) b_{\ell}(s)^{\zeta}\mathrm{d} s.
\end{equation*}
Then, for each $1 \leq n \leq N$ and $\zeta \in \{\pm\}$, we can decompose
\begin{equation}
(\mathscr L^n)^\zeta_{k, \ell} = \sum_{n \leq m \leq N^3} (\mathscr L^n)_{k, \ell}^{m, \zeta},
\end{equation}
such that for any $n \leq m \leq N^3$ and $k,\ell \in \Z_L$ and $t,s \in [0,1]$ with $t > s$, we have 
\begin{equation} \label{eq-Lnm}
\E |(\mathscr L^n)_{k, \ell}^{m, \zeta}(t,s)|^2 \lesssim \langle k - \zeta \ell \rangle ^{-20} (L^{\theta}\alpha T^{4/5})^m L^{40}.
\end{equation}
\end{proposition}
\begin{proof}
Using Proposition 11.2 of \cite{WKE}, note that 
\begin{align}
(\mathscr L^n)_{k, \ell}^{m, \zeta}(t,s) &= \sum_{\mathcal T} \widetilde{\mathcal J}_{\mathcal T}(t,s,k,\ell) \\
\E \left| (\mathscr L^n)_{k, \ell}^{m, \zeta}(t,s)\right|^2 &= \sum_{\mathcal Q} \widetilde{\mathcal K}_{\mathcal Q}(t,s,k,\ell),
\end{align}
where the sum in $\mathcal T$ is taken over flower trees and the sum in $\mathcal Q$ is taken over flower couples $(\mathcal T^+, \mathcal T^-)$, where $\mathcal T$ has sign $\zeta$ and $\mathcal T^{\pm}$ has sign $\pm \zeta$. Additionally $\mathcal T$ and $\mathcal T^{\pm}$ have height $n$ and scale $m$, for $n \leq m \leq (1 + 2N)n < N^3$. For $t \geq s$, the quantities $\widetilde{J}_{\mathcal T}, \widetilde{\mathcal K}_{\mathcal Q}$ are defined similar to (\ref{eq-J}) and (\ref{eq-KQ}):
\begin{equation} \label{eq-Jtilde}
(\widetilde{\mathcal J}_{\mathcal T})_k(t) = \left(\frac{\alpha T}{L}\right)^m \zeta(\mathcal{T}) \sum_{\mathscr D} \epsilon_{\mathscr D} \int_\mathcal{D} \prod_{\mathfrak n \in \mathcal N}e^{\zeta_{\mathfrak n} 2 \pi i Tt_{\mathfrak n}\Omega_{\mathfrak n}} \diff t_{\mathfrak n} \boldsymbol{\delta}(t_{\mathfrak f^p} - s) \prod_{\mathfrak f \neq \mathfrak l \in \mathcal L} \sqrt{n_{\mathrm{in}}(k_{\mathfrak l})}g_{k_{\mathfrak l}}^{\zeta_\mathfrak l}(\omega)\boldsymbol{1}_{k_{\mathfrak f} = \ell},
\end{equation}
and 
\begin{equation}\label{eq-KQtilde}
(\widetilde{\mathcal K}_{\mathcal Q})(t,s,k, \ell)  = \left( \frac{\alpha T}{L}\right)^{2m} \zeta(\mathcal Q) \sum_{\mathscr E} \epsilon_{\mathscr E}\int_\mathcal{E} \prod_{\mathfrak n \in \mathcal N}e^{\zeta_{\mathfrak n} 2 \pi i Tt_{\mathfrak n}\Omega_{\mathfrak n}} \diff t_{\mathfrak n} \prod_{\mathfrak f} \boldsymbol{\delta}(t_{\mathfrak f^p} - s) \prod^+_{\mathfrak f \neq \mathfrak l \in \mathcal L} n_{\mathrm{in}}(k_{\mathfrak l})\boldsymbol{1}_{k_{\mathfrak f} = \ell},
\end{equation}
where $\mathscr D$ (resp. $\mathscr E$) is a $k$-decoration of $\mathcal T$ (resp. $\mathcal Q$). The set $\mathcal D$ is defined in (\ref{eq-D}), while $\mathcal E$ is defined in (\ref{eq-E}), but with $s$ replaced with $t$. The only differences in each are the Dirac factors $\boldsymbol{\delta}(t_{\mathfrak f_p} -s)$, where $\mathfrak f^p$ is the parent of $\mathfrak f$ for both flowers, and $\boldsymbol{1}_{k_{\mathfrak f}} = \ell$. The estimate (\ref{eq-Lnm}) follows from the proof of Proposition \ref{prop-couples} and the following observations: 

(1) If $\mathcal Q$ is an admissible flower couple and $\widetilde{\mathcal Q}$ is congruent to $\mathcal Q$ in the sense of Definition \ref{def-congruent}, then $\widetilde{\mathcal Q}$ is also an admissible flower couple, if we choose its flower to be the image of the flower of $\mathcal Q$. Therefore, we may exploit the same cancellations as in Lemma \ref{lem-splice} and decompose the right hand side of (\ref{eq-Lnm}) into sums of the form of (\ref{eq-splice-exp}). The only thing we must be careful of is if the chain contains the flower, which can be avoided via shortening the chain. This may lead to at most two extra double bonds in the molecule, which will cause at most $L^2$ loss in \eqref{eq-couples}.

(2) Note that $k - \zeta \ell$ is a linear combination of $k_{\mathfrak l}$ for $\mathfrak f \neq \mathfrak l \in \mathcal L$, so the decay factor $\langle k - \zeta \ell \rangle^{-20}$ can be obtained from the factors of ${n_{\mathrm{in}}}(k_{\mathfrak l})$. Additionally, we may replace $\boldsymbol{1}_{k_{\mathfrak f} = \ell}$ by $\psi(L(k_{\mathfrak f - \ell}))$ for some suitable cutoff $\psi$. Then, we may repeat all previous arguments, with a loss of at most $L^3$ in \eqref{eq-couples}. 

(3) The Dirac factor $\boldsymbol{\delta}(t_{\mathfrak f^p} - s)$ means that in (\ref{eq-KQtilde}), we are omitting the integration in $t_{\mathfrak f^p}$ for both flowers $\mathfrak f$. However, this difference will cause at most $L^{20}$ loss in \eqref{eq-couples}. 

(4) Unlike in the proof of Proposition \ref{prop-couples}, we may no longer assume that every assignment in the decoration is $\lesssim L^{\theta}$ for $\theta \ll 1$ as $\ell$ can be arbitrarily large. However, as a result, we may assume that only nodes $\mathfrak n$ along the stem may satisfy $k_{\mathfrak n} \gtrsim L^{\theta}$. Remark \ref{rmk-remove-assumptions} clarifies how this is taken into account in the algorithm.
\end{proof}
\subsection{Proof of Proposition \ref{prop-remainder}}
We first record the following: 

\begin{lemma} \label{lem-finite}
Let $k \in \Z_L$, $0 < \sigma \leq 2$ and $\sigma \neq 1$, and $\zeta \in \{\pm\}$. Consider the function
\begin{equation*}
f_{(k)}^\zeta: m \mapsto |k|^\sigma + \zeta |k + m|^\sigma, \hspace{.5cm} \mathrm{Dom}(f_{(k)}^\zeta) = \{m \in \Z_L: |m| \lesssim L^{\theta}, ||k|^\sigma + \zeta |k + m|^\sigma|\lesssim R\}, 
\end{equation*}
where $R \geq L$. Then there exists $k' \in \Z_L$ and $C = C(\sigma)$, where $|k'| \lesssim R^C$, such that $|f_{(k)}^\zeta- f_{(k')}^\zeta| \leq R^{-1}$ on $\mathrm{Dom}(f_{(k)}^\zeta)$. In particular if $\zeta = +$, $k = k'$ and $|k| \lesssim R^C$. Otherwise, $k' = \min(k, R^C)$ and $|f_{(k)}^\zeta(m)| \lesssim L^2$ on $\mathrm{Dom}(f_{(k)}^\zeta)$.
\end{lemma}
\begin{proof}
In the case of $\zeta = +$, it is clear that there are at most $R^{1 + \frac{1}{\sigma}}$ values of $k$ that we must consider. If $\zeta = -$, we separate into two cases, when (a) $0 < \sigma < 1$, and (b) $1 < \sigma \leq 2$: 

\noindent(a) For large $k$, $$|f_{(k)}^-(m)| \lesssim m \cdot \mathrm{min}(|k|, |k+m|)^{\sigma - 1},$$
so considering $|k| \lesssim R^{\frac{2}{1-\sigma}}$ is sufficient, of which there are fewer than $R^C$ for some large $C$ depending on $\sigma$. 

\noindent(b) For large $k$ and $m \neq 0$, $$m \cdot \mathrm{min}(|k|, |k+m|)^{\sigma - 1}\lesssim |f_{(k)}^-(m)|,$$
so when $|k| \gtrsim \left( \frac{1}{m} R\right)^{\frac{1}{\sigma - 1}}$ or $|k| \gtrsim R^{1 + \frac{1}{\sigma - 1}}$, the only element in $\mathrm{Dom}(f_{(k)})$ is $m = 0$. So, we only need to consider at most $R^{2 + \frac{1}{\sigma - 1}}$ values of $k$. 

\end{proof}

\begin{lemma}{(Gaussian Hypercontractivity)} \label{lem-hypercontractivity}
Let $\{g_k\}$ i.i.d. Gaussians or random phase. Given $\zeta_j \in \{\pm\}$ and random variable $X$ of the form 
\begin{equation}
X = \sum_{k_1, \ldots, k_n} a_{k_1, \ldots, k_n} \prod_{j = 1}^n g_{k_j}^{\zeta_j}(\omega),
\end{equation}
where $a_{k_1, \ldots, k_n}$ are constants. Then, for $q \geq 2$,
\begin{equation}
\E \left|X\right|^q \leq (q-1)^{\frac{nq}{2}}\cdot \left( \E |X|^2\right)^{q/2}
\end{equation}
\end{lemma}
\begin{proof}
See, for instance, Lemma 2.6 of \cite{Hypercontractivity}. 
\end{proof}

\begin{proof}[Proof of Proposition \ref{prop-remainder}]
Note that 
\begin{align*}
\|\mathscr L^n(a)\|_Z^2 &= \sup_{0 \leq t \leq 1} L^{-1} \sum_{k \in \Z_L} \langle k \rangle^{10} |(\mathscr L^n a)_k(t)|^2 \\
&= \sup_{0 \leq t \leq 1} L^{-1} \sum_{k \in Z_L} \langle k \rangle^{10} \left| \sum_{\zeta \in \{\pm\}} \sum_{\ell} \int\limits_{0 \leq s \leq t} \sum_{n \leq m \leq N^3} (\mathscr L^n)_{k, \ell}^{m, \zeta}a_{\ell}^\zeta (s) \mathrm{d}s  \right|^2 \\
&\lesssim \sup_{0 \leq s \leq t \leq 1} \sup_{\zeta} \sup_m L^{-1}N^3 \sum_{k \in \Z_L} \left| \langle k \rangle^{5}  \sum_{\ell \in \Z_L} \langle \ell \rangle^{-5} (\mathscr L^n)_{k, \ell}^{m, \zeta}(t,s) \langle \ell \rangle ^5 a_\ell^\zeta(s) \right|^2 \\
& \lesssim \|a\|_Z^2 N^3\sup_{0 \leq s \leq t \leq 1} \sup_\zeta \sup_m  \\ 
& \hspace{1.5cm} \left( \sup_\ell \sum_{k \in \Z_L} \langle k\rangle^5 \langle \ell \rangle^{-5} |(\mathscr L^n)_{k, \ell}^{m, \zeta}(t,s)|\right) \left( \sup_k \sum_{\ell \in \Z_L} \langle k\rangle^5 \langle \ell \rangle^{-5} |(\mathscr L^n)_{k, \ell}^{m, \zeta}(t,s)|\right) \\
& \lesssim \|a\|_Z^2 N^3\sup_{0 \leq s \leq t \leq 1} \sup_\zeta \sup_m  \sup_{k, \ell} \left( \langle k - \zeta \ell \rangle^{8} |(\mathscr L^n)_{k, \ell}^{m, \zeta}(t,s)|\right)^2 \\
& {\hspace{1.5cm}}\left( \sup_\ell\sum_k \frac{\langle k \rangle^5}{\langle \ell \rangle^5} \langle k - \zeta \ell \rangle^{-8}\right) \left( \sup_k\sum_\ell \frac{\langle k \rangle^5}{\langle \ell \rangle^5} \langle k - \zeta \ell \rangle^{-8}\right).
\end{align*} 
So, for fixed $m$ and $\zeta$, it remains to show that with probability $\geq 1 - L^{-A}$, 
\begin{equation} \label{remains}
\sup_{k, \ell} \sup_{0 \leq s \leq t \leq 1}  \langle k - \zeta \ell \rangle^{8} |(\mathscr L^n)_{k, \ell}^{m, \zeta}(t,s)| \lesssim (L^\theta \alpha T^{4/5})^{m/2} L^{45}.
\end{equation}
Note that if we were to fix $k,\ell,s,t$ the bound would hold with probability $ \geq 1 - L^{-A}$ due to hypercontractivity. However, as we consider infinitely many $k, \ell$, the collection of exceptional sets may have measure larger than $L^{-A}$. However, we show that this is not the case. In applying Lemma \ref{lem-hypercontractivity}, what varies for each choice of $k, \ell$ are the constants $a_{k_1, \ldots k_n}$. Note that the only dependence of $a_{k_1, \ldots, k_n}$ on $k,\ell$ are through the quantities $k - \zeta\ell$, and $\mu_j := |k|^\sigma - \zeta_{\mathfrak n_{j+1}} |k_{j}|^\sigma$, where $k_j = k_{\mathfrak n_{j + 1}}$ for $\mathfrak n_j$ the $j$th branching node in the stem from top to bottom and $\mathfrak n_{n+1} = \mathfrak f$. So, we will show that we need to consider only finitely many values of these variables and take a supremum in them rather than in $k$ and $\ell$.

Regarding $k - \zeta \ell$, we may again restrict each non-flower leaf $\mathfrak \ell$ to have $|k_{\mathfrak \ell}| \lesssim L^{\theta}$, yielding at most $C^{N^3}L^2$ choices for $k - \zeta \ell$. The analysis of $\mu_j$ is more complex. Consider the integral part of (\ref{eq-Jtilde}): 
\begin{align*}
(\widetilde{\mathcal J}_{\mathcal T})_{\mathrm{int}}(t,s,k,\ell) &= \int_\mathcal{D} \prod_{\mathfrak n \in \mathcal N}e^{\zeta_{\mathfrak n} 2 \pi i Tt_{\mathfrak n}\Omega_{\mathfrak n}} \diff t_{\mathfrak n}\boldsymbol{\delta}(t_{\mathfrak f^p} - s) \\
&=  \left( e^{\zeta_{\mathfrak f^p} 2\pi i Ts\Omega_{\mathfrak f^p}} \right) \int_\mathcal{D} \prod_{\mathfrak f^p \neq \mathfrak n \in \mathcal N}e^{\zeta_{\mathfrak n} 2 \pi i Tt_{\mathfrak n}\Omega_{\mathfrak n}} \diff t_{\mathfrak n}. \nonumber
\end{align*}
The dependence on $\mu_n$ is only evident in $\Omega_{\mathfrak f^p}$, so does not effect $\left|(\widetilde{\mathcal J}_{\mathcal T})_{\mathrm{int}}\right|$. Otherwise, note that for $1 \leq j \leq n - 1$ (with the convention that $\mu_0 = 0$),
\begin{align*}
\Omega_{\mathfrak n_j} &= \Omega_{\mathfrak n_j}' + \mu_{j - 1} - \mu_j,
\end{align*}
where $\Omega_{\mathfrak n_j}'$ depends on the non-flower leaves so satisfies $|\Omega_{\mathfrak n_j}'| \lesssim C^{N^3}L^\theta$.

Now we will consider a diadic decomposition (starting at $L^{100}$) of $\Omega_{\mathfrak n_j}$ for $1 \leq j \leq n - 2$, letting the largest of these be denoted by $R$. Using Chebyshev, for each $R$, it suffices to show that for $p \geq 2$ and $p \sim C$ (and correspondingly $A \sim p)$,
\begin{equation}\label{eq-LR-bound}
\E \left|\sup_{k, \ell} \sup_{0 \leq s \leq t \leq 1}  \langle k - \zeta \ell \rangle^{8} |(\mathscr L^n)_{R, k, \ell}^{m, \zeta}(t,s)|\right|^p \lesssim p^{np}(L^\theta \alpha T^{4/5})^{mp/2} L^{40p}R^{-p/20},
\end{equation}
where $(\mathscr L^n)_{R, k, \ell}^{m, \zeta}$ denotes the corresponding $R$-contribution to $(\mathscr L^n)_{k, \ell}^{m, \zeta}$ and if $R = L^{100}$, the above holds without the factor $R^{-p/20}$. 

To establish (\ref{eq-LR-bound}), we reduce to considering finitely many values of $\mu_j$  $(j = 1, \ldots, n-1)$ for each $R$. We can assume $|k| \geq L^2$ as otherwise there are at most $L^8$ pairs $(k, \ell)$. Then, we make the following observations: 

\begin{enumerate}
\item If all $|\Omega_j| \leq R$ for $1 \leq j \leq n - 2$, then for the same values of $j$, $|\mu_j| \lesssim R$. Therefore, using Lemma \ref{lem-finite}, there are at most $R^C$ values of $k'\in \Z_L$ such that $f_{(k')}$ approximates $\mu_j$ for all $1 \leq j \leq n-1$. We may include $j = n-1$ as either Lemma \ref{lem-finite} imposes a bound on $k$ or we may approximate $\mu_{n-1}$ up to $R^{-1}$ using $f_{(k')}$. This is sufficient as when we apply (\ref{eq-osc-int}), it is not necessary to know $q_{\mathfrak n}$ precisely, but within $O(1)$. 
\item No $\mathfrak n_j$ is spliced as node $\mathfrak n_j$ on the stem satisfies $|k_{\mathfrak n_j}| \geq L^2/2$ while those off the stem satisfy $|k_{\mathfrak n}| \leq L$, implying that any double bond is LG. Therefore, it must be the case that  $|q_{\mathfrak n_j}| \gtrsim R$ for some $1 \leq j \leq n - 1$ if $R > L^{100}$. To see this, if $|\Omega_{\mathfrak n_{j'}}| \sim R$ for some $1 \leq j' \leq n - 2$, then either $|q_{\mathfrak n_{j'+1}}| \gtrsim R$ or $|q_{\mathfrak n_{j'}}| \gtrsim R$. This allows us to improve the bound (\ref{eq-Lnm}) for an additional gain of $L^4R^{-1}$, as for fixed $k'$, there are $\lesssim L^4$ options for $q_{\mathfrak n_j}$. 
\end{enumerate}

So, we may consider $(\widetilde{\mathscr L^n})_{R, k - \zeta \ell, k'}^{m, \zeta}$, where we have removed the unimodular factor $e^{\zeta_{\mathfrak f^p} 2\pi i Ts\Omega_{\mathfrak f^p}}$ and now rather than depending on $k, \ell$, it depends on the $O(L^2)$ choices for $k - \zeta \ell$ and the $R^C$ choices of $k'$. Now, it suffices to show (\ref{eq-LR-bound}) but for $(\widetilde{\mathscr L^n})_{R, k - \zeta \ell, k'}^{m, \zeta}$, replacing the supremum in $k, \ell$ with one in $k - \zeta \ell, k'$, of which there are $L^CR^C$ choices. Note that if $R > L^{100}$, $(\widetilde{\mathscr L^n})_{R, k - \zeta \ell, k'}^{m, \zeta}$ satisfies (\ref{eq-Lnm}), with an additional gain of $L^4R^{-1}$. Similarly if we take derivatives in $(t,s)$, $\partial_{(t,s)}(\widetilde{\mathscr L^n})_{R, k - \zeta \ell, k'}^{m, \zeta}$ satisfies (\ref{eq-Lnm}) with an additional gain of $L^{12}R^{-1}$.

Therefore, applying Lemma \ref{lem-hypercontractivity}, then integrating in $(t,s)$ and summing in $k - \zeta \ell$ and $k'$, we obtain the following, for $F(t,s) = \langle k - \zeta \ell \rangle^8  (\widetilde{\mathscr L^n})_{R, k - \zeta \ell, k'}^{m, \zeta}(t,s)$,
\begin{align}   
\mathbb E\|F\|_{L_{t,s}^p \ell^p_{k - \zeta \ell, k'}}^p, \mathbb E\| \partial_{(t,s)}F\|_{L_{t,s}^p \ell^p_{k - \zeta \ell, k'}}^p &\lesssim p^{mp}(L^\theta \alpha T^{4/5})^{mp/2} L^{35p + C}R^{C - p/2}, \label{eq-F}.
\end{align}
By Gagliardo-Nirenberg, we may conclude 
\begin{align}
\mathbb E\| \langle k - \zeta \ell \rangle^8  (\widetilde{\mathscr L^n})_{R, k - \zeta \ell, k'}^{m, \zeta}(t,s)\|_{L_{t,s}^\infty \ell^\infty_{k - \zeta \ell, k'}}^p&\lesssim p^{mp}(L^\theta \alpha T^{4/5})^{mp/2} L^{35p + C}R^{C - p/2}.
\end{align}
Choosing $p \geq 40C$, we have proven (\ref{eq-LR-bound}). If $R = L^{100}$, (\ref{eq-F}) no longer has the factor $R^{-p/2}$ and may have an additional factor of $L^C$, which is sufficient.  
\end{proof}

\section{Proof of main Theorem} \label{sec-thm}
We begin with a final proposition bounding the remainder $b$ and then prove Theorems \ref{mainthm} and \ref{mmtthm}.

\begin{proposition} \label{prop-b}
With probability $\geq 1 - L^{-A}$, the mapping defined by the right hand side of (\ref{eq-binv}) is a contraction mapping from the set $\{b: \|b\|_Z \leq L^{-500}\}$ to itself. 
\end{proposition}
\begin{proof}
Throughout, we exclude the exceptional set of probability $\leq L^{-A}$ from Proposition \ref{prop-remainder}. Additionally, note that on another set of probability $\geq 1 - L^{-A}$, 
\begin{align}
\|\mathscr R(t)\|_{Z}&\lesssim (L^\theta \alpha T^{4/5})^N L^5 \label{eq-Rest} \\ 
\|(\mathcal J_n)(t)\|_Z &\lesssim(L^\theta \alpha T^{4/5})^{n/2} L^5, \label{eq-Jest}
\end{align}
due to Proposition \ref{prop-couples} and using the techniques in the proof of Proposition \ref{prop-remainder} to remove the expectation. Also both $\mathscr R$ and $\mathcal J_n$ are represented by trees preserved under congruence. From now on, exclude these exceptional sets and suppose $\|b\|_Z \leq L^{-500}$. Note that 
\begin{equation}
\|\mathscr R\|_Z + \|\mathscr Q(b,b)\|_Z + \|\mathscr C(b,b,b)\|_Z \lesssim L^{-600}. 
\end{equation}
To see this, for $\mathscr Q$ and $\mathscr C$, we use that
\begin{align*}
\| \mathcal I \mathcal C(u,v,w)\|_Z &\lesssim L^{20}\|u\|_Z\|v\|_Z \|w\|_Z,
\end{align*}
combined with the fact that at least two of $u,v,w$ are $b$ and the other satisfies (\ref{eq-Jest}), while the bound for $\mathscr R$ follows from \ref{eq-Rest} and $N = \frac{10^6}{\epsilon}$. Finally, we note that 
\begin{equation}
(1 - \mathscr L)^{-1} = (1 - \mathscr L^N)^{-1}( 1 + \mathscr L + \ldots \mathscr L^{N-1}). 
\end{equation}
maps $Z$ to $Z$, where $1 - \mathscr L^N$ can be inverted via Neumann series since we have chosen $N$ large enough for $\|\mathscr L^N\|_{Z \to Z} < 1$. Therefore,
\begin{equation}
\|(1 - \mathscr L)^{-1}\|_{Z \to Z} \lesssim L^{30}.
\end{equation}
\end{proof}

\begin{proof}[Proof of Theorem \ref{mainthm}]
Note that almost surely, the initial data lies in $L^2$ and in this case, the solution is globally well-posed and satisfies $|\widehat{u}(k,t)| \lesssim L$ by conservation of mass. Denote the complement of all exceptional sets coming from Proposition \ref{prop-remainder} and \ref{prop-b} by $E$, so that $\mathbb P(E) \geq 1 - L^{-A}$. 
Note also that $\E |\widehat{u}(k,t)|^2 = \E|a_k(s)|^2$ where $s = \frac{t}{T}$. It is enough to consider $\E\left( |a_k(s)|^2 \boldsymbol{1}_{E}\right)$ as $\E\left( |a_k(s)|^2 \boldsymbol{1}_{E^c}\right) \lesssim L^{-A + 10}$. Therefore, 
\begin{align*}
\E\left(|a_k(s)|^2\boldsymbol{1}_{E}\right) = \sum_{0 \leq n_1, n_2 \leq N} &\E \left( (\mathcal J_{n_1})_k(s)  \overline{\mathcal (J_{n_2})_k(s)} \boldsymbol{1}_{E}\right) \\
&+ 2 \mathrm{Re}\sum_{0 \leq n \leq N} \E \left( (\mathcal J_n)_k(s) \overline{b_k(s)}\boldsymbol{1}_{E}\right) + \E(|b_k(s)|^2 \boldsymbol{1}_{E}).
\end{align*}
The terms involving $b$ are easily bounded by $L^{-100}$ using (\ref{eq-Jest}) as well as Proposition \ref{prop-b}. So, we turn our attention to the first term. Note that 
\begin{equation}
\left| \E \left( (\mathcal J_{n_1})_k(s)  \overline{\mathcal (J_{n_2})_k(s)} \boldsymbol{1}_{E^c}\right) \right|\leq \left(\E |(\mathcal J_{n_1})_k(s)|^4\right)^{1/4} \left(\E |(\mathcal J_{n_2})_k(s)|^4\right)^{1/4} \left( \mathbb P(E^c)^{1/2}\right) \lesssim L^{-A/2 + 10}
\end{equation}
by (\ref{eq-Jest}) and Lemma \ref{lem-hypercontractivity}, so we may replace $\boldsymbol{1}_E$ by $1$ when we consider the correlation between $\mathcal J_{n_1}$ and $\mathcal J_{n_2}$. This becomes precisely 
\begin{equation}
\E \left( (\mathcal J_{n_1})_k(s)  \overline{\mathcal (J_{n_2})_k(s)}\right) = \sum_\mathcal {Q}\mathcal{K_Q}(t,t, k),  
\end{equation}
for $\mathcal Q = \left( \mathcal T^+, \mathcal T^-\right)$, where $n(\mathcal T^\pm) \leq 2N$. Note that as congruence preserves the order of each tree in a couple, Proposition \ref{prop-couples} bounds this by $o_{\ell_k^\infty}\left(\frac{t}{T_{\mathrm{kin}}}\right)$ when $n_1 + n_2 \geq 3$. When $n_1 = n_2 = 0$, we are left with $n_{\mathrm{in}}(k)$, and when $n_1 + n_2 = 1$, 
\begin{equation*}
2\mathrm{Re} \left(\E \left( (\mathcal J_0)_k(s) \overline{(\mathcal J_1)_k(s)}\right)\right) = 0,
\end{equation*}
as the correlation can easily be calculated to be purely imaginary.

It remains to show that for $n_1 + n_2 = 2$, the correlation is $o_{\ell_k^\infty}\left(\frac{t}{T_{\mathrm{kin}}}\right)$. We must consider $K_2 := \E |(\mathcal J_1)_k(s)|^2 + 2\mathrm{Re}\left(\E \left( (\mathcal J_0)_k(s) \overline{(J_2)_k(s)} \right)\right)$. As in the proof of Theorem 1.3 of \cite{2019}, we have 
\begin{align}\label{eq-K2}
K_2 &= \frac{\alpha^2 s^2 T^2}{L^2} \sum_{\substack{k_1 - k_2 + k_3 \in \Z_L \\ \Omega(\vec k) \neq 0}} \phi_k \phi_{k_1} \phi_{k_2} \phi_{k_3} \left[\frac{1}{\phi_k} - \frac{1}{\phi_{k_1}} + \frac{1}{\phi_{k_2}} - \frac{1}{\phi_{k_3}} \right] \left| \frac{\sin(\pi Ts \Omega(\vec{k}))}{\pi Ts \Omega(\vec{k})}\right|^2 + O\left( \frac{Ts}{T_{\mathrm{kin}}}L^{-\delta}\right), 
\end{align}
where $\phi_{k_i} = n_{\mathrm{in}}(k_i)$. In order to justify this, following \cite{2019}, we need an improved bound for 3-vector counting when $\Omega = 0$ and $\{k_1, k_3\} \neq \{k, k_2\}$. In this case, there are no choices of $(k_1, k_2, k_3)$ that satisfy this. Using Corollary \ref{cor-iterates}, we have that 
\begin{equation*}
K_2 = \alpha^2 T^2 s^2 \sum_{\substack{(f_1, f_3) \in \Z^2 \\ |f_i| \lesssim TsL^{-1 + 3\delta}}} \int_{u_1-u_2+u_3 = k} W(\vec{u}) \left| \frac{\sin(\pi Ts \Omega(\vec{u}))}{\pi Ts \Omega(\vec{u})}\right|^2 e(-L(u_1 f_1 + u_3f_3)) \diff u + O(\frac{Ts}{T_{\mathrm{kin}}}L^{-\delta}),
\end{equation*}
where $W(\vec{u})$ represents all $\phi$-terms and $\vec{u} = (u_1, u_2, u_3, k)$. Now, we use the fact that for smooth functions $f$,
\begin{align}\label{eq-conv}
\left|t\int\left| \frac{\sin(\pi t x)}{\pi t x}\right|^2 f(x) \diff x - f(0)\right| \lesssim C(f) t^{-\frac{1}{2}},
\end{align}
where $C(f)$ depends on the $\ell^\infty$ norms of $f$ and $f'$. Therefore, the same is true if we replace $f$ with $|f|$ (we may approximate $|f|$ by smooth functions $f_\delta = \left(f^2 + \delta^2\right)^\frac{1}{2}$). So, we may conclude that 
\begin{align*}
K_2 &\lesssim \alpha^2 T^2 s^2 \sum_{\substack{(f_1, f_3) \in \Z^2 \\ |f_i| \lesssim TsL^{-1 + 3\delta}}} \int_{u_1-u_2+u_3 = k} |W(\vec{u})| \left|\frac{\sin(\pi Ts \Omega(\vec{u}))}{\pi Ts \Omega(\vec{u})}\right|^2 \diff u + O \left(\frac{t}{T_{\mathrm{kin}}} L^{-\delta}\right) \\
&\lesssim \alpha^2 t \left[(TsL^{-1 + 3\delta})^2 (Ts)^{-\frac{1}{2}}\right] + O(\frac{t}{T_{\mathrm{kin}}}L^{-\delta}) \\
&\lesssim O(\frac{t}{T_\mathrm{kin}}L^{-\delta}), 
\end{align*}
by our choice of $T$. 
\end{proof}

\begin{proof}[Proof of Theorem \ref{mmtthm}] If $1 < \sigma < 2$, then the proof is identical to above other than the fact we are taking the expectation over $E_1 \supset E$, where $E_1$ is the event that (\ref{DISP}) has a smooth solution on $[0,T]$. However, if $0 < \sigma < 1$, we will have to additionally consider the correlations between $\mathcal{J}_{n_1}$ and $\mathcal J_{n_2}$, where $n_1 + n_2 = 2$. In this case, we consider $K_2$ and may rewrite it as (\ref{eq-K2}). This requires an improved 3-vector counting of $L^2T^{-1}L^{-\delta}$ when we impose $\Omega = 0$, which we may achieve bounding $\Omega$ to within $L^{-\frac{1}{2 - \sigma}}$ instead of $T^{-1}$ in Lemma \ref{lem-3vc}. Then, we may apply Corollary \ref{cor-iterates}, using that $T^{2 - \sigma}L^{-1 + 3\delta} \ll 1$, to see that 
\begin{align*}
K_2 &= \alpha^2 T^2 s^2\int_{\xi_i \in \R} \phi_\xi\phi_{\xi_1} \phi_{\xi_2} \phi_{\xi_3} \left[\frac{1}{\phi_\xi} - \frac{1}{\phi_{\xi_1}} + \frac{1}{\phi_{\xi_2}} - \frac{1}{\phi_{\xi_3}} \right] \left| \frac{\sin(\pi Ts \Omega(\vec{\xi}))}{\pi Ts \Omega(\vec{\xi})}\right|^2 + O\left( \frac{t}{T_{\mathrm{kin}}}L^{-\delta}\right).  
\end{align*}
Applying (\ref{eq-conv}), we obtain the first iterate of (\ref{WKE}), and may conclude the proof.  
\end{proof}

\section*{Declarations}
\noindent \textit{Acknowledgements} The author would like to thank her advisor Zaher Hani for all of his invaluable help and feedback. 

\vspace{.1cm}

\noindent \textit{Funding.} This research was supported by the Simons Collaboration Grant on Wave Turbulence and NSF grant DMS-1936640. 

\vspace{.1cm}

\noindent \textit{Data Availability Statement.} This manuscript has no associated data.

\section*{References}
\printbibliography[heading = none]

\end{document}